\documentclass[reqno]{amsart}

\usepackage{graphicx}
\usepackage{amssymb}
\usepackage{xcolor}
\usepackage{enumitem}
\usepackage{etoolbox}
\usepackage{natbib}
\usepackage{mleftright}
\usepackage{mathtools}
\usepackage{hyperref}
\usepackage{lipsum}

\newtheorem{theorem}{Theorem}[section]

\newtheorem{lemma}[theorem]{Lemma}
\newtheorem{proposition}[theorem]{Proposition}
\newtheorem{corollary}[theorem]{Corollary}
\newtheorem{remark}[theorem]{Remark}

\DeclareMathOperator*{\supp}{supp}

\numberwithin{equation}{section}
\hypersetup{colorlinks=false}
\mathtoolsset{showonlyrefs} 
\allowdisplaybreaks[2]
\mleftright

\begin{document}
\title[Global well-posedness and decays for viscous $\beta$-plane equations]{Global well-posedness and temporal decay estimates for the viscous $\beta$-plane equations} 
\author[T. Yoshizawa]{Tomoaki Yoshizawa}
\address{Graduate School of Mathematical Sciences, The University of Tokyo, 3-8-1 Komaba, Meguro-ku, Tokyo, 153-8914, Japan.}
\email{tomoaki@ms.u-tokyo.ac.jp}
\maketitle
\begin{abstract}
We consider the Cauchy problem of the viscous $\beta$-plane equations.
We first establish the global well-posedness of the system for the initial data sufficiently small compared to the Rossby parameter.
The smoothing effect of the flow is then proved, and it is shown that the obtained global solution satisfies the equation in the classical sense.
We also reveal that under some additional assumptions, the solution decays as fast as the corresponding linear solution and asymptotically behaves like a constant multiple of the integral kernel of the linearized equation. In particular, the decay rate of the solution is faster than expected from the flow of the heat equation.
\end{abstract} 
\tableofcontents


\section{Introduction} \label{sect:1}
In this article, we consider the Cauchy problem of the viscous $\beta$-plane equations given by
	\begin{align}
		\begin{dcases*}
			\partial_t\omega+(u\cdot \nabla)\omega =\Delta \omega+\beta L_1\omega & in $(0, \infty)\times\mathbb R^2$, \\
			u=\nabla^\perp(-\Delta)^{-1}\omega 	&in $(0, \infty)\times\mathbb R^2$, \\
			\omega(0, \cdot)=\omega_0 & in $\mathbb R^2, $
		\end{dcases*}\label{eq:1.8}
	\end{align}
where
	$\omega=\omega(t, x)\colon(0,\infty)\times\mathbb R^2\to\mathbb R$
 is the unknown scalar vorticity and the operator
	$L_1$
is defined by
	$L_1\coloneqq \partial_1(-\Delta)^{-1}.$
The system describes the dynamics of the two-dimensional fluids under the influence of the Coriolis force whose magnitude varies with respect to latitude. The Rossby parameter
	$\beta\in\mathbb R$
corresponds to the rate of change of the Coriolis parameter.
The aim of this paper is to study the global well-posedness of the system \eqref{eq:1.8} and related properties of the equation, such as smoothing effect, and temporal decay estimates of the solutions. For the research on the inviscid model, we refer to \cite{EW17} and \cite{PW18}.

Let us first consider the global well-posedness of \eqref{eq:1.8}. By the Duhamel formula, 
the system can be rewritten to the integral equation which is formally equivalent to the original problem \eqref{eq:1.8}: 
	\begin{align}
		\omega(t)
		=T_\beta(t)\omega_0-\int_0^tT_\beta(t-\tau)\operatorname{div}(u(\tau)\omega(\tau))\, d\tau.\label{eq:1.9}
	\end{align}
Here, the operator 
	$T_\beta(t)$
denotes the linear propagator defined by
	$T_\beta(t)=e^{t(\Delta+\beta L_1)}.$
To be precise, we show the unique existence of the global solution of the integral equation \eqref{eq:1.9}.
Before stating our first main theorem on this matter, let us briefly recall known results on the systems related to \eqref{eq:1.8}.
If we set
	$\beta=0$, 
the system \eqref{eq:1.8} coincides with the usual two-dimensional Navier--Stokes equations in vorticity formulation.
For this system, Giga, Miyakawa, and Osada~\cite{GMO88} show the existence of the global solution when the initial vorticity
	$\omega_0$
lies in the space of the finite Radon measures in
	$\mathbb R^2.$
The uniqueness of the solution for general initial data in that space is later complemented by Gallagher and Gallay~\cite{GG05}. See also \cite{BA94}, \cite{Bre94} and \cite{Kat94} for the related discussions on the global well-posedness of the system.
	
It is also of note that the linear part on the right-hand side of the first equation in \eqref{eq:1.8} involves both of the operators with dissipative and dispersive nature, namely, 
	$\Delta$
and
	$L_1.$
For such systems, the linearized solution sometimes exhibits faster decays compared to those expected from the heat equation, and the resulting estimates may benefit us when attempting to establish the global well-posedness of the nonlinear problem. Examples of such results are those of the incompressible Navier--Stokes equations with the Coriolis force in
	$\mathbb R^3$
given by
	\begin{align}
		\begin{dcases*}
			\partial_tu+(u\cdot \nabla)u +\nabla p-\Delta u+\Omega e_3\times u=0 & in $(0, \infty)\times\mathbb R^3$, \\
			\operatorname{div} u=0 &in $(0, \infty)\times\mathbb R^3$, \\
			u(0, \cdot)=u_0 & in $\mathbb R^3,$
		\end{dcases*}\label{eq:1.8'}
	\end{align}
where 
	$u=u(t, x)\colon (0, \infty)\times\mathbb R^3\to\mathbb R^3$
and
	$p=p(t, x)\colon (0, \infty)\times\mathbb R^3\to\mathbb R$
are unknown functions, 
	$e_3\coloneqq (0, 0, 1)\in\mathbb R^3$, 
and
	$\Omega\in\mathbb R.$
In \cite{CDGG02} and \cite{CDGG06}, Chemin, Desjardins, Gallagher, and Grenier show that for any initial velocity
	$u_0\in L^2(\mathbb R^3)^3+ H^\frac12(\mathbb R^3)^3$
there is a constant
	$C>0$
such that if
	$|\Omega|\geq C, $
the system \eqref{eq:1.8'} admits the unique global solution.
Subsequently, Iwabuchi and Takada~\cite{IT13} prove the unique existence of the global solution of \eqref{eq:1.8'} in
	$\dot H^{s}(\mathbb R^3)$
for
	$\frac12<s<\frac34, $
if the initial data
	$u_0$
satisfies
	$\left\|u_0\right\|_{\dot H^s}\leq C|\Omega|^{\frac s2-\frac14} $
for a constant
	$C>0$
independent of
	$u_0$
and
	$\Omega.$
The range of 
	$s$
is later extended to
	$\frac12<s<\frac9{10}$
by Koh, Lee, and Takada~\cite{KLT14}.
Later on, Ahn, Kim, and Lee~\cite{AKL22} generalize the result and show the global well-posedness in the case where the Laplacian is replaced by the fractional one
	$(-\Delta)^\alpha $
for
	$\frac12<\alpha<\frac52.$
We also refer to \cite{AKL21} and \cite{Kim22} for the study of the incompressible magnetohydrodynamics (MHD) equations in the rotational frame of reference.

Now we give the precise statement of our first main theorem. We establish the global well-posedness of the system~\eqref{eq:1.8} in a framework similar to those of~\cite{IT13} and its subsequent works.
\begin{theorem}\label{thm:1.1}
	Let
		$\beta\neq0, 2<p_1, p_2, r_1<\infty, 2\leq r_2<\infty$, 
	and
		$0\leq\delta\leq\frac15$
	satisfy
		\begin{align}
			&\delta <\min\left\{\frac2{p_1}, \frac2{p_2}\right\}, 
			\quad\max\left\{1-\frac1{p_1}, 1-\frac1{p_2}\right\}\leq \frac1{p_1}+\frac{1}{p_2}-\frac\delta2\label{eq:1.2}, \\
			&\max\left\{\frac12-\frac{1}{p_1}+\frac\delta2-\frac32\left(1-\frac2{p_2}\right), \frac{1}{2}\left(1-\frac2{p_1}\right)\right\}
			\leq\frac{1}{r_1}
			\leq \min\left\{\frac12-\frac{1}{p_1}+\frac\delta2, \frac54\left(1-\frac2{p_1}\right)\right\}, \notag\\
			&1-\frac1{p_2}+\frac\delta2-\frac32\left(1-\frac2{p_1}\right)
			\leq\frac1{r_2}
			\leq 1-\frac{1}{p_2}+\frac\delta2\label{eq:1.4}
		\end{align}
	and either
		\begin{align}
			r_2>2\quad\text{and}\quad\frac{1}{2}\left(1-\frac2{p_2}\right)\leq\frac{1}{r_2}\leq\frac{5}{4}\left(1-\frac2{p_2}\right)\label{eq:1.5}
		\end{align}
	or
		\begin{align}
			r_2=2\quad\text{and}\quad\frac{1}{2}\left(1-\frac2{p_2}\right)<\frac{1}{2}<\frac{5}{4}\left(1-\frac2{p_2}\right).\label{eq:1.6}
		\end{align}
	Then there is a constant
		$C_0>0$
	such that for any
		$\omega_0\in \dot H^{-1+\delta}\cap\dot H^\delta$
	with
		\begin{align}
			\left|\beta\right|^{-\frac\delta3}\left\|\omega_0\right\|_{\dot H^{-1+\delta}}+\left|\beta\right|^{-\frac{1+\delta}3}\left\|\omega_0\right\|_{\dot H^{\delta}}
			\leq C_0, \label{eq:1.7}
		\end{align}
	there exists a unique solution 
		$\omega\in C([0, \infty); \dot H^{-1+\delta})\cap L^{r_1}([0, \infty); \dot W^{-1+\delta, p_1})\cap L^{r_2}([0, \infty); \dot W^{\delta, p_2})$
	of the equation \eqref{eq:1.9}.
\end{theorem}
\begin{remark}\label{rem:1.2'}
\rm{Although the system~\eqref{eq:1.8} describes the motions of two-dimensional fluids, 
we are forced to impose a smallness condition for the initial data
	$\omega_0$
to obtain the global solution, since this is necessary to exploit the dispersive nature of the operator
	$L_1$
in our framework.
This stands in strong contrast to the previous works on the two-dimensional viscous fluids such as~\cite{GMO88}, where the smallness condition is unnecessary to show the global well-posedness.}
\end{remark}

\begin{remark}\label{rem:1.2''}\rm{
In the case of the Navier--Stokes equations with the Coriolis force in
	$\mathbb R^3, $
the divergence-free condition on the velocity field
	$u$
allows us to rewrite the nonlinear term to
	$(u\cdot \nabla)u
	=\operatorname{div}(u\otimes u), $
where the nonlinearity appears as a product of two terms without derivatives, namely, 
	$u\otimes u.$
By contrast, in our setting the nonlinear term is given by
	$(u\cdot \nabla)\omega
	=\operatorname{div}(u\omega), $
and therefore we have to deal with the product of two unknown functions where one factor has at least one more derivative than the other. Here, the one-derivative difference between the two factors is critical in that the failure of the embedding
	$\dot H^{1}(\mathbb R^2)\not\hookrightarrow L^\infty(\mathbb R^2)$
causes several difficulties in handling the nonlinear estimate of the equation. This structure of the nonlinearity is the main reason why we need some technicalities which do not appear in the analysis of \cite{IT13}, \cite{KLT14} and \cite{AKL22}, such as the additional regularity of the initial data
	$\omega_0\in \dot H^\delta$
and the use of the two auxiliary space-time norms, i.e., those of 
	$L^{r_1}([0, \infty); \dot W^{-1+\delta, p_1})$
and
	$L^{r_2}([0, \infty); \dot W^{\delta, p_2}).$
}\end{remark}
\begin{remark}\label{rem:1.2}\rm
	The set of all
		$\mathcal A\coloneqq (\delta, p_1, r_1, p_2, r_2)$
	which satisfies the conditions in Theorem~\ref{thm:1.1} is not empty if
		$0\leq\delta\leq\frac15.$
	In fact, for such choice of
		$\delta, $
	it is possible to define an exponent
		$p$
	by the formula
		$\frac1p=\frac13+\frac\delta6$
	and set
		$p_1=p_2=p.$
	In that case, the assumptions on
		$(\delta, p_1, p_2)$
	in Theorem~\ref{thm:1.1} are always satisfied, and the conditions on
		$(r_1, r_2)$
	with
		$2<r_2<\infty$
	are equivalent to
		\begin{align*}
			\frac{1-\delta}6
			\leq\frac1{r_1}
			\leq \frac{1+2\delta}6
			\quad\text{and}\quad
			\frac{1+5\delta}6
			\leq\frac1{r_2}
			\leq\frac5{12}(1-\delta).
		\end{align*}
	The set of
		$r_1$
	and
		$r_2$
	satisfying this condition is not empty, and hence nor is the set of
		$\mathcal A$
	with the conditions in Theorem~\ref{thm:1.1}.
\end{remark}
\begin{remark}\label{rem:1.3}\rm
	Let
		$\lambda >0$.
	If a pair
		$(\omega, \beta)$
	gives a solution to \eqref{eq:1.8}, so does
		$(\omega_\lambda, \beta_\lambda)$
	defined by
		\begin{align}
			\omega_\lambda(t, x)\coloneqq\lambda^2\omega(\lambda^2 t, \lambda x)
			\quad\text{and}\quad
		\beta_\lambda\coloneqq \lambda^3\beta.\label{eq:1.10}
		\end{align}
	The condition \eqref{eq:1.7} is invariant under this scaling transformation, and several quantities such as 
	$|\beta|^{-\frac{1+s}{3}}\left\|\omega_0\right\|_{\dot H^s}, s\in\mathbb R, $
which later appear in this article, also satisfy the same scaling invariance condition.
\end{remark}

We next go on to the discussion on the smoothing effect of the flow by the equation~\eqref{eq:1.8}. As for the related systems, Ahn--Kim--Lee~\cite{AKL22} point out that in the case of the incompressible fractional Navier--Stokes equations with the Coriolis force, the obtained global solution is smooth away from the time
	$t=0.$
Kim~\cite{Kim22} observes similar regularizing effect for the MHD equations with the Coriolis force. Both of these results are established via the energy estimates under a Serrin-type condition on the solutions.

In our second main theorem, we show that with a generic choice of the exponents
	$\mathcal A, $
the solution of the integral equation \eqref{eq:1.9} is also instantaneously regularized and turns out to be a classical solution of \eqref{eq:1.8}. We should note that in our framework, this is achieved not through the energy estimate of the differential equation~\eqref{eq:1.8}, but by the direct estimates of the integral equation~\eqref{eq:1.9}.
\begin{theorem}\label{thm:1.2}
	Let
		$\beta\neq 0$
	be given
	and
		$\mathcal A = (\delta, p_1, r_1, p_2, r_2)$
	satisfy the assumption of Theorem~\ref{thm:1.1}. Assume also that
		\begin{align}
			\frac{1}{r_1}<\frac12-\frac{1}{p_1}+\frac\delta2
			\quad\text{and}\quad
			\frac{1}{r_2}<1-\frac{1}{p_2}+\frac\delta2.\label{eq:1.11}
		\end{align}
	Then for any
		$m\geq 0$
	the solution 
		$\omega$
	obtained in Theorem~\ref{thm:1.1} satisfies
		$\omega\in C((0, \infty); \dot H^{-1+\delta+m})\cap C^{1}((0, \infty); \dot H^{\delta+m})$
	and 
	gives a classical solution of 
	 \eqref{eq:1.8} 
	for
		$t>0$. 
	Moreover, if the initial data satisfy
		$\omega_0\in \dot H^{-1+\delta}\cap \dot H^{\delta+m}$
	for some
		$m\geq0$, 
	then the solution
		$\omega$
	satisfies
		$\omega\in C([0, \infty); \dot H^{-1+\delta+m})$.
\end{theorem}

Finally, we show the temporal decay estimates and asymptotics of the solution.
In the case of
	$\beta=0, $
temporal decay estimates of the solution
	$\omega$
	\begin{align}
		\left\|\partial_t^\ell\partial_x^\alpha\omega(t)\right\|_{L^p}
		\leq Ct^{-\ell-\frac{|\alpha|}2-1+\frac1p}\label{eq:1.7'}
	\end{align}
is obtained by \cite{Kat94} and \cite{GG99} (see also \cite{GGS10}) for the integrable initial data
	$\omega_0\in L^1, $
 where
	$\ell\in\mathbb N\cup\{0\}, $
	$\alpha\in(\mathbb N\cup\{0\})^2, $
and
	$p\in[1, \infty]$.
The estimates here imply that the solutions of the two-dimensional vorticity equations decay as fast as the solutions of the heat equation. In fact, it is even true that the obtained solution asymptotically behaves like a constant multiple of the heat kernel. Indeed, in \cite{GK88}, \cite{Car94}, \cite{GW05} and \cite{Mae08}, it is shown that the solution
	$\omega$
satisfies the asymptotic estimate
	\begin{align*}
		\lim_{t\to\infty}t^{\ell+\frac{|\alpha|}{2}+1-\frac1p}\left\|\partial_t^\ell\partial_x^\alpha\omega(t)
		-\int_{\mathbb R^2}\omega_0(y)\, dy\cdot\partial_t^\ell\partial_x^\alpha G_t\right\|_{L^p}=0
	\end{align*}
for
	$\ell\in\mathbb N\cup\{0\}$, 
	$\alpha\in(\mathbb N\cup\{0\})^2$, 
	$p\in[1, \infty]$.
Here, 
	$G_t$
denotes the Gauss kernel defined by
	$G_t(x)\coloneqq (4\pi t)^{-1}\exp(-\frac{|x|^2}{4t}).$
Our next aim is to establish similar estimates on decays and asymptotics for the solutions of the system \eqref{eq:1.8}. As in Proposition~\ref{cor:2.3}, we see that the linearized solutions of \eqref{eq:1.8} in fact decay faster than those in the case of
	$\beta=0, $
due to the dispersive effect of the operator
	$\beta L_1.$
Therefore, we may expect the solution of \eqref{eq:1.8} to decay faster than \eqref{eq:1.7'} under suitable assumptions on the initial data
	$\omega_0.$
Now our attempt is to rigorously justify this observation.

For the incompressible Navier--Stokes equations with the Coriolis force, temporal decays and asymptotics of the solutions have also been studied. For example, Ahn, Kim, and Lee~\cite{AKL22} derive the estimates for the solution
	$u$
of the type
	\begin{align*}
		\left\|u(t)\right\|_{L^p}\leq C\left\|u_0\right\|_{L^q}t^{-\frac3{2}(\frac1q-\frac1p)}(1+|\Omega|t)^{-1+\frac2p}
	\end{align*}
for exponents
	$p$
and
	$q$
satisfying
	$1<q\leq \frac p{p-1}\leq 2\leq p<3$
and several conditions relavant to the global well-posedness.
Egashira and Takada~\cite{ET23} also consider the decay estimates and show that under suitable assumptions on the initial data
	$u_0\in L^1(\mathbb R^3)\cap \dot H^s(\mathbb R^3)$
for some
	$\frac12<s<\frac9{10}, $
the solution exhibits the faster temporal decay
	\begin{align*}
		\left\|u(t)\right\|_{L^p}\leq Ct^{-\frac3{2}(1-\frac1p)}(1+|\Omega|t)^{-1+\frac2p}
	\end{align*}
for exponents
	$p$
included in an interval 
	$[2, p_*], $
where the upper bound
	$p_*$
is less than
	$3.$
They also derive asymptotic expansions of the solution
	$u(t)$
in 
	$L^p(\mathbb R^3), $
if the initial velocity fulfills the additional assumption
	$(1+|x|)u_0\in L^1(\mathbb R^3).$
The decay for the energy critical norm
	$\left\|u(t)\right\|_{\dot H^\frac12}$
is also dealt with in~\cite{IKNP24}.
As for the MHD equations with the Coriolis force, Kim~\cite{Kim22} proves that if the initial data
	$(u_0, B_0)$
is sufficiently small compared to the speed of rotation
	$|\Omega|$
in
	$L^1(\mathbb R^3)\cap H^{\frac12+\gamma}(\mathbb R^3)$
for some
	$\gamma\in (0, 1), $
the velocity field
	$u$
of the solution satisfies the faster decay estimate
	\begin{align*}
		\left\|u(t)\right\|_{L^p}\leq Ct^{-\frac32(1-\frac1p)-1+\frac2p}|\Omega|^{-1+\frac2p}
	\end{align*}
for exponents
	$p$
included in a certain subinterval of $[2, 3)$.

In Theorem~\ref{thm:1.3} stated below, we show the corresponding temporal decay estimates of the solution
 of \eqref{eq:1.8} under some suitable conditions for
	$\omega_0\in L^1$
and the related exponents
	$\mathcal A$.
More precisely, we prove that the solution
	$\omega$
obtained in Theorem~\ref{thm:1.1} decays as fast as the linear solution 
	$T_\beta(t)\omega_0$
and behaves like a constant multiple of the integral kernel of the linearized equation as
	$t\to\infty.$
New feature of our result is that temporal decays shown here are valid for all 
	$L^p$-type
norms with
	$2\leq p\leq\infty, $
which is not the case for the previous results such as \cite{AKL22},	\cite{ET23}, and \cite{Kim22}.
In the argument concerning the proof of 
Theorem~\ref{thm:1.3}, for simplicity we assume that
	$p_1$
and
	$p_2$
are defined by
	$\frac{1}{p_1}=\frac{1}{p_2}=\frac1p\coloneqq\frac13+\frac\delta6.$
Note that the set of all
	$\mathcal A=(\delta, p, r_1, p, r_2)$
 satisfying the assumption of Theorem~\ref{thm:1.1} and the condition \eqref{eq:1.11} is not empty if
	$0\leq \delta\leq \frac15$: 
See Remark~\ref{rem:1.2}.
\begin{theorem}\label{thm:1.3}
	Assume that
		$0\leq \delta<\frac1{13}$
	and
		$\mathcal A=(\delta, p, r_1, p, r_2)$
	satisfies the assumption of Theorem~\ref{thm:1.1} and the condition \eqref{eq:1.11}. Then there is a constant
		$C_{\mathcal A}>0$
	such that if
		$\omega_0\in \dot H^{-1+\delta}\cap \dot H^{3-\frac2p}\cap  L^1\cap \dot W^{-1+\delta, \frac{p}{p-1}}$
	satisfies
		\begin{align*}
			\left|\beta\right|^{-\frac{\delta}{3}}\left\|\omega_0\right\|_{\dot H^{-1+\delta}}
			+\left|\beta\right|^{-\frac{1+\delta}{3}}\left\|\omega_0\right\|_{\dot H^{\delta}}
			\leq C_{\mathcal A}, 
		\end{align*}
	then for any
		$s\geq 0$
	and
		$a\in[2, \infty]$
	there exists a positive constant
		$C$
	depending only on
		$s, $
		$\mathcal A, $
		$\left\|\omega_0\right\|_{L^1}, $
		$|\beta|^{-\frac\delta3}\left\|\omega_0\right\|_{\dot H^{-1+\delta}}$, 
		$|\beta|^{-\frac{1}{3}(\frac{15}4-\frac2p)}\left\|\omega_0\right\|_{\dot H^{\frac{11}4-\frac2p}}$
	and 
		$|\beta|^{-\frac13(1+\delta-\frac2{p'})}\left\|\omega_0\right\|_{\dot W^{-1+\delta, p'}}$, 
	satisfying
		\begin{align}
			\left\|\omega(t)\right\|_{\dot W^{s, a}}
			\leq Ct^{-\frac s2-1+\frac1a}\min\left\{1, |\beta|^{-1+\frac2a}t^{-\frac32(1-\frac2a)}\right\}\label{eq:1.12}
		\end{align}
	for all
		$t>0.$
	Moreover, if the condition
		$\frac12-\frac1p+\frac\delta2-\frac32(1-\frac2p)<\frac1{r_1}$
	holds, then it follows that
		\begin{align}
			\lim_{t\to\infty}\left(t^{-\frac s2-1+\frac1a}\min\left\{1, |\beta|^{-1+\frac2a}t^{-\frac32(1-\frac2a)}\right\}\right)^{-1}\left\|\omega(t)-K_{\beta, t}\int_{\mathbb R^2}\omega_0(y)\, dy\right\|_{\dot W^{s, a}}=0.\notag
		\end{align}
	Here, the function 
		$K_{\beta, t}\colon \mathbb R^2\to \mathbb R$
	is the integral kernel of the linearized equation defined by
		$K_{\beta, t}=e^{t\beta L_1}G_t$.
\end{theorem}
\begin{remark}\label{rem:1.7}\rm
	Theorem~\ref{thm:1.3} allows us to obtain temporal decays of the solution in the full range of the Sobolev spaces
	$\dot W^{s, a}$
for
	$a\in[2, \infty]$
and
	$s\geq0.$
Compared to the existing results on the related systems in which faster decays of the linearized solutions are expected by dispersion, our result improves the range of the exponents from two perspectives: it yields the
	$L^\infty$
decays, and arbitrarily higher orders of spatial derivatives are permitted. However, it seems difficult to extend the decay estimates to the time derivatives of
	$\omega$, 
since differentiation with respect to the time variable cause the term
	$L_1\omega$, 
whose rate of the decay is even slower than that of the original term
	$\omega.$
\end{remark}

\section{Preliminaries}\label{sect:2}
In this section, we collect some preliminaries necessary for our later arguments. We start by recalling the definition of the homogeneous Besov spaces 
	$\dot B^s_{p, r}$. 
Let $\varphi\colon \mathbb R^2\to [0, 1]$ be a smooth radial function satisfying 
	$\supp\varphi\subset\left\{3/4\leq |\xi|\leq 8/3\right\}$ and
	$\sum_{k\in\mathbb Z}\varphi(2^{-k}\xi)=1$
for all 
	$\xi\in\mathbb R^2\setminus\{0\}$. 
We define the homogeneous Littlewood--Paley projection operators 
	$P_k$ 
for 
	$k\in\mathbb Z$ 
by 
	$P_kf\coloneqq \mathcal F^{-1}(\varphi(2^{-k}\cdot)\widehat f)$, 
	$f\in\mathcal S'.$
We also define 
	$\widetilde{P}_k\coloneqq P_{k-1}+P_k+P_{k+1}.$
Note that 
	$\widetilde P_kP_k=P_k$
hold for any 
	$k\in\mathbb Z$.
For 
	$s\in\mathbb R, 1\leq p\leq \infty$
and 
	$1\leq r\leq \infty$, 
we define the homogeneous Besov norm 
	$\left\|\cdot\right\|_{\dot B^s_{p, r}}$ 
 by
	$$\left\|f\right\|_{\dot B^s_{p, r}}
	\coloneqq \left\|\left(2^{sk}\left\|P_kf\right\|_{L^p}\right)_{k\in\mathbb Z}\right\|_{\ell^r}.$$
The homogeneous Besov space
	$\dot B^s_{p, r}$
is the set of all 
	$f\in\mathcal S'$
satisfying
	\begin{align*}
		\left\|f\right\|_{\dot B^s_{p, r}}<\infty
		\quad\text{and}\quad
		\big\|\mathcal F^{-1}\big(\theta(\lambda\cdot)\widehat f\big)\big\|_{L^\infty}\to0
		\text{ as $\lambda\to\infty$ for any $\theta\in\mathcal D(\mathbb R^2).$}
	\end{align*}
For basic properties of the operators
	$P_k$
and the Besov spaces
	$\dot B^s_{p, r}$
necessary for our analysis, see ~\cite{BCD11}, for example.

We next deduce several estimates of the linear propagators.
\begin{lemma}\label{lem:2.1}
	There exists a constant 
		$C>0$ 
	such that for any 
		$f\in \mathcal S, k\in\mathbb Z, \beta\in \mathbb R\setminus\{0\}$ 
	and 
		$t\in\mathbb R\setminus\{0\}$, 
	it follows that
		\begin{align}\label{eq:2.1}
			\left\|e^{t\beta L_1}P_kf\right\|_{L^\infty}
			\leq C2^{3k}|\beta t|^{-1}\left\|P_kf\right\|_{L^1}.
		\end{align}
\end{lemma}

\proof
	 See \citep[Lemma~2.1]{PW18} (see also \citep[PROPOSITION~5.1]{EW17}).
\qed

\begin{proposition}\label{prop:2.2}
	For any 
		$s\leq s', $
		$2\leq p\leq\infty$, 
	and
		$1\leq r\leq\infty$,  
	there is a constant 
		$C>0$ 
	satisfying
		\begin{align}\label{eq:2.2}
			\left\|T_\beta(t)f\right\|_{\dot B^{s'}_{p, r}}
			\leq C t^{-\frac 12(s'-s)-1+\frac 2p}\min\left\{1, \left|\beta\right|^{-1+\frac 2p}t^{-\frac32(1-\frac2p)}\right\}\left\|f\right\|_{\dot B^s_{p', r}}
		\end{align}
	for all $t>0$ and $f\in \dot B^s_{p', r}$.
\end{proposition}
\proof
	It suffices to consider the case $s'\geq 0=s$.
	For any
		$g\in L^2$
	and
		$k\in\mathbb Z, $
	we have 
		\begin{align}\label{eq:2.3}
			\left\|T_\beta(t)\widetilde P_kg\right\|_{L^2}
			\lesssim e^{-ct2^{2k}}\left\|g\right\|_{L^2}
		\end{align}
	where
		$c>0$
	 is some universal constant .

	By Lemma~\ref{lem:2.1} we also have
		$$\left\|T_\beta(t)\widetilde P_kg\right\|_{L^\infty}
		\lesssim 2^{3k}\left|\beta t\right|^{-1}\sum_{|\ell|\leq 1}\left\|e^{t\Delta}P_{k+\ell}g\right\|_{L^1}$$
	for any
		$g\in L^1$
	and
		$k\in\mathbb Z.$
	By Bernstein's inequality, we see that
		$$\left\|T_\beta(t)\widetilde P_kg\right\|_{L^\infty}
		\lesssim 2^k\left\|T_\beta(t)\widetilde P_kg\right\|_{L^2}
		=2^k\left\|e^{t\Delta}\widetilde P_kg\right\|_{L^2}
		\lesssim 2^{2k}\left\|e^{t\Delta}\widetilde P_kg\right\|_{L^1}.$$
	Since there is a constant
		$c>0$ 
	which satisfies
		$\left\|e^{t\Delta} P_{k+\ell}g\right\|_{L^1}
		\lesssim e^{-ct2^{2k}}\left\|g\right\|_{L^1}$
	for each 
		$\ell\in \{-1, 0, 1\}$, 
	the above estimates give
		\begin{align}\label{eq:2.4}
			\left\|T_\beta(t)\widetilde P_kg\right\|_{L^\infty}
			\lesssim\min\left\{2^{2k}, 2^{3k}\left|\beta t\right|^{-1}\right\}e^{-ct2^{2k}}\left\|g\right\|_{L^1}.
		\end{align}
	By combining \eqref{eq:2.3} and \eqref{eq:2.4}, for any
		$2\leq p\leq \infty$, 
	the Riesz--Thorin interpolation theorem yields
		\begin{align}\label{eq:2.5}
			\left\|T_\beta(t)\widetilde P_kg\right\|_{L^p}
			\lesssim \min\left\{2^{2k(1-\frac 2p)}, 2^{3k(1-\frac2p)}\left|\beta t\right|^{-1+\frac 2p}\right\}e^{-c_pt2^{2k}}\left\|g\right\|_{L^{p'}}
		\end{align}
	for all
		$g\in L^{p'}$, 
	where 
		$c_p$
	is a positive constant.
	
	Taking 
		$g=P_{k}f$
	in \eqref{eq:2.5} leads to
		\begin{align}\label{eq:2.6}
			\left\|T_\beta(t)P_kf\right\|_{L^p}
			\lesssim \min\left\{2^{2k(1-\frac 2p)}, 2^{3k(1-\frac2p)}\left|\beta t\right|^{-1+\frac 2p}\right\}e^{-c_pt2^{2k}}\left\|P_kf\right\|_{L^{p'}}.
		\end{align}
	Since the estimate
		$$2^{s'k}\min\left\{2^{2k(1-\frac 2p)}, 2^{3k(1-\frac2p)}\left|\beta t\right|^{-1+\frac 2p}\right\}e^{-c_pt2^{2k}}
		\lesssim t^{-\frac{s'}{2}}\min\left\{t^{-1+\frac 2p}, t^{-\frac{3}{2}(1-\frac2p)}\left|\beta t\right|^{-1+\frac 2p}\right\}$$
	holds for all
		$t>0$
	and
		$k\in\mathbb Z$, 
	from \eqref{eq:2.6} we have
		\begin{align}\label{eq:2.7}
			2^{s'k}\left\|T_\beta(t)P_kf\right\|_{L^p}
			\lesssim t^{-\frac{s'}{2}-1+\frac 2p}\min\left\{1, \left|\beta\right|^{-1+\frac 2p}t^{-\frac{3}{2}(1-\frac2p)}\right\}\left\|P_kf\right\|_{L^{p'}}.
		\end{align}
	By summing up with respect to
		$k\in\mathbb Z$, 
	we obtain the desired estimate.
\qed

\begin{proposition}\label{cor:2.3}
	Assume that
		$s\geq 0$
	and
		$a\in[2, \infty].$
	Then there is a constant
		$C>0$
	satisfying 
		\begin{align}
			\left\|T_\beta(t)f\right\|_{\dot W^{s, a}}
			\leq Ct^{-\frac s2-1+\frac1a}\min\left\{1, \left|\beta\right|^{-1+\frac2a}t^{-\frac32(1-\frac2a)}\right\}\left\|f\right\|_{L^1}\label{eq:2.69}
		\end{align}
	for all
		$t>0$
	and
		$f\in L^1.$
\end{proposition}
\proof
	It is enough to prove the following stronger estimate
		\begin{align}
			\left\|T_\beta(t)f\right\|_{\dot B^s_{a, 1}}
			\lesssim t^{-\frac s2-1+\frac1a}\min\left\{1, \left|\beta\right|^{-1+\frac2a}t^{-\frac32(1-\frac2a)}\right\}\left\|f\right\|_{\dot B^0_{1, \infty}}.\label{eq:2.70}
		\end{align}
	By \eqref{eq:2.6}, we have
		\begin{align}
			&\quad\sum_{k\in\mathbb Z}2^{sk}\left\|T_\beta(t)P_kf\right\|_{L^a}\label{eq:2.71}\\
			&\lesssim \sum_{k\in\mathbb Z}2^{sk}\min\left\{2^{2k(1-\frac 2a)}, 2^{3k(1-\frac2a)}\left|\beta t\right|^{-1+\frac 2a}\right\}e^{-c_at2^{2k}}\left\|P_kf\right\|_{L^{a'}}\label{eq:2.72}\\
			&\lesssim \sum_{k\in\mathbb Z}2^{sk}\min\left\{2^{2k(1-\frac 1a)}, 2^{2k(1-\frac 1a)+k(1-\frac2a)}\left|\beta t\right|^{-1+\frac 2a}\right\}e^{-c_at2^{2k}}\left\|P_kf\right\|_{L^{1}}\label{eq:2.73}\\
			&\leq \min\left\{\sum_{k\in\mathbb Z}2^{sk+2k(1-\frac 1a)}e^{-c_at2^{2k}}, \sum_{k\in\mathbb Z}2^{sk+2k(1-\frac 1a)+k(1-\frac2a)}\left|\beta t\right|^{-1+\frac 2a}e^{-c_at2^{2k}}\right\}\left\|f\right\|_{\dot B^0_{1, \infty}}.\notag
		\end{align}
	Since 
		$1-\frac1a>0$
	and
		$1-\frac 2a\geq0, $
	it is possible to check that
		\begin{align}
			\sup_{t>0}\sum_{k\in\mathbb Z}t^{\frac s2+1-\frac 1a}2^{sk+2k(1-\frac 1a)}e^{-c_at2^{2k}}<\infty\label{eq:2.75}
		\end{align}
	and
		\begin{align}
			\sup_{t>0}\sum_{k\in\mathbb Z}t^{\frac s2+1-\frac1a+\frac12(1-\frac2a)}2^{sk+2k(1-\frac 1a)+k(1-\frac2a)}e^{-c_at2^{2k}}<\infty.\label{eq:2.76}
		\end{align}
	Therefore we obtain
		\begin{align}
			\left\|T_\beta(t)f\right\|_{\dot B^s_{a, 1}}
			&\lesssim \min\left\{t^{-\frac s2-1+\frac 1a}, |\beta|^{-1+\frac 2a}t^{-\frac s2-1+\frac 1a-\frac32(1-\frac 2a)}\right\}\left\|f\right\|_{\dot B^0_{1, \infty}}\label{eq:2.77}.
		\end{align}
	This completes the proof.
\qed

\begin{proposition}\label{prop:2.4'}
	Let 
		$s\geq0, $
		$a\in [2, \infty]$
	be given and assume that
		$\omega_0\in L^1.$
	Then it holds that
		\begin{align}
			\lim_{t\to\infty}\left(t^{-\frac s2-1+\frac1a}\min\left\{1, |\beta|^{-1+\frac2a}t^{-\frac32(1-\frac2a)}\right\}\right)^{-1}\left\|T_\beta(t)\omega_0-K_{\beta, t}\int_{\mathbb R^2}\omega_0(y)\, dy\right\|_{\dot W^{s, a}}
			=0.
		\end{align}
\end{proposition}
\proof
The proof can similarly be done as in \citep[Theorem~3.4]{ET23} (see also \citep[THEOREM~3.1]{FM01}).
Let
	$M_{s, a}$
denote the function on
	$(0, \infty)$
defined by
	\begin{align}
		M_{s, a}(t)\coloneqq t^{-\frac s2-1+\frac1a}\min\left\{1, |\beta|^{-1+\frac2a}t^{-\frac32(1-\frac2a)}\right\}.
	\end{align}
Since the operator
	$T_\beta(t)$
is given by
	\begin{align}
		\left(T_\beta(t)\omega_0\right)(x)
		=\int_{\mathbb R^2}K_{\beta, t}(x-y)\omega_0(y)\, dy, 
	\end{align}
we have
	\begin{align}
		\left\|T_\beta(t)\omega_0-K_{\beta, t}\int_{\mathbb R^2}\omega_0(y)\, dy\right\|_{\dot W^{s, a}}
		&=\left\|\int_{\mathbb R^2}\left( K_{\beta, t}(\cdot-y)-K_{\beta, t}(\cdot)\right)\omega_0(y)\, dy\right\|_{\dot W^{s, a}}\label{eq:2.78}\\
		&\lesssim \int_{\mathbb R^2}\left\| K_{\beta, t}(\cdot-y)-K_{\beta, t}\right\|_{\dot B^s_{a, 1}}|\omega_0(y)|\, dy.\label{eq:2.79}
	\end{align}
By the definition of
	$K_{\beta, t}$, 
a change of variables yields
	\begin{align}
		K_{\beta, t}(x)
		&=\mathcal F^{-1}(e^{-t|\xi|^2}e^{t\beta\frac{i\xi_1}{|\xi|^2}})(x)
		=t^{-1}\mathcal F^{-1}(e^{-|\xi|^2}e^{t^\frac32\beta\frac{i\xi_1}{|\xi|^2}})(t^{-\frac12}x)\label{eq:2.80}\\
		&=t^{-1}(e^{t^\frac32\beta L_1}G_1)(t^{-\frac12}x).\label{eq:2.81}
	\end{align}
Thus we obtain
	\begin{align}
		&\left\| K_{\beta, t}(\cdot-y)-K_{\beta, t}\right\|_{\dot B^s_{a, 1}}
		=t^{-1}\left\|\left(e^{t^\frac32\beta L_1}G_1\right)(t^{-\frac12}(\cdot -y))-\left(e^{t^\frac32\beta L_1}G_1\right)(t^{-\frac12}\cdot)\right\|_{\dot B^s_{a, 1}}\notag\\
		&\lesssim t^{-\frac s2-1+\frac1a}\left\|e^{t^\frac32\beta L_1}\left(G_1(\cdot -t^{-\frac12}y)-G_1\right)\right\|_{\dot B^s_{a, 1}}.
	\end{align}
Bernstein's inequality gives
	\begin{align}
		&\left\|e^{t^\frac32\beta L_1}\left(G_1(\cdot -t^{-\frac12}y)-G_1\right)\right\|_{\dot B^s_{a, 1}}\\
		&\lesssim \left\|e^{t^\frac32\beta L_1}\left(G_1(\cdot -t^{-\frac12}y)-G_1\right)\right\|_{\dot B^{s+1-\frac2a}_{2, 1}}\\
		&=\left\|G_1(\cdot -t^{-\frac12}y)-G_1\right\|_{\dot B^{s+1-\frac2a}_{2, 1}}
		\lesssim \left\|G_1(\cdot -t^{-\frac12}y)-G_1\right\|_{\dot B^{s+2-\frac2a}_{1, 1}}.
	\end{align}
Using the dispersive estimate obtained in Lemma~\ref{lem:2.1}, we also have
	\begin{align}
		&\left\|e^{t^\frac32\beta L_1}\left(G_1(\cdot -t^{-\frac12}y)-G_1\right)\right\|_{\dot B^s_{a, 1}}\\
		&\lesssim |\beta|^{-1+\frac2a}t^{-\frac32(1-\frac2a)}\left\|G_1(\cdot -t^{-\frac12}y)-G_1\right\|_{\dot B^{s+3(1-\frac2a)}_{a', 1}}\\
		&\lesssim |\beta|^{-1+\frac2a}t^{-\frac32(1-\frac2a)}\left\|G_1(\cdot -t^{-\frac12}y)-G_1\right\|_{\dot B^{s+3-\frac4a}_{1, 1}}.
	\end{align}
Thus, by the embeddings
	$W^{s+4-\frac4a, 1}
	\hookrightarrow B^{s+4-\frac4a}_{1, \infty}
	\hookrightarrow  B^{s+3-\frac4a}_{1, 1}
	\hookrightarrow \dot  B^{s+3-\frac4a}_{1, 1}
	\cap \dot  B^{s+2-\frac2a}_{1, 1}, $
we obtain
	\begin{align}
		&\left\|T_\beta(t)\omega_0-K_{\beta, t}\int_{\mathbb R^2}\omega_0(y)\, dy\right\|_{\dot W^{s, a}}\\
		&\lesssim M_{s, a}(t)\int_{\mathbb R^2}\left\|G_1(\cdot -t^{-\frac12}y)-G_1\right\|_{ B^{s+3-\frac4a}_{1, 1}}|\omega_0(y)|\, dy\\
		&\lesssim M_{s, a}(t)\int_{\mathbb R^2}\left\|G_1(\cdot -t^{-\frac12}y)-G_1\right\|_{ B^{s+4-\frac4a}_{1, \infty}}|\omega_0(y)|\, dy\\
		&\lesssim M_{s, a}(t)\int_{\mathbb R^2}\left\|G_1(\cdot -t^{-\frac12}y)-G_1\right\|_{ W^{s+4-\frac4a, 1}}|\omega_0(y)|\, dy.
	\end{align}
Since
	$\omega_0\in L^1$, 
the dominated convergence theorem ensures that
	\begin{align}
		\lim_{t\to\infty}M_{s, a}(t)^{-1}\left\|T_\beta(t)\omega_0-K_{\beta, t}\int_{\mathbb R^2}\omega_0(y)\, dy\right\|_{\dot W^{s, a}}=0.
	\end{align}
This completes the proof.
\qed

\begin{proposition}[Strichartz estimates]\label{prop:2.3}
	Assume that the exponents
		$2<p<\infty$, 
		$2\leq r<\infty$
	and 
		$s\geq 0$
	satisfy either
		\begin{align}
			r>2 
			\quad\text{and} \quad
			\frac{1}{2}\left(1-\frac{2}{p}\right)+\frac s2
			\leq\frac 1r
			\leq\frac54\left(1-\frac 2p\right)+\frac{s}{2}\label{eq:2.31}
		\end{align}
	or
		\begin{align}
			r=2 
			\quad\text{and} \quad
			\frac{1}{2}\left(1-\frac{2}{p}\right)+\frac s2
			<\frac 12
			<\frac54\left(1-\frac 2p\right)+\frac{s}{2}.\label{eq:2.32}
		\end{align}
	Then there is a constant
		$C>0$
	which satisfies
		\begin{align}\label{eq:2.8}
			\left\|T_\beta(t)f\right\|_{L^r([0, \infty); \dot W^{s, p})}
			\leq C\left|\beta\right|^{-\frac{1}{3}(\frac2r+\frac2p-1-s)}\left\|f\right\|_{L^2}
		\end{align}
	for all
		$f\in L^2$
	and
		$\beta\neq 0$.
\end{proposition}
\proof
	The proof is similar to those of \citep[Lemma~3.3]{IT15} and \citep[Lemma~2]{AKL22}. We also refer to \cite{Caz03} as a general reference.
	Using the conditions
		$2<p<\infty$
	and
		$r\geq2$, 
	we have
		\begin{align}\label{eq:2.9}
			\left\|T_\beta(t)f\right\|_{L^r([0, \infty); \dot W^{s, p})}
			&\lesssim\left\|T_\beta(t)f\right\|_{L^r([0, \infty); \dot B^{s}_{p, 2})}\\
			&=\left\|\left\|\left(2^{sk}\left\|T_\beta(t)P_kf\right\|_{L^p}\right)_{k\in\mathbb Z}\right\|_{\ell^2}\right\|_{L^r([0, \infty))}\label{eq:2.10}\\
			&\lesssim\left\|\left(\left\|\left\|T_\beta(t)P_kf\right\|_{\dot W^{s,p}}\right\|_{L^r([0, \infty))}\right)_{k\in\mathbb Z}\right\|_{\ell^2}\label{eq:2.11}.
		\end{align}
	Thus,  we first consider the estimate of 
		$\left\|\left\|T_\beta(t)P_kf\right\|_{\dot W^{s,p}}\right\|_{L^r([0, \infty))}$ for 
		$k\in\mathbb Z$.
		
	Assume that
		$\varphi\in C^\infty_c((0, \infty)\times~\mathbb R^2)$.
	Then the Cauchy--Schwarz inequality yields
		\begin{align}
			&\quad\left|\int_0^\infty\int_{\mathbb R^2}\left(T_\beta(t)(-\Delta)^\frac s2P_kf\right)(x)\overline{\varphi(t, x)}\, dx\, dt\right|\label{eq:2.12}\\
			&=\left|\int_{\mathbb R^2}\int_0^\infty\left(P_kf\right)(x)\overline{\left(T_{-\beta}(t)(-\Delta)^\frac s2\widetilde P_k\varphi(t, \cdot)\right)(x)}\, dt\, dx\right|\label{eq:2.13}\\
			&\leq \left\|P_kf\right\|_{L^2}\left\|\int_0^\infty T_{-\beta}(t)(-\Delta)^\frac s2\widetilde P_k\varphi(t, \cdot)\, dt\right\|_{L^2}.\label{eq:2.14}
		\end{align}
	Moreover, we have
		\begin{align}
			&\quad\left\|\int_0^\infty T_{-\beta}(t)(-\Delta)^\frac s2\widetilde P_k\varphi(t, \cdot)\, dt\right\|_{L^2}^2\label{eq:2.15}\\
			&=\int_0^\infty\int_0^\infty\int_{\mathbb R^2} \left(T_{-\beta}(t)(-\Delta)^\frac s2\widetilde P_k\varphi(t, \cdot)\right)(x)\overline{\left(T_{-\beta}(\tau)(-\Delta)^\frac s2\widetilde P_k\varphi(\tau, \cdot)\right)(x)}\, dx\, dt\, d\tau\label{eq:2.16}\\
			&=\int_0^\infty\int_0^\infty\int_{\mathbb R^2}\left(\widetilde P_k\varphi(t, \cdot)\right)(x)\overline{\left(T_{\beta}(t)T_{-\beta}(\tau)(-\Delta)^s\widetilde P_k\varphi(\tau, \cdot)\right)(x)}\, dx\, dt\, d\tau\label{eq:2.17}\\
			&\leq\int_0^\infty\int_0^\infty \left\|\widetilde P_k\varphi(t, \cdot)\right\|_{L^{p'}}\left\|T_{\beta}(t)T_{-\beta}(\tau)(-\Delta)^s\widetilde P_k\varphi(\tau, \cdot)\right\|_{L^p}\, dt\, d\tau\label{eq:2.18}\\
			&\lesssim\left\|\varphi\right\|_{L^{r'}([0, \infty); L^{p'})}\left\|\int_0^\infty\left\|T_{\beta}(t)T_{-\beta}(\tau)(-\Delta)^s\widetilde P_k\varphi(\tau, \cdot)\right\|_{L^p}\, d\tau\right\|_{L^r([0, \infty))}.\label{eq:2.19}
		\end{align}
	By an argument similar to that of Proposition~\ref{prop:2.2}, we see that
		\begin{align}
			&\quad\left\|T_{\beta}(t)T_{-\beta}(\tau)(-\Delta)^s\widetilde P_k\varphi(\tau, \cdot)\right\|_{L^p}
			=\left\|e^{(t-\tau)\beta L_1}e^{(t+\tau)\Delta}(-\Delta)^s\widetilde P_k\varphi(\tau, \cdot)\right\|_{L^p}\label{eq:2.20}\\
			&\lesssim2^{2sk}\min\left\{2^{2k(1-\frac2p)}, \left|\beta(t-\tau)\right|^{-1+\frac2p}2^{3k(1-\frac 2p)}\right\}e^{-c(t+\tau)2^{2k}}\left\|\varphi(\tau, \cdot)\right\|_{L^{p'}}\label{eq:2.21}\\
			&\lesssim|t-\tau|^{-s-1+\frac2p}\min\left\{1, \left|\beta\right|^{-1+\frac2p}\left|t-\tau\right|^{-\frac32(1-\frac2p)}\right\}\left\|\varphi(\tau, \cdot)\right\|_{L^{p'}}.\label{eq:2.22}
		\end{align}
	Thus, we obtain the estimate
		\begin{align}
			&\left\|\int_0^\infty\left\|T_{\beta}(t)T_{-\beta}(\tau)(-\Delta)^s\widetilde P_k\varphi(\tau, \cdot)\right\|_{L^p}\, d\tau\right\|_{L^r([0, \infty))}\label{eq:2.23}\\
			&\lesssim \left\|\int_0^\infty|t-\tau|^{-s-1+\frac2p}\min\left\{1, \left|\beta\right|^{-1+\frac2p}\left|t-\tau\right|^{-\frac32(1-\frac2p)}\right\}\left\|\varphi(\tau, \cdot)\right\|_{L^{p'}}\, d\tau\right\|_{L^r([0, \infty))}.\quad\label{eq:2.24}
		\end{align}
	If
		$\frac{1}{2}(1-\frac{2}{p})+\frac s2
		<\frac 1r
		<\frac54(1-\frac 2p)+\frac{s}{2}$, 
	Young's convolution inequality yields
		\begin{align}
			&\quad\left\|\int_0^\infty|t-\tau|^{-s-1+\frac2p}\min\left\{1, \left|\beta\right|^{-1+\frac2p}\left|t-\tau\right|^{-\frac32(1-\frac2p)}\right\}\left\|\varphi(\tau, \cdot)\right\|_{L^{p'}}\, d\tau\right\|_{L^r([0, \infty))}\label{eq:2.25}\\
			&\leq \left\|\left|t\right|^{-s-1+\frac2p}\min\left\{1, \left|\beta\right|^{-1+\frac2p}\left|t\right|^{-\frac32(1-\frac2p)}\right\}\right\|_{L^{\frac r2}(\mathbb R)}\left\|\varphi\right\|_{L^{r'}([0, \infty); L^{p'})} \label{eq:2.26}
		\end{align}
	since
		$1+\frac1r=\frac2r+\frac1{r'}$.
For any
		$\lambda>0$,
	we split the integral as
		\begin{align}
			&\quad\left\|\left|t\right|^{-s-1+\frac2p}\min\left\{1, \left|\beta\right|^{-1+\frac2p}\left|t\right|^{-\frac32(1-\frac2p)}\right\}\right\|_{L^{\frac r2}(\mathbb R)}\label{eq:2.27}\\
			&\leq \left(\int_{|t|\leq \lambda}|t|^{-\frac r2(s+1-\frac2p)}\, dt+\int_{|t|>\lambda}\left(|\beta|^{-1+\frac2p}|t|^{-s-\frac52(1-\frac2p)}\right)^\frac r2\, dt\right)^\frac{2}{r}\label{eq:2.28}\\
			&\lesssim\lambda^{\frac 2r-s-1+\frac 2p}+|\beta|^{-1+\frac2p}\lambda^{\frac2r-s-\frac52(1-\frac2p)}.\label{eq:2.29}
		\end{align}
	If we choose 
		$\lambda>0$
	such that
		$\lambda^{\frac 2r-s-1+\frac 2p}=|\beta|^{-1+\frac2p}\lambda^{\frac2r-s-\frac52(1-\frac2p)}$, 
	i.e., 
		$\lambda=\left|\beta\right|^{-\frac23}$
	holds, 
	the above estimate leads to
		\begin{align}
			\left\|\left|t\right|^{-s-1+\frac2p}\min\left\{1, \left|\beta\right|^{-1+\frac2p}\left|t\right|^{-\frac32(1-\frac2p)}\right\}\right\|_{L^{\frac r2}(\mathbb R)}
			\lesssim \left|\beta\right|^{-\frac23(\frac 2r-s-1+\frac 2p)}.\label{eq:2.30}
		\end{align}
	Plugging this into \eqref{eq:2.26}, we obtain
		\begin{align}
			&\quad\left\|\int_0^\infty|t-\tau|^{-s-1+\frac2p}\min\left\{1, \left|\beta\right|^{-1+\frac2p}\left|t-\tau\right|^{-\frac32(1-\frac2p)}\right\}\left\|\varphi(\tau, \cdot)\right\|_{L^{p'}}\, d\tau\right\|_{L^r([0, \infty))}\label{eq:2.37}\\
			&\lesssim |\beta|^{-\frac{2}{3}(\frac2r-s-1+\frac2p)}\left\|\varphi\right\|_{L^{r'}([0, \infty); L^{p'})}.\label{eq:2.38}
		\end{align}
	If
		$r>2$
	and
		$\frac{1}{2}(1-\frac2p)+\frac s2=\frac1r$, 
	the Hardy--Littlewood--Sobolev inequality yields the same estimate
		\begin{align}
			&\quad\left\|\int_0^\infty|t-\tau|^{-s-1+\frac2p}\min\left\{1, \left|\beta\right|^{-1+\frac2p}\left|t-\tau\right|^{-\frac32(1-\frac2p)}\right\}\left\|\varphi(\tau, \cdot)\right\|_{L^{p'}}\, d\tau\right\|_{L^r([0, \infty))}\label{eq:2.35}\\
			&\lesssim \left\|\varphi\right\|_{L^{r'}([0, \infty); L^{p'})}
			=|\beta|^{-\frac{2}{3}(\frac2r-s-1+\frac2p)}\left\|\varphi\right\|_{L^{r'}([0, \infty); L^{p'})}.\label{eq:2.36}
		\end{align}
		
	In the case of
		$r>2$
	and
		$\frac 1r=\frac54(1-\frac2p)+\frac s2, $
	the Hardy--Littlewood--Sobolev inequality gives 
		\begin{align}
			&\quad\left\|\int_0^\infty|t-\tau|^{-s-1+\frac2p}\min\left\{1, \left|\beta\right|^{-1+\frac2p}\left|t-\tau\right|^{-\frac32(1-\frac2p)}\right\}\left\|\varphi(\tau, \cdot)\right\|_{L^{p'}}\, d\tau\right\|_{L^r([0, \infty))}\label{eq:2.39}\\
			&\lesssim \left\|\left|t\right|^{-s-1+\frac2p}\min\left\{1, \left|\beta\right|^{-1+\frac2p}\left|t\right|^{-\frac32(1-\frac2p)}\right\}\right\|_{L^{\frac r2, \infty}(\mathbb R)}\left\|\varphi\right\|_{L^{r'}([0, \infty); L^{p'})} \label{eq:2.40}\\
			&\lesssim \left|\beta\right|^{-1+\frac2p}\left\|\varphi\right\|_{L^{r'}([0, \infty); L^{p'})}. \label{eq:2.45}
		\end{align}
	Since we have
		$-\frac23(\frac2r-s-1+\frac2p)=-1+\frac2p, $
	\eqref{eq:2.45} is equivalent to \eqref{eq:2.38}. 
	
	In each case by combining \eqref{eq:2.14}, \eqref{eq:2.19}, \eqref{eq:2.24}, and \eqref{eq:2.38}, 
	we have
		\begin{align}
			\left\|\left\|T_\beta(t)P_kf\right\|_{\dot W^{s,p}}\right\|_{L^r([0, \infty))}
			\lesssim\left|\beta\right|^{-\frac13(\frac 2r-s-1+\frac 2p)}\left\|P_kf\right\|_{L^2}.\label{eq:2.33}
		\end{align}
	Summing up this estimate with respect to 
		$k\in\mathbb Z$
	and plugging the consequent estimate into \eqref{eq:2.11}, we obtain the desired estimate
		\begin{align}
			\left\|\left\|T_\beta(t)f\right\|_{\dot W^{s,p}}\right\|_{L^r([0, \infty))}
			\lesssim\left|\beta\right|^{-\frac13(\frac 2r-s-1+\frac 2p)}\left\|f\right\|_{L^2}.\label{eq:2.34}
		\end{align}
\qed

We also prepare some estimates necessary for dealing with the nonlinear term.
	\begin{proposition}[Fractional Leibniz rule]\label{prop:2.4}
		Let 
			$s\geq 0$ 
		and 
			$1<r<\infty$. 
		Assume also that 
			$1<p_1, p_2, q_1, q_2\leq\infty$
		satisfy
			$\frac1r=\frac{1}{p_j}+\frac{1}{q_j}$
		for
			$j=1, 2$.
		Then there exists a constant
			$C>0$
		such that
			\begin{align}
				\left\|fg\right\|_{\dot W^{s, r}}
				\leq C\left(\left\|f\right\|_{\dot W^{s, p_1}}\left\|g\right\|_{L^{q_1}}+\left\|f\right\|_{L^{q_2}}\left\|g\right\|_{\dot W^{s, p_2}}\right)\label{eq:2.46}
			\end{align}
		for all
			$f\in \dot W^{s, p_1}\cap L^{q_2}$
		and
			$g\in \dot W^{s, p_2}\cap L^{q_1}$.
	\end{proposition}
\proof
	See \citep[Theorem~1]{GO14}, for example.
\qed

\begin{proposition}\label{prop:2.5}
	Let 
		$p_1, p_2\in(2, \infty), r_1, r_2\in[2, \infty)$ 
	and
		$0\leq \delta<\min\{\frac2{p_1}, \frac2{p_2}\}$
	satisfy
    		\begin{align}
    			&\max\left\{1-\frac1{p_1}, 1-\frac1{p_2}\right\}\leq \frac{1}{p_1}+\frac{1}{p_2}-\frac{\delta}2, \label{eq:2.47}\\
			&\frac12-\frac1{p_1}+\frac\delta2-\frac32\left(1-\frac2{p_2}\right)
			\leq\frac{1}{r_1}
			\leq\frac12-\frac1{p_1}+\frac\delta2, \label{eq:2.48}\\
			&1-\frac1{p_2}+\frac\delta2-\frac32\left(1-\frac2{p_1}\right)
			\leq\frac{1}{r_2}
			\leq1-\frac1{p_2}+\frac\delta2.\label{eq:2.49}
    		\end{align}
	Then there is a constant 
		$C>0$
	such that for all
		$\omega_1\in L^{r_1}([0, \infty); \dot W^{-1+\delta, p_1})$
	and
		$\omega_2\in L^{r_2}([0, \infty); \dot W^{\delta, p_2})$, 
	we have
		\begin{align}
			&\quad\left\|\int_0^tT_\beta(t-\tau)\operatorname{div}(u_1(\tau)\omega_2(\tau))\, d\tau\right\|_{L^{r_1}([0, \infty); \dot W^{-1+\delta,p_1})}\label{eq:2.50}\\
			&\leq C\left|\beta\right|^{-\frac23(1-\frac1{r_2}-\frac{1}{p_2}+\frac\delta2)}\left\|\omega_1\right\|_{L^{r_1}([0, \infty); \dot W^{-1+\delta, p_1})}\left\|\omega_2\right\|_{L^{r_2}([0, \infty); \dot W^{\delta, p_2})}\label{eq:2.51}
		\end{align}
	and
		\begin{align}
			&\quad\left\|\int_0^tT_\beta(t-\tau)\operatorname{div}(u_1(\tau)\omega_2(\tau))\, d\tau\right\|_{L^{r_2}([0, \infty); \dot W^{\delta,p_2})}\label{eq:2.52}\\
			&\leq C\left|\beta\right|^{-\frac23(\frac12-\frac1{r_1}-\frac{1}{p_1}+\frac\delta2)}\left\|\omega_1\right\|_{L^{r_1}([0, \infty); \dot W^{-1+\delta, p_1})}\left\|\omega_2\right\|_{L^{r_2}([0, \infty); \dot W^{\delta, p_2})}, \label{eq:2.53}
		\end{align}
	where 
		$u_j\coloneqq \nabla^\perp(-\Delta)^{-1}\omega_j$
	for
		$j=1, 2.$
\end{proposition}
\begin{remark}\rm
		The condition 
			$\max\left\{1-\frac1{p_1}, 1-\frac1{p_2}\right\}\leq \frac{1}{p_1}+\frac{1}{p_2}-\frac{\delta}2$
		is enough to obtain
			$\delta<\min\{\frac2{p_1}, \frac2{p_2}\}.$
		Indeed, we have
			$1-\frac1{p_2}\leq\frac{1}{p_1}+\frac1{p_2}-\frac\delta2$
		and this is equivalent to
			$\delta\leq \frac2{p_1}+\frac4{p_2}-2.$
		This in turn yields
			$\delta<\frac2{p_1}, $
		since we have 
			$p_2>2.$
		we similarly obtain
			$\delta<\frac2{p_2}.$
		We also note that the equality in
			$\frac1{r_2}\leq 1-\frac1{p_2}+\frac\delta2$
		cannot be attained, since we have
			$1-\frac1{p_2}+\frac\delta2
			\geq 1-\frac1{p_2}
			>\frac12
			\geq\frac1{r_2}.$
\end{remark}
\proof[Proof of Proposition~\ref{prop:2.5}]
	Let 
		$q\in(1, \infty)$ 
	be defined by
		$\frac{1}{q}=\frac{1}{p_1}+\frac{1}{p_2}-\frac{\delta}2.$
	We first prove \eqref{eq:2.51}. By Proposition~\ref{prop:2.2} and Proposition~\ref{prop:2.4}, we have
		\begin{align}
			&\quad\left\|\int_0^tT_\beta(t-\tau)\operatorname{div}(u_1(\tau)\omega_2(\tau))\, d\tau\right\|_{\dot W^{-1+\delta,p_1}}\label{eq:2.54}\\
			&\lesssim\int_0^t\left\|T_\beta(t-\tau)(-\Delta)^\frac\delta2(u_1(\tau)\omega_2(\tau))\right\|_{L^{p_1}}\, d\tau\label{eq:2.55}\\
			&\lesssim \int_0^t\left\|T_{2\beta}\left(\frac{t-\tau}{2}\right)e^{\frac12(t-\tau)\Delta}(-\Delta)^\frac\delta2(u_1(\tau)\omega_2(\tau))\right\|_{\dot B^0_{p_1, 2}}\, d\tau\label{eq:2.56}\\
			&\lesssim \int_0^t(t-\tau)^{-1+\frac2{p_1}}\min\left\{1, |\beta|^{-1+\frac2{p_1}}(t-\tau)^{-\frac32(1-\frac2{p_1})}\right\}\left\|e^{\frac12(t-\tau)\Delta}(-\Delta)^\frac\delta2(u_1(\tau)\omega_2(\tau))\right\|_{\dot B^0_{p_1', 2}}\, d\tau\notag\\
			&\lesssim \int_0^t(t-\tau)^{-1+\frac2{p_1}-\frac1q+\frac1{p_1'}}\min\left\{1, |\beta|^{-1+\frac2{p_1}}(t-\tau)^{-\frac32(1-\frac2{p_1})}\right\}\left\|(-\Delta)^\frac\delta2(u_1(\tau)\omega_2(\tau))\right\|_{L^{q}}\, d\tau\notag\\
			&\lesssim \int_0^t(t-\tau)^{-\frac1{p_2}+\frac\delta2}\min\left\{1, |\beta|^{-1+\frac2{p_1}}(t-\tau)^{-\frac32(1-\frac2{p_1})}\right\}\left\|\omega_1(\tau)\right\|_{\dot W^{-1+\delta, p_1}}\left\|\omega_2(\tau)\right\|_{\dot W^{\delta, p_2}}\, d\tau, \notag
		\end{align}
	where we have also used the embeddings 
		$L^{p'}\hookrightarrow \dot B^{0}_{p', 2}, \dot B^{0}_{p, 2}\hookrightarrow L^p$
	for
		$2\leq p<\infty$
	and the Sobolev embeddings. 
	If
		$1-\frac1{p_2}+\frac\delta2-\frac32\left(1-\frac2{p_1}\right)
		<\frac{1}{r_2}
		<1-\frac1{p_2}+\frac\delta2$, 
	Young's convolution inequality leads to 
		\begin{align}
			&\quad\left\|\int_0^tT_\beta(t-\tau)\operatorname{div}(u_1(\tau)\omega_2(\tau))\, d\tau\right\|_{L^{r_1}([0, \infty); \dot W^{-1+\delta,p_1})}\label{eq:2.60}\\
			&\lesssim \left\|t^{-\frac1{p_2}+\frac\delta2}\min\left\{1, |\beta|^{-1+\frac2{p_1}}t^{-\frac32(1-\frac2{p_1})}\right\}\right\|_{L^{r_2'}([0, \infty))}\left\|\omega_1\right\|_{L^{r_1}([0, \infty); \dot W^{-1+\delta, p_1})}\left\|\omega_2\right\|_{L^{r_2}([0, \infty); \dot W^{\delta, p_2})}.\notag
		\end{align}
	By splitting the range of the integral at
		$t=\left|\beta\right|^{-\frac23}, $
	we have
		\begin{align}
			\left\|t^{-\frac1{p_2}+\frac\delta2}\min\left\{1, |\beta|^{-1+\frac2{p_1}}t^{-\frac32(1-\frac2{p_1})}\right\}\right\|_{L^{r_2'}([0, \infty))}
			\lesssim \left|\beta\right|^{-\frac23(1-\frac1{r_2}-\frac{1}{p_2}+\frac\delta2)}
			\label{eq:2.62}
		\end{align}
	so we obtain \eqref{eq:2.51}. If 
		$\frac{1}{r_2}
		=1-\frac1{p_2}+\frac\delta2-\frac32\left(1-\frac2{p_1}\right), $
	the Hardy--Littlewood--Sobolev inequality yields the same result.
	
	We can prove \eqref{eq:2.53} in a similar way. In fact, 
		\begin{align}
			&\quad\left\|\int_0^tT_\beta(t-\tau)\operatorname{div}(u_1(\tau)\omega_2(\tau))\, d\tau\right\|_{\dot W^{\delta,p_2}}\label{eq:2.63}\\
			&\leq\int_0^t\left\|T_\beta(t-\tau)(-\Delta)^\frac{\delta}2\operatorname{div}(u_1(\tau)\omega_2(\tau))\right\|_{L^{p_2}}\, d\tau\label{eq:2.64}\\
			&\lesssim \int_0^t\left\|T_{2\beta}\left(\frac{t-\tau}{2}\right)e^{\frac12(t-\tau)\Delta}(-\Delta)^\frac{\delta}2\operatorname{div}(u_1(\tau)\omega_2(\tau))\right\|_{\dot B^0_{p_2, 2}}\, d\tau\label{eq:2.65}\\
			&\lesssim \int_0^t(t-\tau)^{-1+\frac2{p_2}}\min\left\{1, |\beta|^{-1+\frac2{p_2}}(t-\tau)^{-\frac32(1-\frac2{p_2})}\right\}\left\|e^{\frac12(t-\tau)\Delta}(-\Delta)^\frac{\delta}2\operatorname{div}(u_1(\tau)\omega_2(\tau))\right\|_{\dot B^0_{p_2', 2}}\, d\tau\notag\\
			&\lesssim \int_0^t(t-\tau)^{-1+\frac2{p_2}-\frac1q+\frac1{p_2'}-\frac12}\min\left\{1, |\beta|^{-1+\frac2{p_2}}(t-\tau)^{-\frac32(1-\frac2{p_2})}\right\}\left\|(-\Delta)^\frac{\delta}2(u_1(\tau)\omega_2(\tau))\right\|_{L^{q}}\, d\tau\notag\\
			&\lesssim \int_0^t(t-\tau)^{-\frac12-\frac1{p_1}+\frac\delta2}\min\left\{1, |\beta|^{-1+\frac2{p_2}}(t-\tau)^{-\frac32(1-\frac2{p_2})}\right\}\left\|\omega_1(\tau)\right\|_{\dot W^{-1+\delta, p_1}}\left\|\omega_2(\tau)\right\|_{\dot W^{\delta, p_2}}\, d\tau, \notag
		\end{align}
	and applying Young's convolution inequality (or the Hardy--Littlewood--Sobolev inequality) leads to the desired estimate.
\qed

\section{Proof of Theorem~\ref{thm:1.1}}\label{sect:3}
In this section, we prove the global well-posedness of \eqref{eq:1.9}.
\proof[Proof of Theorem~\ref{thm:1.1}]
	Define the function spaces
		$Y_1$
	and
		$Y_2$
	by
		\begin{align*}
			Y_1\coloneqq L^{r_1}([0, \infty); \dot W^{-1+\delta, p_1})
			\quad\text{and}\quad
			Y_2\coloneqq L^{r_2}([0, \infty); \dot W^{\delta, p_2}), 
		\end{align*}
	respectively.
	Our strategy is based on the standard application of Banach's fixed point theorem similar to \cite{IT13}. Since 
		$(s, p, r)=(0, p_i, r_i)$, 
		$i=1, 2$
	satisfies the assumption of Proposition~\ref{prop:2.3}, there is a constant
		$C_1>0$
	such that
		\begin{align}
			\left\|T_\beta(t)\omega_0\right\|_{Y_1}
			\leq C_1|\beta|^{-\frac13(\frac2{r_1}+\frac2{p_1}-1)}\left\|\omega_0\right\|_{\dot H^{-1+\delta}}\label{eq:3.1}
		\end{align}
	and
		\begin{align}
			\left\|T_\beta(t)\omega_0\right\|_{Y_2}
			\leq C_1|\beta|^{-\frac13(\frac2{r_2}+\frac2{p_2}-1)}\left\|\omega_0\right\|_{\dot H^{\delta}}\label{eq:3.2}
		\end{align}
	for all
		$\omega_0\in\dot H^{-1+\delta}\cap \dot H^\delta.$
	For a fixed initial datum 
		$\omega_0\in\dot H^{-1+\delta}\cap \dot H^\delta, $
	define a function space
		$Y$
	by
		\begin{align}
			Y\coloneqq\left\{
			\omega \in Y_1\cap Y_2\;\middle|\;
				\begin{aligned}
					& \left\|\omega\right\|_{Y_1}
			\leq 2C_1|\beta|^{-\frac13(\frac2{r_1}+\frac2{p_1}-1)}\left\|\omega_0\right\|_{\dot H^{-1+\delta}},\\
					& \left\|\omega\right\|_{Y_2}
			\leq 2C_1|\beta|^{-\frac13(\frac2{r_2}+\frac2{p_2}-1)}\left\|\omega_0\right\|_{\dot H^{\delta}}
				\end{aligned}
			\right\},\label{eq:3.3}
		\end{align}
	and its metric function
		$d$
	by
		\begin{align}
			d(\omega_1, \omega_2)
			\coloneqq \left|\beta\right|^{-\frac13(1+\delta -\frac2{p_1}-\frac{2}{r_1})}\left\|\omega_1-\omega_2\right\|_{Y_1}
			+\left|\beta\right|^{-\frac13(2+\delta -\frac2{p_2}-\frac{2}{r_2})}\left\|\omega_1-\omega_2\right\|_{Y_2}.\label{eq:3.4}
		\end{align}
	We also define a map
		$\Phi=\Phi(\omega)$
	for 
		$\omega\in Y$
	by
		\begin{align}
			\Phi(\omega)(t)
			\coloneqq T_\beta(t)\omega_0-\int_0^tT_\beta(t-\tau)\operatorname{div}(u(\tau)\omega(\tau))\, d\tau.\label{eq:3.5}
		\end{align}
	We show that 
		$\Phi$
	is a contraction on the complete metric space
		$(Y, d)$
	under a suitable choice of 
		$\omega_0$
	and
		$\beta$.
	For 
		$\omega\in Y$, 
	we have
		\begin{align}
			\left\|\Phi(\omega)\right\|_{Y_1}
			&\leq \left\|T_\beta(t)\omega_0\right\|_{Y_1}
			+\left\|\int_0^tT_\beta(t-\tau)\operatorname{div}(u(\tau)\omega(\tau))\, d\tau\right\|_{Y_1}, \label{eq:3.17}\\
			\left\|\Phi(\omega)\right\|_{Y_2}
			&\leq \left\|T_\beta(t)\omega_0\right\|_{Y_2}
			+\left\|\int_0^tT_\beta(t-\tau)\operatorname{div}(u(\tau)\omega(\tau))\, d\tau\right\|_{Y_2}\label{eq:3.18}
		\end{align}
	and by Proposition~\ref{prop:2.5}, there is a constant 
		$C_2>0$
	satisfying
		\begin{align}
			\left\|\Phi(\omega)\right\|_{Y_1}
			&\leq C_1|\beta|^{-\frac13(\frac2{r_1}+\frac2{p_1}-1)}\left\|\omega_0\right\|_{\dot H^{-1+\delta}}
			+C_2\left|\beta\right|^{-\frac23(1-\frac1{r_2}-\frac{1}{p_2}+\frac\delta2)}\left\|\omega\right\|_{Y_1}\left\|\omega\right\|_{Y_2}\label{eq:3.19}\\
			&\leq C_1|\beta|^{-\frac13(\frac2{r_1}+\frac2{p_1}-1)}\left\|\omega_0\right\|_{\dot H^{-1+\delta}}\label{eq:3.20}\\
			&\quad+4{C_1}^2C_2|\beta|^{-\frac13(\frac2{r_1}+\frac2{p_1}-1)}\left\|\omega_0\right\|_{\dot H^{-1+\delta}}\left|\beta\right|^{-\frac{1+\delta}{3}}\left\|\omega_0\right\|_{\dot H^{\delta}}\label{eq:3.21}
		\end{align}
	and
		\begin{align}
			\left\|\Phi(\omega)\right\|_{Y_2}
			&\leq C_1|\beta|^{-\frac13(\frac2{r_2}+\frac2{p_2}-1)}\left\|\omega_0\right\|_{\dot H^{\delta}}
			+C_2\left|\beta\right|^{-\frac23(\frac12-\frac1{r_1}-\frac{1}{p_1}+\frac\delta2)}\left\|\omega\right\|_{Y_1}\left\|\omega\right\|_{Y_2}\label{eq:3.22}\\
			&\leq C_1|\beta|^{-\frac13(\frac2{r_2}+\frac2{p_2}-1)}\left\|\omega_0\right\|_{\dot H^{\delta}}\label{eq:3.23}\\
			&\quad+4{C_1}^2C_2|\beta|^{-\frac13(\frac2{r_2}+\frac2{p_2}-1)}\left\|\omega_0\right\|_{\dot H^{-1+\delta}}\left|\beta\right|^{-\frac\delta3}\left\|\omega_0\right\|_{\dot H^{\delta}}.\label{eq:3.24}
		\end{align}
	
	For $\omega_1, \omega_2\in Y$, we also have
		\begin{align}
			&\quad\Phi(\omega_1)-\Phi(\omega_2)\label{eq:3.6}\\
			&=-\int_0^tT_\beta(t-\tau)\operatorname{div}((u_1(\tau)-u_2(\tau))\omega_1(\tau))\, d\tau-\int_0^tT_\beta(t-\tau)\operatorname{div}(u_2(\tau)(\omega_1(\tau)-\omega_2(\tau)))\, d\tau\notag
		\end{align}
	so again by Proposition~\ref{prop:2.5}, it follows that
		\begin{align}
			&\quad\left\|\Phi(\omega_1)-\Phi(\omega_2)\right\|_{Y_1}\label{eq:3.8}\\
			&\leq C_2\left|\beta\right|^{-\frac23(1-\frac1{r_2}-\frac{1}{p_2}+\frac\delta2)}\left(\left\|\omega_1-\omega_2\right\|_{Y_1}\left\|\omega_1\right\|_{Y_2}+\left\|\omega_2\right\|_{Y_1}\left\|\omega_1-\omega_2\right\|_{Y_2}\right)\label{eq:3.9}\\
			&\leq 2C_1C_2\left(\left|\beta\right|^{-\frac13(1+\delta)}\left\|\omega_0\right\|_{\dot H^\delta}\left\|\omega_1-\omega_2\right\|_{Y_1}+\left|\beta\right|^{-\frac13(\frac2{p_1}+\frac{2}{r_1}-\frac{2}{p_2}-\frac{2}{r_2}+1+\delta)}\left\|\omega_0\right\|_{\dot H^{-1+\delta}}\left\|\omega_1-\omega_2\right\|_{Y_2}\right)\notag
		\end{align}
	and
		\begin{align}
			&\quad\left\|\Phi(\omega_1)-\Phi(\omega_2)\right\|_{Y_2}\label{eq:3.11}\\
			&\leq C_2\left|\beta\right|^{-\frac23(\frac12-\frac1{r_1}-\frac{1}{p_1}+\frac\delta2)}\left(\left\|\omega_1-\omega_2\right\|_{Y_1}\left\|\omega_1\right\|_{Y_2}+\left\|\omega_2\right\|_{Y_1}\left\|\omega_1-\omega_2\right\|_{Y_2}\right)\label{eq:3.12}\\
			&\leq 2C_1C_2\left(\left|\beta\right|^{-\frac13(\frac2{p_2}+\frac{2}{r_2}-\frac{2}{p_1}-\frac{2}{r_1}+\delta)}\left\|\omega_0\right\|_{\dot H^\delta}\left\|\omega_1-\omega_2\right\|_{Y_1}+\left|\beta\right|^{-\frac\delta3}\left\|\omega_0\right\|_{\dot H^{-1+\delta}}\left\|\omega_1-\omega_2\right\|_{Y_2}\right).\notag
		\end{align}
	Combining these estimates leads to
		\begin{align}
			&\quad d(\Phi(\omega_1), \Phi(\omega_2))\label{eq:3.14}\\
			&\leq 4C_1C_2\left(\left|\beta\right|^{-\frac13(1+\delta)}\left\|\omega_0\right\|_{\dot H^{\delta}}\left|\beta\right|^{-\frac13(1-\frac{2}{r_1}-\frac{2}{p_1}+\delta)}\left\|\omega_1-\omega_2\right\|_{Y_1}\right.\label{eq:3.15}\\
			&\phantom{\leq 2C_1C_2(|\beta}\left.+\left|\beta\right|^{-\frac\delta3}\left\|\omega_0\right\|_{\dot H^{-1+\delta}}\left|\beta\right|^{-\frac13(2-\frac2{r_2}-\frac2{p_2}+\delta)}\left\|\omega_1-\omega_2\right\|_{Y_2}\right).\label{eq:3.16}
		\end{align}
	Thus if the condition
		$ 4C_1C_2\max\left\{\left|\beta\right|^{-\frac13(1+\delta)}\left\|\omega_0\right\|_{\dot H^{\delta}}, \left|\beta\right|^{-\frac\delta3}\left\|\omega_0\right\|_{\dot H^{-1+\delta}}\right\}\leq \frac12$
	is fulfilled, the map
		$\Phi$
	defines a contraction on the complete metric space
		$(Y, d).$
	Now Banach's fixed point theorem ensures that 
		$\Phi$
	has a unique fixed point
		$\omega\in Y, $
	and this gives the desired solution of \eqref{eq:1.9}.
	
	Next we show the continuity of the solution 
		$\omega$. 
	Let us start by confirming that
		$\omega(t)\in \dot H^{-1+\delta}$
	for any
		$t\geq 0.$
	We have
		\begin{align}
			\left\|\omega(t)\right\|_{\dot H^{-1+\delta}}
			&\leq \left\|T_\beta(t)\omega_0\right\|_{\dot H^{-1+\delta}}
			+\int_0^t\left\|T_\beta(t-\tau)\operatorname{div}(u(\tau)\omega(\tau))\right\|_{\dot H^{-1+\delta}}\, d\tau\label{eq:3.25}
		\end{align}
	and obviously
		$\left\|T_\beta(t)\omega_0\right\|_{\dot H^{-1+\delta}}
		\leq \left\|\omega_0\right\|_{\dot H^{-1+\delta}}
		<\infty.$
	For the nonlinear term, since the assumption of Theorem~\ref{thm:1.1} ensures that
		\begin{align}
			\frac1q
			=\frac1{p_1}+\frac{1}{p_2}-\frac\delta2
			\geq \max\left\{1-\frac1{p_1}, 1-\frac1{p_2}\right\}
    			>\frac12,\label{eq:3.26}
		\end{align}
	we have
		\begin{align}
			&\quad\int_0^t\left\|T_\beta(t-\tau)\operatorname{div}(u(\tau)\omega(\tau))\right\|_{\dot H^{-1+\delta}}\, d\tau\label{eq:3.27}\lesssim\int_0^t\left\|e^{(t-\tau)\Delta}(-\Delta)^\frac\delta2(u(\tau)\omega(\tau))\right\|_{ L^{2}}\, d\tau\\
			&\lesssim\int_0^t(t-\tau)^{-\frac1q+\frac12}\left\|(-\Delta)^\frac\delta2(u(\tau)\omega(\tau))\right\|_{ L^{q}}\, d\tau
			\label{eq:3.28}\\
			&\lesssim \int_0^t(t-\tau)^{-\frac1q+\frac12}\left\|u(\tau)\right\|_{\dot W^{-1+\delta, p_1}}\left\|\omega(\tau)\right\|_{\dot W^{\delta, p_2}}\, d\tau
			\label{eq:3.29}\\
			&\leq \left\|\omega\right\|_{Y_1}\left\|\omega\right\|_{Y_2}\left(\int_0^t(t-\tau)^{-r_3(\frac1q-\frac12)}\, d\tau\right)^{\frac1{r_3}}
			\label{eq:3.30}, 
		\end{align}
	where we used Proposition~\ref{prop:2.4} and set
		$\frac{1}{r_3}=1-\frac1{r_1}-\frac1{r_2}$. 
	To obtain the finiteness of the last term, it is enough to show
		$\frac1q-\frac12<1-\frac1{r_1}-\frac1{r_2}, $
	which is equivalent to
		$\frac1{r_1}+\frac1{r_2}<\frac{1}{2}-\frac1{p_1}+\frac\delta 2+1-\frac1{p_2}.$
	The assumption of Theorem~\ref{thm:1.1} gives
		\begin{align}
			\frac{1}{r_1}\leq \frac12-\frac1{p_1}+\frac\delta2
			\quad\text{and}\quad
			\frac1{r_2}\leq \frac54\left(1-\frac2{p_2}\right), \label{eq:3.31}
		\end{align}
	so it suffices to show
		$\frac{5}{4}(1-\frac2{p_2})
		< 1-\frac1{p_2}, $
		i.e., 
			$p_2< 6$.
		This is true because we have
			\begin{align}
				1-\frac{1}{p_2}
				\leq \frac1{p_1}+\frac1{p_2}-\frac\delta2
				\leq \frac12+\frac1{p_2}\label{eq:3.32}
			\end{align}
		and this ensures that
			$p_2\leq 4.$
		We consequently have
			\begin{align}
				\int_0^t\left\|T_\beta(t-\tau)\operatorname{div}(u(\tau)\omega(\tau))\right\|_{\dot H^{-1+\delta}}\, d\tau
				\lesssim \left\|\omega\right\|_{Y_1}\left\|\omega\right\|_{Y_2}t^{1-\frac{1}{r_1}-\frac{1}{r_2}-\frac1q+\frac12}
				<\infty, \label{eq:3.33}
			\end{align}
		and this implies 
			$\omega(t)\in \dot H^{-1+\delta}$
		for all
			$t\geq 0$.
		It is clear that 
			$\left\|T_\beta(t)\omega_0-\omega_0\right\|_{\dot H^{-1+\delta}}\to 0$
		as
			$t\downarrow 0$
		so the above estimate also gives the continuity of
			$\omega$
		at
			$t=0$.
		Now we are left to prove that 
			$\omega(t)\in\dot H^{-1+\delta}$
		is continuous for
			$t\in (0, \infty).$
			
		Let 
			$M_0, s, t>0$
		satisfy
			$0<s<t<M_0.$
		It is sufficient to show
			$\left\|\omega(t)-\omega(s)\right\|_{\dot H^{-1+\delta}}\to 0$
		as
			$t-s\to 0.$
		We write
			\begin{align}
				&\quad\omega(t)-\omega(s)
				=T_\beta(t)\omega_0-T_\beta(s)\omega_0
				\quad\label{eq:3.34}\\
				&-\int_0^te^{(t-\tau)\Delta}e^{(t-\tau)\beta L_1}\operatorname{div}(u(\tau)\omega(\tau))\, d\tau
				+\int_0^se^{(t-\tau)\Delta}e^{(t-\tau)\beta L_1}\operatorname{div}(u(\tau)\omega(\tau))\, d\tau\quad
				\label{eq:3.35}\\
				&-\int_0^se^{(t-\tau)\Delta}e^{(t-\tau)\beta L_1}\operatorname{div}(u(\tau)\omega(\tau))\, d\tau
				+\int_0^se^{(s-\tau)\Delta}e^{(t-\tau)\beta L_1}\operatorname{div}(u(\tau)\omega(\tau))\, d\tau\quad
				\label{eq:3.36}\\
				&-\int_0^se^{(s-\tau)\Delta}e^{(t-\tau)\beta L_1}\operatorname{div}(u(\tau)\omega(\tau))\, d\tau
				+\int_0^se^{(s-\tau)\Delta}e^{(s-\tau)\beta L_1}\operatorname{div}(u(\tau)\omega(\tau))\, d\tau.\label{eq:3.37}
			\end{align}
		We show that each line tends to 0 as
			$t-s\to 0$.
			
		For \eqref{eq:3.34}, 
			\begin{align}
				&\quad\left\|T_\beta(t)\omega_0-T_\beta(s)\omega_0\right\|_{\dot H^{-1+\delta}}\label{eq:3.38}\\
				&\leq \left\|e^{t\Delta}e^{t\beta L_1}\omega_0-e^{s\Delta}e^{t\beta L_1}\omega_0\right\|_{\dot H^{-1+\delta}}
				+\left\|e^{s\Delta}e^{t\beta L_1}\omega_0-e^{s\Delta}e^{s\beta L_1}\omega_0\right\|_{\dot H^{-1+\delta}}\label{eq:3.39}\\
				&\leq\left\|e^{(t-s)\Delta}\omega_0-\omega_0\right\|_{\dot H^{-1+\delta}}
				+\left\|e^{(t-s)\beta L_1}\omega_0-\omega_0\right\|_{\dot H^{-1+\delta}}\label{eq:3.40}
			\end{align}
		and this goes to 0 as $t-s\to 0$, since
			$\omega_0\in \dot H^{-1+\delta}$.
			
		For \eqref{eq:3.35}, an argument similar to that of the proof of
			$\omega(t)\in\dot H^{-1+\delta}$
		for
			$t\geq 0$
		 gives
			\begin{align}
				\left\|\int_s^tT_\beta(t-\tau)\operatorname{div}(u(\tau)\omega(\tau))\, d\tau\right\|_{\dot H^{-1+\delta}}
				&\lesssim \left\|\omega\right\|_{Y_1}\left\|\omega\right\|_{Y_2}(t-s)^{1-\frac{1}{r_1}-\frac{1}{r_2}-\frac1q+\frac12}\label{eq:3.41}\\
				&\to 0\quad\text{as $t-s\to 0$.}\label{eq:3.42}
			\end{align}
		For \eqref{eq:3.36}, we see that
			\begin{align}
				&\quad\left\|\int_0^s\left(e^{(t-\tau)\Delta}-e^{(s-\tau)\Delta}\right)e^{(t-\tau)\beta L_1}\operatorname{div}(u(\tau)\omega(\tau))\, d\tau\right\|_{\dot H^{-1+\delta}}\label{eq:3.43}\\
				&\lesssim \int_0^s\left\|\left(e^{(t-s)\Delta}-1\right)e^{(s-\tau)\Delta}(-\Delta)^\frac\delta2\left(u(\tau)\omega(\tau)\right)\right\|_{L^2}\, d\tau.\label{eq:3.44}
			\end{align}
		By the Plancherel theorem, we have
			\begin{align}
				&\left\|\left(e^{(t-s)\Delta}-1\right)e^{(s-\tau)\Delta}(-\Delta)^\frac\delta2\left(u(\tau)\omega(\tau)\right)\right\|_{L^2}\label{eq:3.45}\\
				&=\frac1{2\pi}\left(\int_{\mathbb R^2}\left|e^{-(t-s)|\xi|^2}-1\right|^2\left|\mathcal F\left(e^{(s-\tau)\Delta}(-\Delta)^\frac\delta2\left(u(\tau)\omega(\tau)\right)\right)(\xi)\right|^2\, d\xi\right)^\frac12\label{eq:3.46}
			\end{align}
		for each
			$\tau\in (0, s).$
		Now take
			$\gamma_1\in (0, 1)$
		small enough to satisfy
			$\frac1q-\frac12+\gamma_1<1-\frac1{r_1}-\frac1{r_2}.$ 
		Then
		\begin{align}
			\left|e^{-(t-s)|\xi|^2}-1\right|
			\leq \left|(t-s)|\xi|^2\int_0^1e^{-\theta(t-s)|\xi|^2}\, d\theta\right|^{\gamma_1}
			\leq (t-s)^{\gamma_1}|\xi|^{2\gamma_1}\qquad
			\label{eq:3.47}
		\end{align}
	so we have
		\begin{align}
			&\left\|\left(e^{(t-s)\Delta}-1\right)e^{(s-\tau)\Delta}(-\Delta)^\frac\delta2\left(u(\tau)\omega(\tau)\right)\right\|_{L^2}\label{eq:3.48}\\
			&\lesssim (t-s)^{\gamma_1}\left(\int_{\mathbb R^2}\left|\mathcal F\left(e^{(s-\tau)\Delta}(-\Delta)^{\frac\delta2+\gamma_1}\left(u(\tau)\omega(\tau)\right)\right)(\xi)\right|^2\, d\xi\right)^\frac12.\label{eq:3.49}\\
			&\lesssim (t-s)^{\gamma_1}\left\|e^{(s-\tau)\Delta}(-\Delta)^{\frac\delta2+\gamma_1}\left(u(\tau)\omega(\tau)\right)\right\|_{L^2}.\label{eq:3.50}
		\end{align}
	Plugging this estimate into \eqref{eq:3.44}, 
		\begin{align}
			&\quad\left\|\int_0^s\left(e^{(t-\tau)\Delta}-e^{(s-\tau)\Delta}\right)e^{(t-\tau)\beta L_1}\operatorname{div}(u(\tau)\omega(\tau))\, d\tau\right\|_{\dot H^{-1+\delta}}\label{eq:3.51}\\
			&\lesssim (t-s)^{\gamma_1}\int_0^s\left\|e^{(s-\tau)\Delta}(-\Delta)^{\frac\delta2+\gamma_1}\left(u(\tau)\omega(\tau)\right)\right\|_{L^2}\, d\tau\label{eq:3.52}\\
			&\lesssim (t-s)^{\gamma_1}\int_0^s(s-\tau)^{-\frac1q+\frac12-\gamma_1}\left\|(-\Delta)^{\frac\delta2}\left(u(\tau)\omega(\tau)\right)\right\|_{L^q}\, d\tau\label{eq:3.53}\\
			&\lesssim (t-s)^{\gamma_1}\left\|\omega\right\|_{Y_1}\left\|\omega\right\|_{Y_2}s^{1-\frac1{r_1}-\frac1{r_2}-\frac1q+\frac12-\gamma_1}\label{eq:3.54}\\
			&\to 0 \quad\text{as $t-s\to 0$}.\label{eq:3.55}
		\end{align}
	
	A similar argument goes through for the estimate of \eqref{eq:3.37}. In fact, we have
		\begin{align}
			&\quad\left\|\int_0^se^{(s-\tau)\Delta}\left(e^{(t-\tau)\beta L_1}-e^{(s-\tau)\beta L_1}\right)\operatorname{div}(u(\tau)\omega(\tau))\, d\tau\right\|_{\dot H^{-1+\delta}}\label{eq:3.56}\\
			&\lesssim\int_0^s\left\|\left(e^{(t-s)\beta L_1}-1\right)e^{(s-\tau)\Delta}(-\Delta)^\frac\delta2(u(\tau)\omega(\tau))\right\|_{L^2}\, d\tau\label{eq:3.57}
		\end{align}
	and by the Planchrel theorem, 
		\begin{align}
			&\left\|\left(e^{(t-s)\beta L_1}-1\right)e^{(s-\tau)\Delta}(-\Delta)^\frac\delta2(u(\tau)\omega(\tau))\right\|_{L^2}\label{eq:3.58}\\
			&=\frac1{2\pi}\left(\int_{\mathbb R^2}\left|e^{(t-s)\beta\frac{i\xi_1}{|\xi|^2}}-1\right|^2\left|\mathcal F\left(e^{(s-\tau)\Delta}(-\Delta)^\frac\delta2\left(u(\tau)\omega(\tau)\right)\right)(\xi)\right|^2\, d\xi\right)^\frac12.\label{eq:3.59}
		\end{align}
	In this case, we take
		$\gamma_2\in(0, 1)$
	small enough to satisfy
		$\frac1q>\frac12+\frac{\gamma_2}2.$ 
	Then the Sobolev embedding and an estimate similar to that of \eqref{eq:3.47} yield
		\begin{align}
			&\int_0^s\left\|\left(e^{(t-s)\beta L_1}-1\right)e^{(s-\tau)\Delta}(-\Delta)^\frac\delta2(u(\tau)\omega(\tau))\right\|_{L^2}\, d\tau\label{eq:3.60}\\
			&\lesssim |\beta|^{\gamma_2}(t-s)^{\gamma_2}\int_0^s\left\|e^{(s-\tau)\Delta}(-\Delta)^\frac{\delta-\gamma_2}2(u(\tau)\omega(\tau))\right\|_{L^2}\, d\tau\label{eq:3.61}\\
			&\lesssim |\beta|^{\gamma_2}(t-s)^{\gamma_2}\int_0^s\left\|e^{(s-\tau)\Delta}(-\Delta)^\frac{\delta}2(u(\tau)\omega(\tau))\right\|_{L^\frac{2}{1+\gamma_2}}\, d\tau\label{eq:3.62}\\
			&\lesssim |\beta|^{\gamma_2}(t-s)^{\gamma_2}\int_0^s(s-\tau)^{-\frac{1}{q}+\frac{1+\gamma_2}{2}}\left\|(-\Delta)^\frac{\delta}2(u(\tau)\omega(\tau))\right\|_{L^q}\, d\tau\label{eq:3.63}\\
			&\lesssim |\beta|^{\gamma_2}(t-s)^{\gamma_2}\left\|\omega\right\|_{Y_1}\left\|\omega\right\|_{Y_2}s^{1-\frac1{r_1}-\frac1{r_2}-\frac{1}{q}+\frac{1+\gamma_2}{2}}
			\to 0\quad\text{as $t-s\to 0$}.\label{eq:3.64}
		\end{align}
	Putting these estimates altogether, we have the continuity of the solution
		$\omega(t)$
	for
		$t\in(0, \infty).$
	This completes the proof of Theorem~\ref{thm:1.1}.
\qed
\begin{remark}\label{rem:3.1}\rm
	The obtained solution 
		$\omega\in Y$
	satisfies
		\begin{align}
			& \left|\beta\right|^{-\frac13(1-\frac{2}{r_1}-\frac{2}{p_1}+\delta)}\left\|\omega\right\|_{Y_1}
			\leq 2C_1\left|\beta\right|^{-\frac\delta3}\left\|\omega_0\right\|_{\dot H^{-1+\delta}}\label{eq:3.65}
		\end{align}
	and
		\begin{align}
			& \left|\beta\right|^{-\frac13(2-\frac2{r_2}-\frac2{p_2}+\delta)}\left\|\omega\right\|_{Y_2}
			\leq 2C_1\left|\beta\right|^{-\frac13(1+\delta)}\left\|\omega_0\right\|_{\dot H^{\delta}}.\label{eq:3.66}
		\end{align}
	These estimates are invariant under the transformation given by \eqref{eq:1.10}
\end{remark}

\section{Proof of Theorem~\ref{thm:1.2}}\label{sect:4}
In this section, we always assume that the exponents 
	$\mathcal A
	=(\delta, p_1, r_1, p_2, r_2)$
satisfy the assumptions in Theorem~\ref{thm:1.1} and the condition~\eqref{eq:1.11}, and that
	$\omega$
denotes the solution of \eqref{eq:1.9} obtained in Theorem~\ref{thm:1.1}.
We begin our discussion by preparing a lemma which describes the smoothing effect in terms of the space-time norms.
\begin{lemma}\label{lem:4.1}
	For any
		$\varepsilon>0$
	and
		$m\geq 0$, 
	the solution
		$\omega$
	satisfies
		$\omega \in L^{r_1}([\varepsilon, \infty); \dot W^{-1+\delta+m, p_1})\cap L^{r_2}([\varepsilon, \infty); \dot W^{\delta+m, p_2}).$
		
	Moreover, 
	if
		$\omega_0\in \dot H^{-1+\delta}\cap \dot H^{\delta +m}$
	for some
		$m\geq0, $
	then we also have
		$\omega \in L^{r_1}([0, \infty); \dot W^{-1+\delta+m, p_1})\cap L^{r_2}([0, \infty); \dot W^{\delta+m, p_2}), $
			\begin{align*}
			&|\beta|^{-\frac13(1+\delta+m-\frac2{p_1}-\frac2{r_1})}\left\|\omega\right\|_{L^{r_1}([0, \infty); \dot W^{-1+\delta+m, p_1})}\\
			&\leq C\left(|\beta|^{-\frac\delta3}\left\|\omega_0\right\|_{\dot H^{-1+\delta}}
			+|\beta|^{-\frac{\delta+m}3}\left\|\omega_0\right\|_{\dot H^{-1+\delta+m}}
			+|\beta|^{-\frac\delta3}\left\|\omega_0\right\|_{\dot H^{-1+\delta}}|\beta|^{-\frac13(1+\delta+m)}\left\|\omega_0\right\|_{\dot H^{\delta+m}}\right)
		\end{align*}
	and
		\begin{align*}
			&|\beta|^{-\frac13(2+\delta+m-\frac2{p_2}-\frac2{r_2})}\left\|\omega\right\|_{L^{r_2}([0, \infty); \dot W^{\delta+m, p_2})}\\
			&\leq C\left(
			|\beta|^{-\frac{1+\delta}3}\left\|\omega_0\right\|_{\dot H^{\delta}}
			+|\beta|^{-\frac{1+\delta+m}3}\left\|\omega_0\right\|_{\dot H^{\delta+m}}
			+\left|\beta\right|^{-\frac{1+\delta}3}\left\|\omega_0\right\|_{\dot H^\delta}
		|\beta|^{-\frac13(\delta+m)}\left\|\omega_0\right\|_{\dot H^{-1+\delta+m}}\right), 
		\end{align*}
	where
		$C>0$
	is independent of
		$\beta$
	and
		$\omega_0.$
\end{lemma}
\proof
	Take a small number
		$\alpha>0$
	satisfying 
		\begin{align}
			\frac{1}{r_1}<\frac12-\frac1{p_1}+\frac\delta2-\frac\alpha2
			\quad\text{and}\quad
			\frac{1}{r_2}<1-\frac1{p_2}+\frac\delta2-\frac\alpha2.\label{eq:4.1}
		\end{align}
	This is possible due to the assumption \eqref{eq:1.11}. We  show that
		$\omega$
	satisfies
		\begin{align}
			\omega\in L^{r_1}([\varepsilon, \infty); \dot W^{-1+\delta+k\alpha, p_1})\cap L^{r_2}([\varepsilon, \infty); \dot W^{\delta+k\alpha, p_2})\label{eq:4.2}
		\end{align}
	for all
		$\varepsilon>0$
	and
		$k\in\mathbb N\cup\{0\}$
	by the induction of
		$k$.
	The case of 	
		$k=0$
	is trivial as
		$\omega\in Y_1\cap Y_2.$
	
	For
		$t\geq \varepsilon$,  
	we write
		\begin{align*}
			\omega(t)
			=T_\beta(t)\omega_0
			-\int_0^{\frac t2}T_\beta(t-\tau)\operatorname{div}(u(\tau)\omega(\tau))\, d\tau
			-\int_{\frac t2}^tT_\beta(t-\tau)\operatorname{div}(u(\tau)\omega(\tau))\, d\tau
		\end{align*}
	and estimate each term on the right-hand side.
	
	For the first linear term, we have 
		\begin{align*}
			\left\|T_\beta(t)\omega_0\right\|_{\dot W^{-1+\delta+m, p_1}}
			&=\left\|e^{\frac t2\Delta}T_{2\beta}( t/2)\omega_0\right\|_{\dot W^{-1+\delta+m, p_1}}
			\lesssim t^{-\frac m2}\left\|T_{2\beta}( t/2)\omega_0\right\|_{\dot W^{-1+\delta, p_1}}
		\end{align*}
	for any
		$m\geq 0.$
	Therefore, we obtain 
		\begin{align}
			\left\|T_\beta(t)\omega_0\right\|_{L^{r_1}([\varepsilon, \infty); \dot W^{-1+\delta+m, p_1})}
			&\lesssim \varepsilon^{-\frac m2}\left\|T_{2\beta}( t/2)\omega_0\right\|_{L^{r_1}([0, \infty); \dot W^{-1+\delta, p_1})}
			<\infty.\label{eq:4.5}
		\end{align}
	Similarly, we have
		\begin{align}
			\left\|T_\beta(t)\omega_0\right\|_{L^{r_2}([\varepsilon, \infty); \dot W^{\delta+m, p_2})}
			&\lesssim \varepsilon^{-\frac m2}\left\|T_{2\beta}( t/2)\omega_0\right\|_{L^{r_2}([0, \infty); \dot W^{\delta, p_2})}
			<\infty.\label{eq:4.9}
		\end{align}
	For the second term, recall that
		$\frac1q=\frac1{p_1}+\frac1{p_2}-\frac\delta2, $
	and by Proposition~\ref{prop:2.2} we have
		\begin{align}
			&\quad\left\|\int_0^{\frac t2}T_\beta(t-\tau)\operatorname{div}(u(\tau)\omega(\tau))\, d\tau\right\|_{\dot W^{-1+\delta+m, p_1}}\label{eq:4.6}\\
			&\lesssim \int_0^\frac t2 (t-\tau)^{-\frac m2-\frac 1q+\frac 1{p_1'}-1+\frac2{p_1}}\min\left\{1, |\beta|^{-1+\frac2{p_1}}(t-\tau)^{-\frac32(1-\frac2{p_1})}\right\}\left\|(-\Delta)^\frac\delta2(u(\tau)\omega(\tau))\right\|_{L^q}\, d\tau\notag\\
			&\lesssim  \varepsilon^{-\frac m2}\int_0^\frac t2 (t-\tau)^{-\frac 1{p_2}+\frac \delta2}\min\left\{1, |\beta|^{-1+\frac2{p_1}}(t-\tau)^{-\frac32(1-\frac2{p_1})}\right\}\left\|u(\tau)\right\|_{\dot W^{-1+\delta, p_1}}\left\|\omega(\tau)\right\|_{\dot W^{\delta, p_2}}\, d\tau.\notag
		\end{align}
	Thus Young's inequality (or the Hardy--Littlewood--Sobolev inequality) yields
		\begin{align}
			&\quad\left\|\int_0^{\frac t2}T_\beta(t-\tau)\operatorname{div}(u(\tau)\omega(\tau))\, d\tau\right\|_{L^{r_1}([\varepsilon, \infty); \dot W^{-1+\delta+m, p_1})}\label{eq:4.10}\\
			&\lesssim\varepsilon^{-\frac m2}\left|\beta\right|^{-\frac23(1-\frac1{r_2}-\frac1{p_2}+\frac\delta2)}\left\|\omega\right\|_{Y_1}\left\|\omega\right\|_{Y_2}<\infty.\label{eq:4.11}
		\end{align}
	The repetition of a similar argument leads to
		\begin{align}
			&\quad\left\|\int_0^{\frac t2}T_\beta(t-\tau)\operatorname{div}(u(\tau)\omega(\tau))\, d\tau\right\|_{L^{r_2}([\varepsilon, \infty); \dot W^{\delta+m, p_2})}\label{eq:4.12}\\
			&\lesssim\varepsilon^{-\frac m2}\left|\beta\right|^{-\frac23(\frac12-\frac1{r_1}-\frac1{p_1}+\frac\delta2)}\left\|\omega\right\|_{Y_1}\left\|\omega\right\|_{Y_2}<\infty.\label{eq:4.13}
		\end{align}
	
	For the estimate of the last term, we use the hypothesis of the induction. Let 
		$k\in \mathbb N\cup\{0\}$, and
	we have
		\begin{align}
			&\quad\left\|\int_{\frac t2}^tT_\beta(t-\tau)\operatorname{div}(u(\tau)\omega(\tau))\, d\tau\right\|_{\dot W^{-1+\delta+(k+1)\alpha, p_1}}\label{eq:4.14}\\
			&\lesssim \int_\frac t2^t (t-\tau)^{-1+\frac2{p_1}}\min\left\{1, |\beta|^{-1+\frac2{p_1}}(t-\tau)^{-\frac32(1-\frac2{p_1})}\right\}\left\|e^{\frac12(t-\tau)\Delta}(-\Delta)^\frac{\delta+(k+1)\alpha}2(u(\tau)\omega(\tau))\right\|_{L^{p_1'}}\, d\tau\notag\\
			&\lesssim \int^t_\frac t2 (t-\tau)^{-\frac1{p_2}+\frac\delta2-\frac\alpha2}\min\left\{1, |\beta|^{-1+\frac2{p_1}}(t-\tau)^{-\frac32(1-\frac2{p_1})}\right\}\left\|(-\Delta)^\frac{\delta+k\alpha}2(u(\tau)\omega(\tau))\right\|_{L^{q}}\, d\tau.\notag
		\end{align}
	By Proposition~\ref{prop:2.4}, we have
		\begin{align}
			\left\|(-\Delta)^\frac{\delta+k\alpha}2(u(\tau)\omega(\tau))\right\|_{L^{q}}
			\lesssim \left\|(-\Delta)^\frac{\delta+k\alpha}{2}u(\tau)\right\|_{L^{p_1}}\left\|\omega(\tau)\right\|_{\dot W^{\delta, p_2}}
			+\left\|u(\tau)\right\|_{\dot W^{\delta, p_1}}\left\|(-\Delta)^\frac{\delta+k\alpha}{2}\omega(\tau)\right\|_{L^{p_2}}.\notag
		\end{align}
	Since 
	$\frac t 2\geq \frac \varepsilon2, $
	Young's inequality leads to
		\begin{align}
			&\quad\left\|\int_{\frac t2}^tT_\beta(t-\tau)\operatorname{div}(u(\tau)\omega(\tau))\, d\tau\right\|_{L^{r_1}([\varepsilon, \infty); \dot W^{-1+\delta+(k+1)\alpha, p_1})}\label{eq:4.18}\\
			&\lesssim\left\|t^{-\frac1{p_2}+\frac\delta2-\frac\alpha2}\min\left\{1, |\beta|^{-1+\frac2{p_1}}t^{-\frac32(1-\frac2{p_1})}\right\}\right\|_{L^{r_2'}([0, \infty))}\label{eq:4.19}\\
			&\phantom{\lesssim}\cdot\left(\left\|\omega\right\|_{L^{r_1}([\frac\varepsilon2, \infty); \dot W^{-1+\delta+k\alpha, p_1})}\left\|\omega\right\|_{Y_2}
			+\left\|\omega\right\|_{L^{r_2}([\frac\varepsilon2, \infty); \dot W^{\delta+k\alpha, p_2})}\left\|\omega\right\|_{Y_1}\right).\label{eq:4.20}
		\end{align}
	The second factor of the right-hand side is finite from the hypothesis of the induction. Due to \eqref{eq:4.1}, we also have
		\begin{align}
			\left\|t^{-\frac1{p_2}+\frac\delta2-\frac\alpha2}\min\left\{1, |\beta|^{-1+\frac2{p_1}}t^{-\frac32(1-\frac2{p_1})}\right\}\right\|_{L^{r_2'}([0, \infty))}
			&\lesssim\left|\beta\right|^{-\frac23(1-\frac1{r_2}-\frac1{p_2}+\frac\delta2-\frac\alpha2)}<\infty. \notag
		\end{align}
	Similarly we have
		\begin{align}
			&\quad\left\|\int_{\frac t2}^tT_\beta(t-\tau)\operatorname{div}(u(\tau)\omega(\tau))\, d\tau\right\|_{L^{r_2}([\varepsilon, \infty); \dot W^{\delta+(k+1)\alpha, p_2})}\label{eq:4.22}\\
			&\lesssim\left|\beta\right|^{-\frac23(\frac12-\frac1{r_1}-\frac1{p_1}+\frac\delta2-\frac\alpha2)}\left(\left\|\omega\right\|_{L^{r_1}([\frac\varepsilon2, \infty); \dot W^{-1+\delta+k\alpha, p_1})}\left\|\omega\right\|_{Y_2}
			+\left\|\omega\right\|_{L^{r_2}([\frac\varepsilon2, \infty); \dot W^{\delta+k\alpha, p_2})}\left\|\omega\right\|_{Y_1}\right).\notag
		\end{align}
	The right-hand side is finite by the hypothesis of the induction. These are enough for the proof of the first half of the lemma.
	
	We next treat the case of 
		$\omega_0\in \dot H^{-1+\delta}\cap \dot H^{\delta+m}$
	for some 
		$m\geq 0$.
	Take 
		$N\in\mathbb N$
	large enough that
		$\alpha \coloneqq \frac mN$
	satisfies \eqref{eq:4.1}.
	Then by Proposition~\ref{prop:2.3}, for 
		$k\in\{0, 1, \dots, N\}$
	we have
		\begin{align}
			\left\|T_\beta(t)\omega_0\right\|_{L^{r_1}([0, \infty); \dot W^{-1+\delta+k\alpha, p_1})}
			\lesssim\left|\beta\right|^{-\frac13(\frac2{r_1}+\frac2{p_1}-1)}\left\|\omega_0\right\|_{\dot H^{-1+\delta+k\alpha}}\label{eq:4.24}
		\end{align}
	and
		\begin{align}
			\left\|T_\beta(t)\omega_0\right\|_{L^{r_2}([0, \infty); \dot W^{\delta+k\alpha, p_2})}
			\lesssim\left|\beta\right|^{-\frac13(\frac2{r_2}+\frac2{p_2}-1)}\left\|\omega_0\right\|_{\dot H^{\delta+k\alpha}}.\label{eq:4.25}
		\end{align}
	An argument similar to the above one is sufficient to show 
		\begin{align}
			&\quad\left\|\int_0^tT_\beta(t-\tau)\operatorname{div}(u(\tau)\omega(\tau))\, d\tau\right\|_{L^{r_1}([0, \infty); \dot W^{-1+\delta+(k+1)\alpha, p_1})}\label{eq:4.26}\\
			&\lesssim\left|\beta\right|^{-\frac23(1-\frac1{r_2}-\frac1{p_2}+\frac\delta2-\frac\alpha2)}\left(\left\|\omega\right\|_{L^{r_1}([0, \infty); \dot W^{-1+\delta+k\alpha, p_1})}\left\|\omega\right\|_{Y_2}
			+\left\|\omega\right\|_{L^{r_2}([0, \infty); \dot W^{\delta+k\alpha, p_2})}\left\|\omega\right\|_{Y_1}\right)\notag
		\end{align}
	and
		\begin{align}
			&\quad\left\|\int_0^tT_\beta(t-\tau)\operatorname{div}(u(\tau)\omega(\tau))\, d\tau\right\|_{L^{r_2}([0, \infty); \dot W^{\delta+(k+1)\alpha, p_2})}\label{eq:4.28}\\
			&\lesssim\left|\beta\right|^{-\frac23(\frac12-\frac1{r_1}-\frac1{p_1}+\frac\delta2-\frac\alpha2)}\left(\left\|\omega\right\|_{L^{r_1}([0, \infty); \dot W^{-1+\delta+k\alpha, p_1})}\left\|\omega\right\|_{Y_2}
			+\left\|\omega\right\|_{L^{r_2}([0, \infty); \dot W^{\delta+k\alpha, p_2})}\left\|\omega\right\|_{Y_1}\right)\notag
		\end{align}
	for 
		$k\in\{0, 1, \dots, N-1\}.$
	Now we define constants
	$a_k=a_k(\beta, \omega_0)$
and
	$b_k=b_k(\beta, \omega_0)$
such that they satisfy
	\begin{align*}
		&|\beta|^{-\frac13(1+\delta+k\alpha-\frac2{p_1}-\frac2{r_1})}\left\|\omega\right\|_{L^{r_1}([0, \infty); \dot W^{-1+\delta+k\alpha, p_1})}
		\lesssim a_k
	\end{align*}
and
	\begin{align*}
		&|\beta|^{-\frac13(2+\delta+k\alpha-\frac2{p_2}-\frac2{r_2})}\left\|\omega\right\|_{L^{r_2}([0, \infty); \dot W^{\delta+k\alpha, p_2})}
		\lesssim b_k
	\end{align*}
for 
	$k\in\{0, 1, \dots, N\}, $
where the implicit constants are independent of
	$\beta$
and
	$\omega_0$.
By Remark~\ref{rem:3.1}, it is possible to set
	\begin{align*}
		a_0
		=(2C_0)^{-1}|\beta|^{-\frac\delta3}\left\|\omega_0\right\|_{\dot H^{-1+\delta}}
		\quad\text{and}\quad
		b_0
		=(2C_0)^{-1}|\beta|^{-\frac{1+\delta}3}\left\|\omega_0\right\|_{\dot H^{\delta}}, 
	\end{align*}
where
	$C_0>0$
is the constant appearing in Theorem~\ref{thm:1.1}.
We next derive the recurrence relations for 
	$(a_k)_k$
and
	$(b_k)_k.$
From the estimates stated avobe, we have
	\begin{align*}
		&|\beta|^{-\frac13(1+\delta+(k+1)\alpha-\frac2{p_1}-\frac2{r_1})}\left\|\omega\right\|_{L^{r_1}([0, \infty); \dot W^{-1+\delta+(k+1)\alpha, p_1})}\\
		&\lesssim |\beta|^{-\frac13(\delta+(k+1)\alpha)}\left\|\omega_0\right\|_{\dot H^{-1+\delta+(k+1)\alpha}}\\
		&+|\beta|^{-\frac13(3+2\delta+k\alpha-\frac2{p_1}-\frac2{r_1}-\frac2{r_2}-\frac2{p_2})}
		\left(\left\|\omega\right\|_{L^{r_1}([0, \infty); \dot W^{-1+\delta+k\alpha, p_1})}
		\left\|\omega\right\|_{Y_2}
		+\left\|\omega\right\|_{L^{r_2}([0, \infty); \dot W^{\delta+k\alpha, p_2})}
		\left\|\omega\right\|_{Y_1}\right)
	\end{align*}
for
	$k\in\{0, 1, \dots, N-1\}.$
For the first term, the interpolation inequality gives
	\begin{align*}
		&|\beta|^{-\frac13(\delta+(k+1)\alpha)}\left\|\omega_0\right\|_{\dot H^{-1+\delta+(k+1)\alpha}}\\
		&\leq 
		\left(|\beta|^{-\frac\delta3}\left\|\omega_0\right\|_{\dot H^{-1+\delta}}\right)^\frac{m-(k+1)\alpha}{m}
		\left(|\beta|^{-\frac13(\delta+m)}\left\|\omega_0\right\|_{\dot H^{-1+\delta+m}}\right)^\frac{(k+1)\alpha}{m}\\
		&\leq |\beta|^{-\frac\delta3}\left\|\omega_0\right\|_{\dot H^{-1+\delta}}
		+|\beta|^{-\frac13(\delta+m)}\left\|\omega_0\right\|_{\dot H^{-1+\delta+m}}.
	\end{align*}
For the second term, using the estimate in Remark~\ref{rem:3.1} and the definition of
	$(a_k)_k$
and
	$(b_k)_k$, 
we have
	\begin{align*}
		&|\beta|^{-\frac13(3+2\delta+k\alpha-\frac2{p_1}-\frac2{r_1}-\frac2{r_2}-\frac2{p_2})}
		\left(\left\|\omega\right\|_{L^{r_1}([0, \infty); \dot W^{-1+\delta+k\alpha, p_1})}
		\left\|\omega\right\|_{Y_2}
		+\left\|\omega\right\|_{L^{r_2}([0, \infty); \dot W^{\delta+k\alpha, p_2})}
		\left\|\omega\right\|_{Y_1}\right)\\
		&\leq |\beta|^{-\frac13(1+\delta+k\alpha-\frac2{p_1}-\frac2{r_1})}
		\left\|\omega\right\|_{L^{r_1}([0, \infty); \dot W^{-1+\delta+k\alpha, p_1})}
		|\beta|^{-\frac13(2+\delta-\frac2{r_2}-\frac2{p_2})}\left\|\omega\right\|_{Y_2}\\
		&+|\beta|^{-\frac13(2+\delta+k\alpha-\frac2{r_2}-\frac2{p_2})}
		\left\|\omega\right\|_{L^{r_2}([0, \infty); \dot W^{\delta+k\alpha, p_2})}
		|\beta|^{-\frac13(1+\delta-\frac2{p_1}-\frac2{r_1})}\left\|\omega\right\|_{Y_1}\\
		&\lesssim b_0a_k+a_0b_k.
	\end{align*}
Thus, we see that it is possible to take the sequence
	$(a_k)_k$
such that it satisfies
	\begin{align*}
		a_{k+1}
		=(2C_0)^{-1}|\beta|^{-\frac\delta3}\left\|\omega_0\right\|_{\dot H^{-1+\delta}}
		+|\beta|^{-\frac13(\delta+m)}\left\|\omega_0\right\|_{\dot H^{-1+\delta+m}}
		+b_0a_k+a_0b_k
	\end{align*}
for
	$k\in\{0, 1, \dots, N-1\}.$
We can derive the recurrence relation for
	$(b_k)_k$
in a similar way. In fact, we have
	\begin{align*}
		&|\beta|^{-\frac13({2+\delta+(k+1)\alpha-\frac2{p_2}-\frac2{r_2})}}\left\|\omega\right\|_{L^{r_2}([0, \infty); \dot W^{\delta+(k+1)\alpha, p_2})}\\
		&\lesssim |\beta|^{-\frac13({1+\delta+(k+1)\alpha)}}
		\left\|\omega_0\right\|_{\dot H^{\delta+(k+1)\alpha}}\\
		&+|\beta|^{-\frac13({3+2\delta+k\alpha-\frac2{p_1}-\frac2{r_1}-\frac2{p_2}-\frac2{r_2})}}
		\left(\left\|\omega\right\|_{L^{r_1}([0, \infty); \dot W^{-1+\delta+k\alpha, p_1})}
		\left\|\omega\right\|_{Y_2}
		+\left\|\omega\right\|_{L^{r_2}([0, \infty); \dot W^{\delta+k\alpha, p_2})}
		\left\|\omega\right\|_{Y_1}\right).
	\end{align*}
For the first term, we have
	\begin{align*}
		&|\beta|^{-\frac13({1+\delta+(k+1)\alpha)}}
		\left\|\omega_0\right\|_{\dot H^{\delta+(k+1)\alpha}}\\
		&\leq \left(|\beta|^{-\frac13(1+\delta)}\left\|\omega_0\right\|_{\dot H^\delta}\right)^\frac{m-(k+1)\alpha}{m}
		\left(|\beta|^{-\frac13(1+\delta+m)}\left\|\omega_0\right\|_{\dot H^{\delta+m}}\right)^\frac{(k+1)\alpha}{m}\\
		&\leq |\beta|^{-\frac13(1+\delta)}\left\|\omega_0\right\|_{\dot H^\delta}
		+|\beta|^{-\frac13(1+\delta+m)}\left\|\omega_0\right\|_{\dot H^{\delta+m}}.
	\end{align*}
The estimate for the second term is the same as before, and we have
	\begin{align*}
		&|\beta|^{-\frac13({3+2\delta+k\alpha-\frac2{p_1}-\frac2{r_1}-\frac2{p_2}-\frac2{r_2})}}
		\left(\left\|\omega\right\|_{L^{r_1}([0, \infty); \dot W^{-1+\delta+k\alpha, p_1})}
		\left\|\omega\right\|_{Y_2}
		+\left\|\omega\right\|_{L^{r_2}([0, \infty); \dot W^{\delta+k\alpha, p_2})}
		\left\|\omega\right\|_{Y_1}\right)\\
		&\lesssim b_0a_k+a_0b_k.
	\end{align*}
Thus, we may assume that
	$(b_k)_k$
satisfies
	\begin{align*}
		b_{k+1}
		=(2C_0)^{-1}|\beta|^{-\frac13(1+\delta)}\left\|\omega_0\right\|_{\dot H^\delta}
		+|\beta|^{-\frac13(1+\delta+m)}\left\|\omega_0\right\|_{\dot H^{\delta+m}}
		+b_0a_k+a_0b_k
	\end{align*}
for
	$k\in\{0, 1, \dots, N-1\}.$
Now, let us set
	\begin{align*}
		x
		\coloneqq (2C_0)^{-1}|\beta|^{-\frac\delta3}\left\|\omega_0\right\|_{\dot H^{-1+\delta}}
		+|\beta|^{-\frac13(\delta+m)}\left\|\omega_0\right\|_{\dot H^{-1+\delta+m}}
	\end{align*}
and
	\begin{align*}
		y
		\coloneqq (2C_0)^{-1}|\beta|^{-\frac13(1+\delta)}\left\|\omega_0\right\|_{\dot H^\delta}
		+|\beta|^{-\frac13(1+\delta+m)}\left\|\omega_0\right\|_{\dot H^{\delta+m}}.
	\end{align*}
Then the recurrence relation of
	$(a_k)_k$
and
	$(b_k)_k$
is given by
	\begin{align*}
		\begin{cases}
			a_{k+1}=b_0a_k+a_0b_k+x, \\
			b_{k+1}=b_0a_k+a_0b_k+y
		\end{cases}
	\end{align*}
for
	$k\in\{0, 1, \dots, N-1\}.$
Let us also define
	$z$
and
	$w$
by the equations
	\begin{align*}
		\begin{cases}
			z=b_0z+a_0w+x, \\
			w=b_0z+a_0w+y.
		\end{cases}
	\end{align*}
Since this is equivalent to
	\begin{align*}
		\begin{cases}
			(1-b_0)z-a_0w=x, \\
			-b_0z+(1-a_0)w=y, 
		\end{cases}
	\end{align*}
and the coefficient matrix of the left-hand side is regular by the condition
	$a_0+b_0\leq\frac12, $
we can explicitly write
	\begin{align*}
		\begin{pmatrix}
			z\\
			w
		\end{pmatrix}
		=\frac1{(1-a_0)(1-b_0)-a_0b_0}
		\begin{pmatrix}
			1-a_0 & a_0\\
			b_0 & 1-b_0
		\end{pmatrix}
		\begin{pmatrix}
			x\\
			y
		\end{pmatrix}.
	\end{align*}
Solving the recurrence relations of 
	$(a_k)_k$
and
	$(b_k)_k, $
we obtain
	\begin{align*}
		a_N
		=z+(a_0+b_0)^{N-1}(b_0(a_0-z)+a_0(b_0-w))
	\end{align*}
and
	\begin{align*}
		b_N
		=w+(a_0+b_0)^{N-1}(b_0(a_0-z)+a_0(b_0-w)).
	\end{align*}
Since the conditions
	$a_0\leq x$
and
	$b_0\leq y$
ensure that
	\begin{align*}
		a_0-z
		&=a_0-\frac{(1-a_0)x+a_0y}{1-a_0-b_0}
		=\frac{(1-a_0)(a_0-x)}{1-a_0-b_0}-\frac{a_0(b_0+y)}{1-a_0-b_0}
		\leq0
	\end{align*}
and
	\begin{align*}
		b_0-w
		&=b_0-\frac{b_0x+(1-b_0)y}{1-a_0-b_0}
		=\frac{(1-b_0)(b_0-y)}{1-a_0-b_0}-\frac{b_0(a_0+x)}{1-a_0-b_0}
		\leq0, 
	\end{align*}
we obtain
	\begin{align*}
		a_N\leq z
		\quad\text{and}\quad
		b_N\leq w.
	\end{align*}
Now, we have
	\begin{align*}
		&z
		=\frac{(1-a_0)x+a_0y}{1-a_0-b_0}
		\lesssim x+a_0y\\
		&=|\beta|^{-\frac\delta3}\left\|\omega_0\right\|_{\dot H^{-1+\delta}}
		+|\beta|^{-\frac13(\delta+m)}\left\|\omega_0\right\|_{\dot H^{-1+\delta+m}}\\
		&\quad+|\beta|^{-\frac\delta3}\left\|\omega_0\right\|_{\dot H^{-1+\delta}}\left(|\beta|^{-\frac13(1+\delta)}\left\|\omega_0\right\|_{\dot H^\delta}
		+|\beta|^{-\frac13(1+\delta+m)}\left\|\omega_0\right\|_{\dot H^{\delta+m}}\right)\\
		&\lesssim |\beta|^{-\frac\delta3}\left\|\omega_0\right\|_{\dot H^{-1+\delta}}
		+|\beta|^{-\frac13(\delta+m)}\left\|\omega_0\right\|_{\dot H^{-1+\delta+m}}
		+|\beta|^{-\frac\delta3}\left\|\omega_0\right\|_{\dot H^{-1+\delta}}|\beta|^{-\frac13(1+\delta+m)}\left\|\omega_0\right\|_{\dot H^{\delta+m}}.
	\end{align*}
Similarly, we obtain
	\begin{align*}
		&w
		=\frac{b_0x+(1-b_0)y}{1-a_0-b_0}
		\lesssim b_0x+y\\
		&=\left|\beta\right|^{-\frac{1+\delta}3}\left\|\omega_0\right\|_{\dot H^\delta}
		\left(|\beta|^{-\frac\delta3}\left\|\omega_0\right\|_{\dot H^{-1+\delta}}
		+|\beta|^{-\frac13(\delta+m)}\left\|\omega_0\right\|_{\dot H^{-1+\delta+m}}\right)\\
		&\quad +|\beta|^{-\frac13(1+\delta)}\left\|\omega_0\right\|_{\dot H^\delta}
		+|\beta|^{-\frac13(1+\delta+m)}\left\|\omega_0\right\|_{\dot H^{\delta+m}}\\
		&\lesssim |\beta|^{-\frac13(1+\delta)}\left\|\omega_0\right\|_{\dot H^\delta}
		+|\beta|^{-\frac13(1+\delta+m)}\left\|\omega_0\right\|_{\dot H^{\delta+m}}
		+\left|\beta\right|^{-\frac{1+\delta}3}\left\|\omega_0\right\|_{\dot H^\delta}
		|\beta|^{-\frac13(\delta+m)}\left\|\omega_0\right\|_{\dot H^{-1+\delta+m}}.
	\end{align*}
As a result, we obtain the bounds
		\begin{align*}
			&|\beta|^{-\frac13(1+\delta+m-\frac2{p_1}-\frac2{r_1})}\left\|\omega\right\|_{L^{r_1}([0, \infty); \dot W^{-1+\delta+m, p_1})}\\
			&\lesssim |\beta|^{-\frac\delta3}\left\|\omega_0\right\|_{\dot H^{-1+\delta}}
		+|\beta|^{-\frac13(\delta+m)}\left\|\omega_0\right\|_{\dot H^{-1+\delta+m}}
		+|\beta|^{-\frac\delta3}\left\|\omega_0\right\|_{\dot H^{-1+\delta}}|\beta|^{-\frac13(1+\delta+m)}\left\|\omega_0\right\|_{\dot H^{\delta+m}}
		\end{align*}
	and
		\begin{align*}
			&|\beta|^{-\frac13(2+\delta+m-\frac2{p_2}-\frac2{r_2})}\left\|\omega\right\|_{L^{r_2}([0, \infty); \dot W^{\delta+m, p_2})}\\
			&\lesssim |\beta|^{-\frac13(1+\delta)}\left\|\omega_0\right\|_{\dot H^\delta}
		+|\beta|^{-\frac13(1+\delta+m)}\left\|\omega_0\right\|_{\dot H^{\delta+m}}
		+\left|\beta\right|^{-\frac{1+\delta}3}\left\|\omega_0\right\|_{\dot H^\delta}
		|\beta|^{-\frac13(\delta+m)}\left\|\omega_0\right\|_{\dot H^{-1+\delta+m}}.
		\end{align*}
	This completes the proof.
\qed

We divide the proof of Theorem~\ref{thm:1.2} into several propositions stated below.

\begin{proposition}\label{prop:4.2}
	The solution
		$\omega$
	satisfies
		$\omega\in C((0, \infty); \dot H^{-1+\delta+m})$
	for any
		$m\geq 0.$
\end{proposition}
\proof
	We first prove that
		$\omega(t)\in \dot H^{-1+\delta +m}$
	for all
		$t>0$
	and
		$m\geq 0$.
	This is done by combining the estimates obtained in Lemma~\ref{lem:4.1} and the idea which already appeared in the proof of Theorem~\ref{thm:1.1}. Let
		$\varepsilon>0$
	and assume that
		$t\geq\varepsilon.$
	As in the proof of Lemma~\ref{lem:4.1}, we write
		\begin{align*}
			\omega(t)
			=T_\beta(t)\omega_0
			-\int_0^{\frac t2}T_\beta(t-\tau)\operatorname{div}(u(\tau)\omega(\tau))\, d\tau
			-\int_{\frac t2}^tT_\beta(t-\tau)\operatorname{div}(u(\tau)\omega(\tau))\, d\tau
		\end{align*}
	and show that each term belongs to
		$\dot H^{-1+\delta+m}.$
	The treatment of the first and second terms are relatively easy. In fact, for the first term we see that
		\begin{align}
			\left\|T_\beta(t)\omega_0\right\|_{\dot H^{-1+\delta +m}}
			=\left\|e^{t\Delta}\omega_0\right\|_{\dot H^{-1+\delta +m}}
			\lesssim t^{-\frac m2}\left\|\omega_0\right\|_{\dot H^{-1+\delta}}<\infty, \label{eq:4.31}
		\end{align}
	and for the second term we have
		\begin{align}
			&\left\|\int_0^{\frac t2}T_\beta(t-\tau)\operatorname{div}(u(\tau)\omega(\tau))\, d\tau\right\|_{\dot H^{-1+\delta+m}}\label{eq:4.32}\\
			&\lesssim \int_0^{\frac t2}\left\|e^{(t-\tau)\Delta}(-\Delta)^\frac{\delta+m}{2}(u(\tau)\omega(\tau))\right\|_{L^2}\, d\tau\label{eq:4.33}\\
			&\lesssim \int_0^{\frac t2}(t-\tau)^{-\frac m2-\frac1q+\frac1{2}}\left\|(-\Delta)^\frac{\delta}{2}(u(\tau)\omega(\tau))\right\|_{L^q}\, d\tau\label{eq:4.34}\\
			&\lesssim \varepsilon^{-\frac m2}t^{1-\frac1{r_1}-\frac{1}{r_2}-\frac1q+\frac12}\left\|\omega\right\|_{Y_1}\left\|\omega\right\|_{Y_2}<\infty.\label{eq:4.35}
		\end{align}
	For the last term, by Lemma~\ref{lem:4.1} we have
		\begin{align*}
			&\left\|\int^t_{\frac t2}T_\beta(t-\tau)\operatorname{div}(u(\tau)\omega(\tau))\, d\tau\right\|_{\dot H^{-1+\delta+m}}
			\\
			&\lesssim \int^t_{\frac \varepsilon2}(t-\tau)^{-\frac1q+\frac1{2}}\left\|(-\Delta)^\frac{\delta+m}{2}(u(\tau)\omega(\tau))\right\|_{L^q}\, d\tau
			\\
			&\lesssim \int^t_{\frac \varepsilon2}(t-\tau)^{-\frac1q+\frac1{2}}\left(\left\|(-\Delta)^\frac{\delta+m}{2}u(\tau)\right\|_{L^{p_1}}\left\|\omega(\tau)\right\|_{\dot W^{\delta, p_2}}
			+\left\|u(\tau)\right\|_{\dot W^{\delta, p_1}}\left\|(-\Delta)^\frac{\delta+m}{2}\omega(\tau)\right\|_{L^{p_2}}\right)\, d\tau
			\\
			&\lesssim t^{1-\frac1{r_1}-\frac1{r_2}-\frac1q+\frac12}\left(\left\|\omega\right\|_{L^{r_1}([\frac\varepsilon2, \infty); \dot W^{-1+\delta+m, p_1})}\left\|\omega\right\|_{Y_2}
			+\left\|\omega\right\|_{L^{r_2}([\frac\varepsilon2, \infty); \dot W^{\delta+m, p_2})}\left\|\omega\right\|_{Y_1}\right)<\infty.
		\end{align*}
	These estimates ensure that
		$\omega(t)\in \dot H^{-1+\delta+m}$
	for any
		$t>0$
	and
		$m\geq 0$.
	
	Let us now show the continuity of
		$\omega(t)\in \dot H^{-1+\delta+m}$
	for
		$t>0.$
	The strategy is almost the same as in the proof of Theorem~\ref{thm:1.1}. Let
		$\varepsilon_0, M_0, s, t$
	satisfy
		$0<\varepsilon_0<\frac s2<\frac t2<s<t<M_0.$
	We write
		\begin{align}
				&\quad\omega(t)-\omega(s)
				=T_\beta(t)\omega_0-T_\beta(s)\omega_0\quad
				\label{eq:4.40}\\
				&-\int_0^\frac{t}{2}e^{(t-\tau)\Delta}e^{(t-\tau)\beta L_1}\operatorname{div}(u(\tau)\omega(\tau))\, d\tau
				+\int_0^\frac{s}{2}e^{(t-\tau)\Delta}e^{(t-\tau)\beta L_1}\operatorname{div}(u(\tau)\omega(\tau))\, d\tau
				\label{eq:4.41}\\
				&-\int_0^\frac{s}{2}e^{(t-\tau)\Delta}e^{(t-\tau)\beta L_1}\operatorname{div}(u(\tau)\omega(\tau))\, d\tau
				+\int_0^\frac{s}{2}e^{(s-\tau)\Delta}e^{(t-\tau)\beta L_1}\operatorname{div}(u(\tau)\omega(\tau))\, d\tau
				\label{eq:4.42}\\
				&-\int_0^\frac{s}{2}e^{(s-\tau)\Delta}e^{(t-\tau)\beta L_1}\operatorname{div}(u(\tau)\omega(\tau))\, d\tau
				+\int_0^\frac{s}{2}e^{(s-\tau)\Delta}e^{(s-\tau)\beta L_1}\operatorname{div}(u(\tau)\omega(\tau))\, d\tau\label{eq:4.43}\\
				&-\int^t_\frac{t}{2}e^{(t-\tau)\Delta}e^{(t-\tau)\beta L_1}\operatorname{div}(u(\tau)\omega(\tau))\, d\tau
				+\int^s_\frac{t}{2}e^{(t-\tau)\Delta}e^{(t-\tau)\beta L_1}\operatorname{div}(u(\tau)\omega(\tau))\, d\tau\quad
				\label{eq:4.44}\\
				&-\int^s_\frac{t}{2}e^{(t-\tau)\Delta}e^{(t-\tau)\beta L_1}\operatorname{div}(u(\tau)\omega(\tau))\, d\tau
				+\int^s_\frac{s}{2}e^{(t-\tau)\Delta}e^{(t-\tau)\beta L_1}\operatorname{div}(u(\tau)\omega(\tau))\, d\tau\quad
				\label{eq:4.45}\\
				&-\int^s_\frac{s}{2}e^{(t-\tau)\Delta}e^{(t-\tau)\beta L_1}\operatorname{div}(u(\tau)\omega(\tau))\, d\tau
				+\int^s_\frac{s}{2}e^{(s-\tau)\Delta}e^{(t-\tau)\beta L_1}\operatorname{div}(u(\tau)\omega(\tau))\, d\tau\quad
				\label{eq:4.46}\\
				&-\int^s_\frac{s}{2}e^{(s-\tau)\Delta}e^{(t-\tau)\beta L_1}\operatorname{div}(u(\tau)\omega(\tau))\, d\tau
				+\int^s_\frac{s}{2}e^{(s-\tau)\Delta}e^{(s-\tau)\beta L_1}\operatorname{div}(u(\tau)\omega(\tau))\, d\tau\quad
				\label{eq:4.47}
			\end{align}
		and show that each line tends to 0 as
			$t-s\to 0.$
		For \eqref{eq:4.40}, we have
			\begin{align}
				&\quad\left\|T_\beta(t)\omega_0-T_\beta(s)\omega_0\right\|_{\dot H^{-1+\delta+m}}\label{eq:4.48}\\
				&\leq \left\|e^{t\Delta}e^{t\beta L_1}\omega_0-e^{s\Delta}e^{t\beta L_1}\omega_0\right\|_{\dot H^{-1+\delta+m}}
				+\left\|e^{s\Delta}e^{t\beta L_1}\omega_0-e^{s\Delta}e^{s\beta L_1}\omega_0\right\|_{\dot H^{-1+\delta+m}}\label{eq:4.49}\\
				&\lesssim s^{-\frac m2}\left\|e^{(t-s)\Delta}\omega_0-\omega_0\right\|_{\dot H^{-1+\delta}}
				+s^{-\frac m2}\left\|e^{(t-s)\beta L_1}\omega_0-\omega_0\right\|_{\dot H^{-1+\delta}}\label{eq:4.50}\\
				&\to 0 \quad\text{as $t-s\to0.$}\label{eq:4.51}
			\end{align}
		For \eqref{eq:4.41}, as in the estimate of \eqref{eq:3.35}, we have
			\begin{align}
				&\quad\left\|\int_\frac{s}{2}^\frac{t}{2}e^{(t-\tau)\Delta}e^{(t-\tau)\beta L_1}\operatorname{div}(u(\tau)\omega(\tau))\, d\tau\right\|_{\dot H^{-1+\delta+m}}\label{eq:4.52}\\
				&\lesssim \int_\frac{s}{2}^\frac{t}{2}\left\|e^{(t-\tau)\Delta}(-\Delta)^\frac{\delta+m}{2}(u(\tau)\omega(\tau))\right\|_{L^2}\, d\tau\label{eq:4.53}\\
				&\lesssim \int_\frac{s}{2}^\frac{t}{2}(t-\tau)^{-\frac m2-\frac1q+\frac12}\left\|(-\Delta)^\frac{\delta}{2}(u(\tau)\omega(\tau))\right\|_{L^q}\, d\tau\label{eq:4.54}\\
				&\lesssim {\varepsilon_0}^{-\frac{m}{2}}(t-s)^{1-\frac1{r_1}-\frac1{r_2}-\frac1q+\frac12}\left\|\omega\right\|_{Y_1}\left\|\omega\right\|_{Y_2}
				\to 0\quad\text{as $t-s\to 0.$}\label{eq:4.55}
			\end{align}
		For \eqref{eq:4.42}, we have
			\begin{align}
				&\quad\left\|\int_0^\frac{s}{2}\left(e^{(t-s)\Delta}-1\right)e^{(s-\tau)\Delta}e^{(t-\tau)\beta L_1}\operatorname{div}(u(\tau)\omega(\tau))\, d\tau\right\|_{\dot H^{-1+\delta+m}}\label{eq:4.56}\\
				&\lesssim \int_0^\frac{s}{2}\left\|\left(e^{(t-s)\Delta}-1\right)e^{(s-\tau)\Delta}(-\Delta)^\frac{\delta+m}{2}(u(\tau)\omega(\tau))\right\|_{L^2}\, d\tau\label{eq:4.57}
			\end{align}
			The repetition of the argument in the proof of Theorem~\ref{thm:1.1} yields
				\begin{align}
					&\left\|\left(e^{(t-s)\Delta}-1\right)e^{(s-\tau)\Delta}(-\Delta)^\frac{\delta+m}{2}(u(\tau)\omega(\tau))\right\|_{L^2}\label{eq:4.58}\\
					&\lesssim (t-s)^{\gamma_1}\left\|e^{(s-\tau)\Delta}(-\Delta)^{\frac{\delta+m}{2}+\gamma_1}(u(\tau)\omega(\tau))\right\|_{L^2}\label{eq:4.59}
				\end{align}
			for 
				$\gamma_1\in(0, 1)$
			with
				$\frac1q-\frac12+\gamma_1<1-\frac1{r_1}-\frac1{r_2}.$
			Then we have
				\begin{align}
					&\quad\left\|\int_0^\frac{s}{2}\left(e^{(t-s)\Delta}-1\right)e^{(s-\tau)\Delta}e^{(t-\tau)\beta L_1}\operatorname{div}(u(\tau)\omega(\tau))\, d\tau\right\|_{\dot H^{-1+\delta+m}}\label{eq:4.60}\\
					&\lesssim (t-s)^{\gamma_1}\int_{0}^{\frac s2}\left\|e^{(s-\tau)\Delta}(-\Delta)^{\frac{\delta+m}{2}+\gamma_1}(u(\tau)\omega(\tau))\right\|_{L^2}\, d\tau
					\label{eq:4.61}\\
					&\lesssim (t-s)^{\gamma_1}\int_{0}^{\frac s2}(s-\tau)^{-\frac m2-\frac1q+\frac12-\gamma_1}\left\|(-\Delta)^{\frac{\delta}{2}}(u(\tau)\omega(\tau))\right\|_{L^q}\, d\tau\label{eq:4.62}\\
					&\lesssim (t-s)^{\gamma _1}{\varepsilon_0}^{-\frac m2}s^{1-\frac1{r_1}-\frac1{r_2}-\frac1q+\frac12-\gamma_1}\left\|\omega\right\|_{Y_1}\left\|\omega\right\|_{Y_2}\to 0 \quad\text{as $t-s\to0.$}
					\label{eq:4.63}
				\end{align}
			For \eqref{eq:4.43}, we similarly have
				\begin{align}
					&\quad\left\|\int_0^\frac{s}{2}\left(e^{(t-s)\beta L_1}-1\right)e^{(s-\tau)\Delta}e^{(t-\tau)\beta L_1}\operatorname{div}(u(\tau)\omega(\tau))\, d\tau\right\|_{\dot H^{-1+\delta+m}}\label{eq:4.64}\\
					&\lesssim |\beta|^{\gamma_2}(t-s)^{\gamma_2}\int_{0}^{\frac s2}\left\|e^{(s-\tau)\Delta}(-\Delta)^{\frac{\delta+m-\gamma_2}{2}}(u(\tau)\omega(\tau))\right\|_{L^2}\, d\tau
					\label{eq:4.65}\\
					&\lesssim |\beta|^{\gamma_2}(t-s)^{\gamma_2}\int_{0}^{\frac s2}(s-\tau)^{-\frac m2-\frac1q+\frac{1+\gamma_2}2}\left\|(-\Delta)^{\frac{\delta}{2}}(u(\tau)\omega(\tau))\right\|_{L^q}\, d\tau\label{eq:4.66}\\
					&\lesssim |\beta|^{\gamma_2}(t-s)^{\gamma _2}{\varepsilon_0}^{-\frac m2}s^{1-\frac1{r_1}-\frac1{r_2}-\frac1q+\frac{1+\gamma_2}{2}}\left\|\omega\right\|_{Y_1}\left\|\omega\right\|_{Y_2}\to 0 \quad\text{as $t-s\to0$}
					\label{eq:4.67}
				\end{align}
			for
				$\gamma_2\in(0, 1)$
			with
				$\frac1q>\frac12+\frac{\gamma_2}2.$
			
			For \eqref{eq:4.44}, we have
				\begin{align*}
					&\quad\left\|\int^t_se^{(t-\tau)\Delta}e^{(t-\tau)\beta L_1}\operatorname{div}(u(\tau)\omega(\tau))\, d\tau\right\|_{\dot H^{-1+\delta+m}}
					\\
					&\lesssim\int^t_s\left\|e^{(t-\tau)\Delta}(-\Delta)^{\frac{\delta+m}2}(u(\tau)\omega(\tau))\right\|_{L^{2}}\, d\tau
					\\
					&\lesssim\int^t_s(t-\tau)^{-\frac1q+\frac12}\left\|(-\Delta)^{\frac{\delta+m}2}(u(\tau)\omega(\tau))\right\|_{L^{q}}\, d\tau
					\\
					&\lesssim (t-s)^{1-\frac1{r_1}-\frac1{r_2}-\frac1q+\frac12}\left(\left\|\omega\right\|_{L^{r_1}([\varepsilon_0, \infty); \dot W^{-1+\delta+m, p_1})}\left\|\omega\right\|_{Y_2}
					+\left\|\omega\right\|_{L^{r_2}([\varepsilon_0, \infty); \dot W^{\delta+m, p_2})}\left\|\omega\right\|_{Y_1}\right)
					\\
					&\to 0 \quad\text{as $t-s\to0$}.
				\end{align*}
			The estimate for \eqref{eq:4.45} can be treated in a similar way.
			
		For \eqref{eq:4.46}, we have
			\begin{align*}
				&\quad\left\|\int^s_\frac{s}{2}\left(e^{(t-s)\Delta}-1\right)e^{(s-\tau)\Delta}e^{(t-\tau)\beta L_1}\operatorname{div}(u(\tau)\omega(\tau))\, d\tau\right\|_{\dot H^{-1+\delta+m}}
				\\
				&\lesssim (t-s)^{\gamma _1}s^{1-\frac1{r_1}-\frac1{r_2}-\frac1q+\frac12-\gamma_1}\left(\left\|\omega\right\|_{L^{r_1}([\varepsilon_0, \infty); \dot W^{-1+\delta+m, p_1})}\left\|\omega\right\|_{Y_2}
					+\left\|\omega\right\|_{L^{r_2}([\varepsilon_0, \infty); \dot W^{\delta+m, p_2})}\left\|\omega\right\|_{Y_1}\right)
					\\
					&\to 0 \quad\text{as $t-s\to0$}.
			\end{align*}
		
		Finally for \eqref{eq:4.47}, we have
			\begin{align*}
				&\quad\left\|\int^s_\frac{s}{2}\left(e^{(t-s)\beta L_1}-1\right)e^{(s-\tau)\Delta}e^{(s-\tau)\beta L_1}\operatorname{div}(u(\tau)\omega(\tau))\, d\tau\right\|_{\dot H^{-1+\delta+m}}
				\\
				&\lesssim |\beta|^{\gamma_2}(t-s)^{\gamma _2}s^{1-\frac1{r_1}-\frac1{r_2}-\frac1q+\frac12+\frac{\gamma_2}{2}}\left(\left\|\omega\right\|_{L^{r_1}([\varepsilon_0, \infty); \dot W^{-1+\delta+m, p_1})}\left\|\omega\right\|_{Y_2}
					+\left\|\omega\right\|_{L^{r_2}([\varepsilon_0, \infty); \dot W^{\delta+m, p_2})}\left\|\omega\right\|_{Y_1}\right)
					\\
					&\to 0 \quad\text{as $t-s\to0$}.
			\end{align*}
		This completes the proof of the continuity.
\qed

\begin{corollary}\label{cor:4.3}
	If 
		$\omega_0\in \dot H^{-1+\delta}\cap \dot H^{\delta+m}$
	for some
		$m\geq 0$, 
	Then 
		$\omega\in C([0, \infty); \dot H^{-1+\delta+m}).$
\end{corollary}
\proof
	We only need to prove the continuity of
		$\omega(t)$
	at
		$t=0.$
	We have
		\begin{align}
		&\quad\left\|\omega(t)-\omega_0\right\|_{\dot H^{-1+\delta+m}}\label{eq:4.85}\\
		&\leq\left\|T_\beta(t)\omega_0-\omega_0\right\|_{\dot H^{-1+\delta+m}}
		+\int_0^t\left\|T_\beta(t-\tau)\operatorname{div}(u(\tau)\omega(\tau))\right\|_{\dot H^{-1+\delta+m}}\, d\tau.\label{eq:4.86}
	\end{align}
	For the first term, since
		$\omega_0\in \dot H^{-1+\delta+m}, $ 
	it is easy to deduce
		\begin{align}
			&\quad\left\|T_\beta(t)\omega_0-\omega_0\right\|_{\dot H^{-1+\delta+m}}
			\to 0
			\quad\text{as $t\to0$.}\label{eq:4.88}
		\end{align}
	For the second term, we have
		\begin{align*}
			&\quad\int_0^t\left\|T_\beta(t-\tau)\operatorname{div}(u(\tau)\omega(\tau))\right\|_{\dot H^{-1+\delta+m}}\, d\tau
			\\
			&\lesssim\int_0^t\left\|e^{(t-\tau)\Delta}(-\Delta)^{\frac{\delta+m}{2}}(u(\tau)\omega(\tau))\right\|_{L^2}\, d\tau
			\\
			&\lesssim\int_0^t(t-\tau)^{-\frac1q+\frac12}\left(\left\|(-\Delta)^\frac{\delta+m}{2}u(\tau)\right\|_{L^{p_1}}\left\|\omega(\tau)\right\|_{\dot W^{\delta, p_2}}
			+\left\|u(\tau)\right\|_{\dot W^{\delta, p_1}}\left\|(-\Delta)^\frac{\delta+m}{2}\omega(\tau)\right\|_{L^{p_2}}\right)\, d\tau
			\\
			&\lesssim t^{1-\frac1{r_1}-\frac1{r_2}-\frac1q+\frac12}\left(\left\|\omega\right\|_{L^{r_1}([0, \infty); \dot W^{-1+\delta+m, p_1})}\left\|\omega\right\|_{Y_2}
					+\left\|\omega\right\|_{L^{r_2}([0, \infty); \dot W^{\delta+m, p_2})}\left\|\omega\right\|_{Y_1}\right)
					\\
					&\to 0 \quad\text{as $t\to0$}.
		\end{align*}
	This completes the proof.
\qed

\begin{proposition}\label{prop:4.4}
	For any
		$m\geq0$, 
	The solution
		$\omega$
	satisfies
		$\omega\in C^1((0, \infty); \dot H^{\delta+m})$
	and
		$\partial_t\omega+(u\cdot\nabla)\omega=\Delta\omega+\beta L_1\omega$
	for
		$t>0$
	in
		$\dot H^{\delta+m}.$
\end{proposition}
\proof
	Take
		$t>0$
	and
		$h>0.$
	We write
		\begin{align}
			&\frac{\omega(t+h)-\omega(t)}{h}
			=\frac{T_\beta(t+h)\omega_0-T_\beta(t)\omega_0}{h}\label{eq:4.94}\\
			&-\frac1h\left(\int_0^{t+h}T_\beta(t+h-\tau)\operatorname{div}(u(\tau)\omega(\tau))\, d\tau-\int_0^{t}T_\beta(t+h-\tau)\operatorname{div}(u(\tau)\omega(\tau))\, d\tau\right)\label{eq:4.95}\\
			&-\frac1h\left(\int_0^{t}T_\beta(t+h-\tau)\operatorname{div}(u(\tau)\omega(\tau))\, d\tau
			-\int_0^{t}T_\beta(t-\tau)\operatorname{div}(u(\tau)\omega(\tau))\, d\tau\right)\label{eq:4.96}
		\end{align}
	and calculate the limits of each line as 
		$h$
	tends to 0.
	
	For \eqref{eq:4.94}, we have
		\begin{align}
			&\quad\left\|\frac{T_\beta(t+h)\omega_0-T_\beta(t)\omega_0}{h}-\Delta T_\beta(t)\omega_0-\beta L_1 T_\beta(t)\omega_0\right\|_{\dot H^{\delta+m}}\label{eq:4.97}\\
			&\leq\left\|\frac{e^{h\Delta}e^{h\beta L_1}T_\beta(t)\omega_0-e^{h\beta L_1}T_\beta(t)\omega_0}{h}-e^{h\beta L_1}\Delta T_\beta(t)\omega_0\right\|_{\dot H^{\delta+m}}\label{eq:4.98}\\
			&+\left\|e^{h\beta L_1}\Delta T_\beta(t)\omega_0-\Delta T_\beta(t)\omega_0\right\|_{\dot H^{\delta+m}}\label{eq:4.99}\\
			&+\left\|\frac{e^{h\beta L_1}T_\beta(t)\omega_0-T_\beta(t)\omega_0}{h}-\beta L_1 T_\beta(t)\omega_0\right\|_{\dot H^{\delta+m}}.\label{eq:4.100}
		\end{align}
	For \eqref{eq:4.98}, by the mean value theorem we can deduce
		\begin{align}
			&\quad\left\|\frac{e^{h\Delta}e^{h\beta L_1}T_\beta(t)\omega_0-e^{h\beta L_1}T_\beta(t)\omega_0}{h}-e^{h\beta L_1}\Delta T_\beta(t)\omega_0\right\|_{\dot H^{\delta+m}}\label{eq:4.101}\\
			&\leq\left(\int_{\mathbb R^2}\left|\left|\xi\right|^2\int_0^1\left(1-e^{-h\theta|\xi|^2}\right)\, d\theta\right|^2|\xi|^{2\delta+2m}\left|\mathcal F\left(T_\beta(t)\omega_0\right)(\xi)\right|^2\, d\xi\right)^\frac12\label{eq:4.102}
		\end{align}
	and the limit of the right-hand side is 0 due to the dominated convergence theorem and that
		$T_\beta(t)\omega_0\in \dot H^{\delta+m+2}.$
	We see that \eqref{eq:4.99} also tends to 0 for the same reason. It can similarly be shown that the limit of \eqref{eq:4.100} is 0 by
		$T_\beta(t)\omega_0\in \dot H^{-1+\delta+m}.$
		
		We next treat \eqref{eq:4.95}. For this term, we have
			\begin{align}
				&\quad\left\|\frac1h\int_t^{t+h}T_\beta(t+h-\tau)\operatorname{div}(u(\tau)\omega(\tau))\, d\tau-\operatorname{div}(u(t)\omega(t))\right\|_{\dot H^{\delta+m}}\label{eq:4.103}\\
				&\leq\left\|\frac1h\int_t^{t+h}T_\beta(t+h-\tau)\left(\operatorname{div}(u(\tau)\omega(\tau))-\operatorname{div}(u(t)\omega(t))\right)\, d\tau\right\|_{\dot H^{\delta+m}}\label{eq:4.104}\\
				&+\left\|\frac1h\int_t^{t+h}\left(T_\beta(t+h-\tau)-1\right)\operatorname{div}(u(t)\omega(t))\, d\tau\right\|_{\dot H^{\delta+m}}\label{eq:4.105}\\
				&\leq \sup_{t<\tau<t+h}\left\|\operatorname{div}(u(\tau)\omega(\tau))-\operatorname{div}(u(t)\omega(t))\right\|_{\dot H^{\delta+m}}\label{eq:4.108}\\
				&+\sup_{t<\tau<t+h}\left\|\left(T_\beta(t+h-\tau)-1\right)\operatorname{div}(u(t)\omega(t))\right\|_{\dot H^{\delta+m}}.\label{eq:4.107}
			\end{align}
		Since 
			$\omega\in C((0, \infty); \dot H^{-1+\delta+m})$
		for any
			$m\geq 0, $
		we see that
			$\operatorname{div}(u\omega)\in C((0, \infty); \dot H^{\delta+m})$
		so the right-hand side of the inequality tends to 0 as 
			$h\to0$.
		
		For \eqref{eq:4.96}, we divide the estimate into several parts. More precisely, we have
			\begin{align}
				&\quad\frac1h\int_0^{t}\left(T_\beta(h)-1\right)T_\beta(t-\tau)\operatorname{div}(u(\tau)\omega(\tau))\, d\tau
				\\
				&=\frac1h\int_0^{\frac t2}\left(e^{h\Delta}-1\right)e^{h\beta L_1}T_\beta(t-\tau)\operatorname{div}(u(\tau)\omega(\tau))\, d\tau
				\label{eq:4.109}\\
				&+\frac1h\int_0^{\frac t2}\left(e^{h\beta L_1}-1\right)T_\beta(t-\tau)\operatorname{div}(u(\tau)\omega(\tau))\, d\tau
				\label{eq:4.110}\\
				&+\frac1h\int_\frac t2^{t}\left(e^{h\Delta}-1\right)e^{h\beta L_1}T_\beta(t-\tau)\operatorname{div}(u(\tau)\omega(\tau))\, d\tau
				\label{eq:4.111}\\
				&+\frac1h\int_\frac t2^{t}\left(e^{h\beta L_1}-1\right)T_\beta(t-\tau)\operatorname{div}(u(\tau)\omega(\tau))\, d\tau
				\label{eq:4.112}
			\end{align}
		and calculate the limits of each term.
		
		For \eqref{eq:4.109}, we have
			\begin{align*}
				&\quad\left\|\frac1h\int_0^{\frac t2}\left(e^{h\Delta}-1\right)e^{h\beta L_1}T_\beta(t-\tau)\operatorname{div}(u(\tau)\omega(\tau))\, d\tau
				-\int_0^{\frac t2}\Delta T_\beta(t-\tau)\operatorname{div}(u(\tau)\omega(\tau))\, d\tau
				\right\|_{\dot H^{\delta +m}}
				\\
				&\leq \left\|\frac1h\int_0^{\frac t2}\left(e^{h\Delta}-1\right)e^{h\beta L_1}T_\beta(t-\tau)\operatorname{div}(u(\tau)\omega(\tau))\, d\tau-\frac1h\int_0^{\frac t2}\left(e^{h\Delta}-1\right)T_\beta(t-\tau)\operatorname{div}(u(\tau)\omega(\tau))\, d\tau\right\|_{\dot H^{\delta+m}}
				\\
				&+\left\|\frac1h\int_0^{\frac t2}\left(e^{h\Delta}-1\right)T_\beta(t-\tau)\operatorname{div}(u(\tau)\omega(\tau))\, d\tau
				-\int_0^{\frac t2}\Delta T_\beta(t-\tau)\operatorname{div}(u(\tau)\omega(\tau))\, d\tau\right\|_{\dot H^{\delta+m}}.
			\end{align*}
		The estimate of the first term can be treated as
			\begin{align}
				&\left\|\frac1h\int_0^{\frac t2}\left(e^{h\Delta}-1\right)\left(e^{h\beta L_1}-1\right)T_\beta(t-\tau)\operatorname{div}(u(\tau)\omega(\tau))\, d\tau\right\|_{\dot H^{\delta+m}}
				\label{eq:4.116}\\
				&\lesssim h\int_0^{\frac t2}\left\|\Delta\beta L_1e^{(t-\tau)\Delta}(-\Delta)^\frac{1+\delta+m}{2}(u(\tau)\omega(\tau))\right\|_{L^2}\, d\tau
				\label{eq:4.117}\\
				&\lesssim h|\beta|\int_0^{\frac t2}(t-\tau)^{-\frac{2+m}{2}-\frac 1q+\frac12}\left\|(-\Delta)^\frac{\delta}{2}(u(\tau)\omega(\tau))\right\|_{L^q}\, d\tau
				\label{eq:4.118}\\
				&\lesssim h|\beta|t^{1-\frac1{r_1}-\frac1{r_2}-\frac{2+m}{2}-\frac 1q+\frac12}\left\|\omega\right\|_{Y_1}\left\|\omega\right\|_{Y_2}, 
				\label{eq:4.119}
			\end{align}
		and the bound goes to 0 as
			$h\to0.$
		For the second term, we have
			\begin{align}
				&\quad\left\|\int_0^{\frac t2}\left(\frac{e^{h\Delta}-1}{h}-\Delta\right)T_\beta(t-\tau)\operatorname{div}(u(\tau)\omega(\tau))\, d\tau\right\|_{\dot H^{\delta+m}}
				\label{eq:4.120}\\
				&\lesssim\int_0^{\frac t2}\left\|\left(\frac{e^{h\Delta}-1}{h}-\Delta\right)e^{(t-\tau)\Delta}(-\Delta)^{\frac{1+\delta+m}{2}}(u(\tau)\omega(\tau))\right\|_{L^2}\, d\tau.
				\label{eq:4.121}
			\end{align}
		Here, the Plancherel theorem yields
			\begin{align}
				&\quad\left\|\left(\frac{e^{h\Delta}-1}{h}-\Delta\right)e^{(t-\tau)\Delta}(-\Delta)^{\frac{1+\delta+m}{2}}(u(\tau)\omega(\tau))\right\|_{L^2}\label{eq:4.122}\\
				&=\frac1{2\pi}\left(\int_{\mathbb R^2}\left||\xi|^2\int_0^1\left(e^{-h\theta |\xi|^2}-1\right)\, d\theta\right|^2\left|\mathcal F\left(e^{(t-\tau)\Delta}(-\Delta)^{\frac{1+\delta+m}{2}}(u(\tau)\omega(\tau))\right)(\xi)\right|^2\, d\xi\right)^\frac12\label{eq:4.123}
			\end{align}
		and by the dominated convergence theorem, this tends to 0 for each
			$\tau\in(0, \frac t2), $
		as
			$h\to0.$
		We also have
			\begin{align}
				&\quad\left\|\left(\frac{e^{h\Delta}-1}{h}-\Delta\right)e^{(t-\tau)\Delta}(-\Delta)^{\frac{1+\delta+m}{2}}(u(\tau)\omega(\tau))\right\|_{L^2}\label{eq:4.124}\\
				&\lesssim\left\|e^{(t-\tau)\Delta}(-\Delta)^{\frac{3+\delta+m}{2}}(u(\tau)\omega(\tau))\right\|_{L^2}\label{eq:4.125}\\
				&\lesssim(t-\tau)^{-\frac{3+m}{2}-\frac1q+\frac12}\left\|(-\Delta)^{\frac{\delta}{2}}(u(\tau)\omega(\tau))\right\|_{L^q}\label{eq:4.126}
			\end{align}
		and this is integrable on
			$(0, \frac t2).$
		Then again by the dominated convergence theorem, we have
			\begin{align}
				\int_0^{\frac t2}\left\|\left(\frac{e^{h\Delta}-1}{h}-\Delta\right)e^{(t-\tau)\Delta}(-\Delta)^{\frac{1+\delta+m}{2}}(u(\tau)\omega(\tau))\right\|_{L^2}\, d\tau
				\to 0\quad\text{as $h\to0$, }
			\end{align}
		and these are sufficient for the estimate of \eqref{eq:4.109}.
		
		For \eqref{eq:4.110}, similarly we have
			\begin{align*}
				&\left\|\frac1h\int_0^{\frac t2}\left(e^{h\beta L_1}-1\right)T_\beta(t-\tau)\operatorname{div}(u(\tau)\omega(\tau))\, d\tau
				-\int_0^{\frac t2}\beta L_1T_\beta(t-\tau)\operatorname{div}(u(\tau)\omega(\tau))\, d\tau\right\|_{\dot H^{\delta+m}}
				\\
				&\lesssim\int_0^{\frac t2}\left\|\left(\frac{e^{h\beta L_1}-1}{h}-\beta L_1\right)e^{(t-\tau)\Delta}(-\Delta)^\frac{1+\delta+m}{2}(u(\tau)\omega(\tau))\right\|_{L^{2}}\, d\tau.
			\end{align*}
		For each
			$\tau\in (0,\frac t2), $
		by the Plancherel theorem we deduce
			\begin{align}
				&\left\|\left(\frac{e^{h\beta L_1}-1}{h}-\beta L_1\right)e^{(t-\tau)\Delta}(-\Delta)^\frac{1+\delta+m}{2}(u(\tau)\omega(\tau))\right\|_{L^{2}}\label{eq:4.130}\\
				&=\frac1{2\pi}\left(\int_{\mathbb R^2}\left|\beta\frac{i\xi_1}{|\xi|^2}\int_0^1\left(e^{h\theta \beta\frac{i\xi_1}{|\xi|^2}}-1\right)\, d\theta\right|^2\left|\mathcal F\left(e^{(t-\tau)\Delta}(-\Delta)^{\frac{1+\delta+m}{2}}(u(\tau)\omega(\tau))\right)(\xi)\right|^2\, d\xi\right)^\frac12\label{eq:4.131}
			\end{align}
		and this tends to as
			$h\to0.$
		Also we have
			\begin{align}
				&\left\|\left(\frac{e^{h\beta L_1}-1}{h}-\beta L_1\right)e^{(t-\tau)\Delta}(-\Delta)^\frac{1+\delta+m}{2}(u(\tau)\omega(\tau))\right\|_{L^{2}}\label{eq:4.132}\\
				&\lesssim|\beta|\left\|e^{(t-\tau)\Delta}(-\Delta)^\frac{\delta+m}{2}(u(\tau)\omega(\tau))\right\|_{L^{2}}\label{eq:4.133}\\
				&\lesssim|\beta|(t-\tau)^{-\frac{m}{2}-\frac1q+\frac12}\left\|(-\Delta)^\frac{\delta}{2}(u(\tau)\omega(\tau))\right\|_{L^{q}}\label{eq:4.134} 
			\end{align}
		and this is integrable on
			$(0, \frac t2).$
		As a result we see that the limit of \eqref{eq:4.110} as 
			$h\to0$ 
		is 0.
		
		Similarly, we can deal with the estimates of \eqref{eq:4.111} and \eqref{eq:4.112}. For \eqref{eq:4.111}, we have
			\begin{align}
				&\quad\left\|\frac1h\int_\frac t2^{t}\left(e^{h\Delta}-1\right)e^{h\beta L_1}T_\beta(t-\tau)\operatorname{div}(u(\tau)\omega(\tau))\, d\tau
				-\int_\frac t2^{t}\Delta T_\beta(t-\tau)\operatorname{div}(u(\tau)\omega(\tau))\, d\tau\right\|_{\dot H^{\delta+m}}\notag\\
				&\leq\left\|\frac1h\int_\frac t2^{t}\left(e^{h\Delta}-1\right)e^{h\beta L_1}T_\beta(t-\tau)\operatorname{div}(u(\tau)\omega(\tau))\, d\tau
				-\frac1h\int_\frac t2^{t}\left(e^{h\Delta}-1\right) T_\beta(t-\tau)\operatorname{div}(u(\tau)\omega(\tau))\, d\tau\right\|_{\dot H^{\delta+m}}
				\label{eq:4.136}\\
				&+\left\|\frac1h\int_\frac t2^{t}\left(e^{h\Delta}-1\right) T_\beta(t-\tau)\operatorname{div}(u(\tau)\omega(\tau))\, d\tau
				-\int_\frac t2^{t}\Delta T_\beta(t-\tau)\operatorname{div}(u(\tau)\omega(\tau))\, d\tau\right\|_{\dot H^{\delta+m}}.
				\label{eq:4.137}
			\end{align}
		For \eqref{eq:4.136}, we have
			\begin{align}
				&\quad\left\|\frac1h\int_\frac t2^{t}\left(e^{h\Delta}-1\right)\left(e^{h\beta L_1}-1\right)T_\beta(t-\tau)\operatorname{div}(u(\tau)\omega(\tau))\, d\tau\right\|_{\dot H^{\delta+m}}
				\label{eq:4.138}\\
				&\lesssim\frac1h\int_\frac t2^{t}\left\|\left(e^{h\Delta}-1\right)\left(e^{h\beta L_1}-1\right)e^{(t-\tau)\Delta}(-\Delta)^\frac{1+\delta+m}{2}(u(\tau)\omega(\tau))\right\|_{L^2}\, d\tau
				\label{eq:4.139}\\
				&\lesssim h\int_\frac t2^{t}\left\|\Delta\beta L_1 e^{(t-\tau)\Delta}(-\Delta)^\frac{1+\delta+m}{2}(u(\tau)\omega(\tau))\right\|_{L^2}\, d\tau
				\label{eq:4.140}\\
				&\lesssim  h|\beta|\int_\frac t2^{t}(t-\tau)^{-\frac1q+\frac12}\left\|(-\Delta)^\frac{2+\delta+m}{2}(u(\tau)\omega(\tau))\right\|_{L^q}\, d\tau
				\label{eq:4.141}\\
				&\lesssim h|\beta|t^{1-\frac1{r_1}-\frac1{r_2}-\frac1q+\frac12}\left(\left\|\omega\right\|_{L^{r_1}([\frac t2, \infty); \dot W^{-1+\delta+m+2, p_1})}\left\|\omega\right\|_{Y_2}
					+\left\|\omega\right\|_{L^{r_2}([\frac t2, \infty); \dot W^{\delta+m+2, p_2})}\left\|\omega\right\|_{Y_1}\right)\notag
			\end{align}
		and the right-hand side goes to 0 as 
			$h\to0.$
		
		For \eqref{eq:4.137}, we have
			\begin{align}
				&\quad\left\|\int_\frac t2^{t}\left(\frac{e^{h\Delta}-1}{h}-\Delta\right) T_\beta(t-\tau)\operatorname{div}(u(\tau)\omega(\tau))\, d\tau\right\|_{\dot H^{\delta+m}}\label{eq:4.143}\\
				&\lesssim \int_\frac t2^{t}\left\|\left(\frac{e^{h\Delta}-1}{h}-\Delta\right) e^{(t-\tau)\Delta}(-\Delta)^\frac{1+\delta+m}{2}(u(\tau)\omega(\tau))\right\|_{L^2}\, d\tau.
				\label{eq:4.144}
			\end{align}
		Here, using the equality \eqref{eq:4.123} is sufficient to show that the integrand tends to 0 as 
			$h\to0$
		for each 
			$\tau\in(\frac t2, t)$.
		We also have
			\begin{align}
				&\quad\left\|\left(\frac{e^{h\Delta}-1}{h}-\Delta\right) e^{(t-\tau)\Delta}(-\Delta)^\frac{1+\delta+m}{2}(u(\tau)\omega(\tau))\right\|_{L^2}\label{eq:4.145}\\
				&\lesssim(t-\tau)^{-\frac1q+\frac12}\left\| (-\Delta)^\frac{3+\delta+m}{2}(u(\tau)\omega(\tau))\right\|_{L^q}\label{eq:4.146}
			\end{align}
		and this is integrable on
			$(\frac t2, t)$.
		Therefore \eqref{eq:4.144} tends to 0 as
			$h\to0.$
		
		Finally for \eqref{eq:4.112}, we have
			\begin{align}
				&\quad\left\|\frac1h\int_\frac t2^{t}\left(e^{h\beta L_1}-1\right)T_\beta(t-\tau)\operatorname{div}(u(\tau)\omega(\tau))\, d\tau
				-\int_\frac t2^{t}\beta L_1T_\beta(t-\tau)\operatorname{div}(u(\tau)\omega(\tau))\, d\tau\right\|_{\dot H^{\delta+m}}
				\notag\\
				&\lesssim\int_\frac t2^{t}\left\|\left(\frac{e^{h\beta L_1}-1}{h}-\beta L_1\right)e^{(t-\tau)\Delta}(-\Delta)^\frac{1+\delta+m}{2}(u(\tau)\omega(\tau))\right\|_{L^2}\, d\tau.\label{eq:4.148}
			\end{align}
		Now \eqref{eq:4.131} ensures that the integrand tends to 0 as 
			$h\to 0$
		for all
			$\tau\in(\frac t2, t)$.
		We have
			\begin{align}
				&\quad\left\|\left(\frac{e^{h\beta L_1}-1}{h}-\beta L_1\right)e^{(t-\tau)\Delta}(-\Delta)^\frac{1+\delta+m}{2}(u(\tau)\omega(\tau))\right\|_{L^2}\\
				&\lesssim |\beta|(t-\tau)^{-\frac1q+\frac12}\left\|(-\Delta)^\frac{\delta+m}{2}(u(\tau)\omega(\tau))\right\|_{L^q}
			\end{align}
		and the bound is integrable on
			$(\frac t2, t).$
		Consequently, we see that \eqref{eq:4.148} tends to 0 as
			$h\to0.$
		
		Combining the above estimates, we see that
			\begin{align}
				\lim_{h\downarrow0}\frac{\omega(t+h)-\omega(t)}{h}
				&=-\operatorname{div}(u(t)\omega(t))+\Delta \omega(t)+\beta L_1\omega(t)\label{eq:4.151}
			\end{align}
		for
			$t>0$
		in
			$\dot H^{\delta+m}.$
		Repeating almost the same argument, we can deduce that this is true when 
			$h$
		tends to 0 from below (see Appendix~{\ref{sect:b}}).
		As a result we have
			$\partial_t\omega=-\operatorname{div}(u\omega)+\Delta \omega+\beta L_1\omega$
		for 
			$t>0.$
		Since the right-hand side is included in
			$C((0, \infty); \dot H^{\delta+m})$, 
		this concludes the proof.
\qed

\proof[Proof of Theorem~\ref{thm:1.2}]
	Combining the results of Proposition~\ref{prop:4.2}, Corollary~\ref{cor:4.3}, and Proposition~\ref{prop:4.4} leads to the conclusion.
\qed

\section{Proof of Theorem~\ref{thm:1.3}}\label{sect:5}
In this section, we assume that 	
	$0\leq \delta\leq\frac15$
and
	$\mathcal A
	=(\delta, p, r_1, p, r_2)$
satisfy the conditions in Theorem~\ref{thm:1.1} and Theorem~\ref{thm:1.2}, where
	$\frac1p\coloneqq \frac13+\frac\delta6\in[\frac13, \frac{11}{30}].$
Let 
	$\omega$
denote the solution obtained in Theorem~\ref{thm:1.1}.

\subsection{Temporal decays of $\dot W^{s, p}$-norms}
We start our argument by a standard energy estimate of 
	$\omega$.

\begin{lemma}\label{lem:5.1}
	For any
		$0<s<t, $
	The equality
		\begin{align}
			\left\|\omega(t)\right\|_{L^2}^2
			+2\int_s^t\left\|\nabla\omega(\tau)\right\|_{L^2}^2\, d\tau
			=\left\|\omega(s)\right\|_{L^2}^2\label{eq:5.1}
		\end{align}
	holds. In particular, if
		$\omega_0\in \dot H^{-1+\delta}\cap \dot H^1$, 
	we can take
		$s=0.$
\end{lemma}
\proof
	Since
		$\omega\in C((0, \infty); \dot H^{-1+\delta+m})\cap C^1((0, \infty); \dot H^{\delta+m})$
	for any
		$m\geq 0, $
	we have
		\begin{align}
			&\quad\frac{d}{d\tau}\left\|\omega(\tau)\right\|_{L^2}^2
			=2\left\langle \partial_{t}\omega(\tau), \omega(\tau)\right\rangle_{L^2}\label{eq:5.2}\\
			&=2\left\langle \Delta\omega(\tau), \omega(\tau)\right\rangle_{L^2}
			+2\left\langle \beta L_1\omega(\tau), \omega(\tau)\right\rangle_{L^2}
			-2\left\langle (u(\tau)\cdot\nabla)\omega(\tau), \omega(\tau)\right\rangle_{L^2}\label{eq:5.3}
		\end{align}
	for $\tau>0$. The second and third term vanish due to the skew-symmetry of the operator
		$L_1$
	and the divergence free condition of
		$u(\tau), $
	respectively.
	Integration by parts gives
		$\left\langle \Delta\omega(\tau), \omega(\tau)\right\rangle_{L^2}
		=-\left\|\nabla\omega(\tau)\right\|_{L^2}^2$.
	Thus we obtain
		\begin{align}
			\frac{d}{d\tau}\left\|\omega(\tau)\right\|_{L^2}^2+2\left\|\nabla\omega(\tau)\right\|_{L^2}^2=0.\label{eq:5.4}
		\end{align}
	By integrating this formula over $(s, t)$, we have \eqref{eq:5.1}. If 
		$\omega_0\in \dot H^{-1+\delta}\cap \dot H^1$, 
	by Corollary~\ref{cor:4.3}, 
		$\omega\in C([0, \infty); L^2)$
	and therefore, we can set
		$s=0.$
\qed

For 
	$s\geq0, a\in[2, \infty]$
and
	$t>0, $
we define
	\begin{align}
		M_{s, a}(t)\coloneqq t^{-\frac s2-1+\frac1a}\min\left\{1, \left|\beta\right|^{-1+\frac 2a}t^{-\frac32(1-\frac2a)}\right\},\label{eq:5.5}
	\end{align}
and
	\begin{align}
		\Omega_{s, a}(t)
		\coloneqq \sup_{0<\tau\leq t}M_{s, a}(\tau)^{-1}\left\|\omega(\tau)\right\|_{\dot W^{s, a}}.\label{eq:5.6}
	\end{align}
We also abbreviate 
	$\Omega_{s, p}(t)$
and
	$M_{s, p}(t)$
as
	$\Omega_s(t)$
and
	$M_s(t), $
respectively.
We show the temporal decay estimates of
	$\left\|\omega(t)\right\|_{L^p}.$	
\begin{proposition}\label{prop:5.2}
	Set
		$\theta_0\coloneqq \frac{\frac2p-\frac1{r_1}}{2-\frac3p}\in(0, 1].$
	There are constants
		$C_1, C_2>0$
	such that for any
		$\omega_0\in \dot H^{-1+\delta}\cap \dot H^{2(1-\frac1p)}\cap L^1$
	satisfying
		$\left|\beta\right|^{-\frac\delta3}\left\|\omega_0\right\|_{\dot H^{-1+\delta}}
		\leq C_1, $
	we have
		\begin{align*}
			\Omega_0(t)\leq
			\begin{cases}
				C_2\left\|\omega_0\right\|_{L^1} & \text{if $\theta_0=1, $}\\
				C_2 \left(\left\|\omega_0\right\|_{L^1}
				+\left\|\omega_0\right\|_{L^1}^{\frac{\delta}{1+\delta}}\left(|\beta|^{-\frac{1+\delta}{3}}\left\|\omega_0\right\|_{\dot H^\delta}\right)^\frac{1}{1+\delta}\left(|\beta|^{-\frac\delta3}\left\|\omega_0\right\|_{\dot H^{-1+\delta}}\right)^{\frac1{1-\theta_0}}\right) & \text{if $\theta_0<1, $}
			\end{cases}
		\end{align*}
	for all
		$t>0.$
\end{proposition}
\proof
	We first confirm that
		$\theta_0\in(0, 1].$
	We have
		$2-\frac3p=\frac12+3(\frac12-\frac1p)>0, $
	and thus the denominator of
		$\theta_0$
	is positive. By the assumption of Theorem~\ref{thm:1.1}, we have
		$\frac2p-\frac1{r_1}\geq \frac2p-(\frac12-\frac1p+\frac\delta2)=\frac12>0.$
	This ensures that
		$\theta_0>0.$
	We also have
		$(2-\frac3p)-(\frac2p-\frac1{r_1})
		\geq 2-\frac5p+\frac12-\frac1p+\frac\delta2-\frac32(1-\frac2p)
		=0,$
	and therefore	
		$\theta_0\leq1.$
	Now we define
		$c\in(1, p']$
	by
		$\frac1c\coloneqq \frac1p-\frac\delta2+\frac{\theta _0}p+\frac{1-\theta_0}2(\geq\frac1p-\frac\delta2+\frac1p=\frac1{p'}).$
	
	Since
		$\omega$
	is a solution of the integral equation \eqref{eq:1.9}, 
	we have
		\begin{align}
			\left\|\omega(t)\right\|_{L^p}\label{eq:5.9}
			&\leq \left\|T_\beta(t)\omega_0\right\|_{L^p}\\
			&+\int_0^\frac t2\left\|T_\beta(t-\tau)\operatorname{div}(u(\tau)\omega(\tau))\right\|_{L^p}\, d\tau
			+\int^t_\frac t2\left\|T_\beta(t-\tau)\operatorname{div}(u(\tau)\omega(\tau))\right\|_{L^p}\, d\tau\label{eq:5.10}
		\end{align}
	for
		$t>0$.
	We deduce the estimates of each term. For the first term, Proposition~\ref{cor:2.3} gives
		\begin{align}
			\left\|T_\beta(t)\omega_0\right\|_{L^p}
			&\lesssim t^{-1+\frac1p}\min\left\{1, \left|\beta\right|^{-1+\frac2p}t^{-\frac32(1-\frac2p)}\right\}\left\|\omega_0\right\|_{L^{1}}.\label{eq:5.11}
		\end{align}
		
	For the second term, we have
		\begin{align}
			&\quad \int_0^\frac t2\left\|T_\beta(t-\tau)\operatorname{div}(u(\tau)\omega(\tau))\right\|_{L^p}\, d\tau\label{eq:5.12}\\
			&\lesssim \int_0^\frac t2\left\|T_\beta(t-\tau)\operatorname{div}(u(\tau)\omega(\tau))\right\|_{\dot B^0_{p, 2}}\, d\tau\label{eq:5.13}\\
			&\lesssim \int_0^\frac t2(t-\tau)^{-\frac12-1+\frac2p}\min\left\{1, |\beta|^{-1+\frac2p}(t-\tau)^{-\frac32(1-\frac2p)}\right\}\left\|e^{\frac12(t-\tau)\Delta}(u(\tau)\omega(\tau))\right\|_{\dot B^0_{p', 2}}\, d\tau\notag\\
			&\lesssim \int_0^\frac t2(t-\tau)^{-\frac12-1+\frac2p-\frac1c+\frac1{p'}}\min\left\{1, |\beta|^{-1+\frac2p}(t-\tau)^{-\frac32(1-\frac2p)}\right\}\left\|u(\tau)\omega(\tau)\right\|_{L^c}\, d\tau\label{eq:5.15}\\
			&\lesssim M_{0}(t)t^{\frac12-\frac1c}\int_0^\frac t2\left\|\omega(\tau)\right\|_{\dot W^{-1+\delta, p}}\left\|\omega(\tau)\right\|_{L^2}^{1-\theta_0}\left\|\omega(\tau)\right\|_{L^p}^{\theta_0}\, d\tau\label{eq:5.16}\\
			&\leq M_{0}(t)t^{\frac12-\frac1c}\left\|\omega_0\right\|_{L^2}^{1-\theta_0}\Omega_0(t)^{\theta_0}\int_0^\frac t2\left\|\omega(\tau)\right\|_{\dot W^{-1+\delta, p}}M_{0}(\tau)^{\theta_0}\, d\tau\label{eq:5.17}\\
			&\leq M_{0}(t)t^{\frac12-\frac1c}\left\|\omega_0\right\|_{L^2}^{1-\theta_0}\Omega_0(t)^{\theta_0}\left\|\omega\right\|_{Y_1}\left(\int_0^\frac t2\left(|\beta|^{-1+\frac2p}\tau^{-1+\frac{1}{p}-\frac32(1-\frac2p)}\right)^{\theta_0 r_1'}\, d\tau\right)^\frac{1}{r_1'}.\label{eq:5.18}
		\end{align}
	The integral on the right-hand side is finite if and only if
		\begin{align}
			\theta_0 r_1'\left(1-\frac1p+\frac32\left(1-\frac2p\right)\right)<1.\label{eq:5.19}
		\end{align}
	The condition is equivalent to
		$-2(1-\frac2p)^2<\frac1{r_1}(\frac12-\frac1p)$, 
	and this is trivially true. We can also check that
		$\frac12-\frac1c+1-\frac1{r_1}-\theta_0(1-\frac1p+\frac32(1-\frac2p))=0.$
	Therefore, we obtain
		\begin{align}
			\int_0^\frac t2\left\|T_\beta(t-\tau)\operatorname{div}(u(\tau)\omega(\tau))\right\|_{L^p}\, d\tau\lesssim M_{0}(t)\left\|\omega_0\right\|_{L^2}^{1-\theta_0}\Omega_0(t)^{\theta_0}\left\|\omega\right\|_{Y_1}|\beta|^{-\theta_0(1-\frac2p)}.\label{eq:5.20}
		\end{align}
	The interpolation inequality and the Sobolev embedding yield 
		\begin{align}
			|\beta|^{-\frac13}\left\|\omega_0\right\|_{L^2}
			\leq |\beta|^{-\frac13}\left\|\omega_0\right\|_{L^1}^\frac{\delta}{1+\delta}\left\|\omega_0\right\|_{L^{\frac2{1-\delta}}}^\frac{1}{1+\delta}
			\lesssim \left\|\omega_0\right\|_{L^1}^\frac{\delta}{1+\delta}\left(|\beta|^{-\frac{1+\delta}{3}} \left\|\omega_0\right\|_{\dot H^\delta}\right)^\frac{1}{1+\delta}.\label{eq:5.21}
		\end{align}
	Since 
		$\omega$
	is the solution given by Theorem~\ref{thm:1.1}, this satisfies
		\begin{align}
			\left|\beta\right|^{-\frac13(1-\frac{2}{r_1}-\frac{2}{p}+\delta)}\left\|\omega\right\|_{Y_1}
			\lesssim |\beta|^{-\frac{\delta}{3}}\left\|\omega_0\right\|_{\dot H^{-1+\delta}}.\label{eq:5.22}
		\end{align}
	By the definition of 
		$\theta_0$, 
	we have
		$-\frac{1-\theta_0}{3}-\frac13(1-\frac{2}{r_1}-\frac{2}{p}+\delta)+\theta_0(1-\frac2p)=0.$
	Consequently, we obtain
		\begin{align}
			&\quad\int_0^\frac t2\left\|T_\beta(t-\tau)\operatorname{div}(u(\tau)\omega(\tau))\right\|_{L^p}\, d\tau\label{eq:5.23}\\
			&\lesssim M_{0}(t)\Omega_0(t)^{\theta_0}\left\|\omega_0\right\|_{L^1}^\frac{\delta(1-\theta_0)}{1+\delta}\left(|\beta|^{-\frac{1+\delta}{3}} \left\|\omega_0\right\|_{\dot H^{\delta}}\right)^\frac{1-\theta_0}{1+\delta}|\beta|^{-\frac{\delta}{3}}\left\|\omega_0\right\|_{\dot H^{-1+\delta}}.\label{eq:5.24}
		\end{align}
	
	For the third term in \eqref{eq:5.10}, we have
		\begin{align}
			&\quad\int^t_\frac t2\left\|T_\beta(t-\tau)\operatorname{div}(u(\tau)\omega(\tau))\right\|_{L^p}\, d\tau\label{eq:5.25}\\
			&\lesssim \int^t_\frac t2(t-\tau)^{-\frac12-1+\frac2p}\min\left\{1, |\beta|^{-1+\frac2p}(t-\tau)^{-\frac32(1-\frac2p)}\right\}\left\|u(\tau)\omega(\tau)\right\|_{L^{p'}}\, d\tau.\label{eq:5.26}
		\end{align}
	Since 
		$p$ 
	is defined by
		$\frac1p=\frac13+\frac\delta6, $
	we have
		$\frac1p+\frac1p-\frac\delta2=1-\frac1p.$
	Thus by H\"older's inequality and the Sobolev embedding, we obtain
		\begin{align}
			\left\|u(\tau)\omega(\tau)\right\|_{L^{p'}}
			\lesssim \left\|\omega(\tau)\right\|_{\dot W^{-1+\delta, p}}\left\|\omega(\tau)\right\|_{L^p}.\label{eq:5.27}
		\end{align}
	Inserting this estimate to \eqref{eq:5.26} and using H\"older's inequality, we have	
		\begin{align}
			&\quad\int^t_\frac t2\left\|T_\beta(t-\tau)\operatorname{div}(u(\tau)\omega(\tau))\right\|_{L^p}\, d\tau\label{eq:5.28}\\
			&\lesssim M_{0}(t)\Omega_0(t)\left\|\omega\right\|_{Y_1}
			\left(\int^t_\frac t2\left((t-\tau)^{-\frac12-\frac1p+\frac\delta2}\min\left\{1, |\beta|^{-1+\frac2p}(t-\tau)^{-\frac32(1-\frac2p)}\right\}\right)^{r_1'}\, d\tau\right)^\frac{1}{r_1'}.\notag\\
			&\lesssim M_{0}(t)\Omega_0(t)\left\|\omega\right\|_{Y_1}|\beta|^{-\frac23(\frac12-\frac1{r_1}-\frac1p+\frac\delta2)}\label{eq:5.30}\\
			&\lesssim M_{0}(t)\Omega_0(t)|\beta|^{-\frac{\delta}{3}}\left\|\omega_0\right\|_{\dot H^{-1+\delta}}.\label{eq:5.31}
		\end{align}
	Combining the above estimates, there is a constant
		$C_*>0$
	which satisfies
		\begin{align}
			\Omega_0(t)
			&\leq C_*\bigg(\left\|\omega_0\right\|_{L^1}
			+\Omega_0(t)|\beta|^{-\frac{\delta}{3}}\left\|\omega_0\right\|_{\dot H^{-1+\delta}}\label{eq:5.32}\\
			&+\Omega_0(t)^{\theta_0}\left\|\omega_0\right\|_{L^1}^\frac{\delta(1-\theta_0)}{1+\delta}\left(|\beta|^{-\frac{1+\delta}{3}} \left\|\omega_0\right\|_{\dot H^{\delta}}\right)^\frac{1-\theta_0}{1+\delta}|\beta|^{-\frac{\delta}{3}}\left\|\omega_0\right\|_{\dot H^{-1+\delta}}			\bigg).\label{eq:5.33}
		\end{align}
	By using Corollary~\ref{cor:4.3}, the assumption 
		$\omega_0\in \dot H^{2(1-\frac1p)}$
	ensures that
		$\omega\in C([0, \infty); \dot H^{1-\frac2p}).$
	By the embedding
		$\dot H^{1-\frac2p}\hookrightarrow L^p$, 
	we see that 
		$\Omega_0(t)$
	is finite and continuous for all
		$t>0.$
	Therefore, if it holds that
		$C_*|\beta|^{-\frac{\delta}{3}}\left\|\omega_0\right\|_{\dot H^{-1+\delta}}
		\leq\frac12, $
	we obtain
		\begin{align}
			\Omega_0(t)
			&\leq C_{**}\bigg(\left\|\omega_0\right\|_{L^1}
			+\Omega_0(t)^{\theta_0}\left\|\omega_0\right\|_{L^1}^\frac{\delta(1-\theta_0)}{1+\delta}\left(|\beta|^{-\frac{1+\delta}{3}} \left\|\omega_0\right\|_{\dot H^{\delta}}\right)^\frac{1-\theta_0}{1+\delta}|\beta|^{-\frac{\delta}{3}}\left\|\omega_0\right\|_{\dot H^{-1+\delta}}	\bigg)\label{eq:5.34}
		\end{align}
	for 
		$C_{**}=2C_*.$
	If
		$\theta_0=1, $
	by taking 
		$\omega_0$
	small enough to satisfy
		$C_{**}|\beta|^{-\frac\delta3}\left\|\omega_0\right\|_{\dot H^{-1+\delta}}
		\leq\frac12, $
	we obtain the desired estimate
		$\Omega_0(t)\leq 2C_{**}\left\|\omega_0\right\|_{L^1}.$
	Now we assume that
		$\theta_0<1, $
	and show the estimate
		\begin{align}
			\Omega_0(t)
			\leq C_{***} \left(\left\|\omega_0\right\|_{L^1}
				+\left\|\omega_0\right\|_{L^1}^{\frac{\delta}{1+\delta}}\left(|\beta|^{-\frac{1+\delta}{3}}\left\|\omega_0\right\|_{\dot H^\delta}\right)^\frac{1}{1+\delta}\left(|\beta|^{-\frac\delta3}\left\|\omega_0\right\|_{\dot H^{-1+\delta}}\right)^{\frac1{1-\theta_0}}\right), \quad\label{eq:5.34'}
		\end{align}
	where the constant
		$C_{***}>0$
	is defined by 
		$C_{***}=\max\left\{2C_{**}, (2C_{**})^{\frac1{1-\theta_0}}\right\}.$
	Assume that \eqref{eq:5.34'} does not hold for some 
		$t=t_0>0$. 
	Then we obtain
		\begin{align*}
			\Omega_0(t_0)
			&=\frac12\left(\Omega_0(t_0)+\Omega_0(t_0)\right)\\
			&> \frac12\left(C_{***}\left\|\omega_0\right\|_{L^1}
			+\Omega_0(t_0)^{\theta_0}
			\left(C_{***}\left\|\omega_0\right\|_{L^1}^{\frac{\delta}{1+\delta}}\left(|\beta|^{-\frac{1+\delta}{3}}\left\|\omega_0\right\|_{\dot H^\delta}\right)^\frac{1}{1+\delta}\left(|\beta|^{-\frac\delta3}\left\|\omega_0\right\|_{\dot H^{-1+\delta}}\right)^{\frac1{1-\theta_0}}\right)^{1-\theta_0}\right)\\
			&\geq C_{**}\bigg(\left\|\omega_0\right\|_{L^1}
			+\Omega_0(t_0)^{\theta_0}\left\|\omega_0\right\|_{L^1}^\frac{\delta(1-\theta_0)}{1+\delta}\left(|\beta|^{-\frac{1+\delta}{3}} \left\|\omega_0\right\|_{\dot H^{\delta}}\right)^\frac{1-\theta_0}{1+\delta}|\beta|^{-\frac{\delta}{3}}\left\|\omega_0\right\|_{\dot H^{-1+\delta}}	\bigg), 
		\end{align*}
	which contradicts \eqref{eq:5.34}. Thus we necessarily have \eqref{eq:5.34'}. This completes the proof.
\qed

\begin{proposition}\label{prop:5.3}
	Let
		$\omega_0$
	satisfy the assumption of Proposition~\ref{prop:5.2}.
Then there is a positive constant
			$C_1>0$
		such that for
			$\omega_0\in \dot H^{-1+\delta}\cap \dot H^{3-\frac2p}\cap L^1$
		satisfying
			$\left|\beta\right|^{-\frac{\delta}3}\left\|\omega_0\right\|_{\dot H^{-1+\delta}}
			\leq C_1, $
		there is a constant
			$C_2>0$
		dependent only on 
			$\mathcal A, $
			$\left\|\omega_0\right\|_{L^1}, $
			$|\beta|^{-\frac{\delta}3}\left\|\omega_0\right\|_{\dot H^{-1+\delta}}, $
		and
			$|\beta|^{-\frac{1+\delta}3}\left\|\omega_0\right\|_{\dot H^{\delta}}, $
		satisfying
			$\Omega_1(t)\leq C_2$
		for all
			$t>0.$
\end{proposition}
\proof
	As in the proof of Proposition~\ref{prop:5.2}, we have
		\begin{align}
			&\quad\left\|\omega(t)\right\|_{\dot W^{1,p}}\label{eq:5.43}
			\leq \left\|T_\beta(t)\omega_0\right\|_{\dot W^{1,p}}\\
			&+\int_0^\frac t2\left\|T_\beta(t-\tau)\operatorname{div}(u(\tau)\omega(\tau))\right\|_{\dot W^{1,p}}\, d\tau
			+\int^t_\frac t2\left\|T_\beta(t-\tau)\operatorname{div}(u(\tau)\omega(\tau))\right\|_{\dot W^{1, p}}\, d\tau.\quad\label{eq:5.44}
		\end{align}
	We can estimate the first and the second terms by using the smoothing property of the heat kernel to obtain
		\begin{align}
			\left\|T_\beta(t)\omega_0\right\|_{\dot W^{1,p}}
			\lesssim M_1(t)\left\|\omega_0\right\|_{L^1}
			\label{eq:5.45}
		\end{align}
	and
		\begin{align}
			&\quad\int_0^\frac t2\left\|T_\beta(t-\tau)\operatorname{div}(u(\tau)\omega(\tau))\right\|_{\dot W^{1,p}}\, d\tau\label{eq:5.46}\\
			&\lesssim M_1(t)\Omega_0(t)^{\theta_0}\left\|\omega_0\right\|_{L^1}^\frac{\delta(1-\theta_0)}{1+\delta}\left(|\beta|^{-\frac{1+\delta}{3}} \left\|\omega_0\right\|_{\dot H^{\delta}}\right)^\frac{1-\theta_0}{1+\delta}|\beta|^{-\frac{\delta}{3}}\left\|\omega_0\right\|_{\dot H^{-1+\delta}}\label{eq:5.47}.
		\end{align}
		
	For the third term of \eqref{eq:5.44}, we have
		\begin{align}
			&\quad\int^t_\frac t2\left\|T_\beta(t-\tau)\operatorname{div}(u(\tau)\omega(\tau))\right\|_{\dot W^{1, p}}\, d\tau\label{eq:5.50}\\
			&\lesssim \int^t_\frac t2(t-\tau)^{-\frac12-1+\frac2p}\min\left\{1, |\beta|^{-1+\frac2p}(t-\tau)^{-\frac32(1-\frac2p)}\right\}\left\|u(\tau)\cdot\nabla\omega(\tau)\right\|_{L^{p'}}\, d\tau.\label{eq:5.51}
		\end{align}
	Since H\"older's inequality yields
			$\left\|u(\tau)\cdot\nabla\omega(\tau)\right\|_{L^{p'}}
			\lesssim
			\left\|\omega(\tau)\right\|_{\dot W^{-1+\delta, p}}\left\|\omega(\tau)\right\|_{\dot W^{1, p}}, $
	similarly to the estimate of the third term of \eqref{eq:5.10}, we have
		\begin{align}
			&\quad\int^t_\frac t2\left\|T_\beta(t-\tau)\operatorname{div}(u(\tau)\omega(\tau))\right\|_{\dot W^{1, p}}\, d\tau\\
			&\lesssim\int^t_\frac t2(t-\tau)^{-\frac12-\frac1p+\frac\delta2}\min\left\{1, |\beta|^{-1+\frac2p}(t-\tau)^{-\frac32(1-\frac2p)}\right\}\left\|\omega(\tau)\right\|_{\dot W^{-1+\delta, p}}\left\|\omega(\tau)\right\|_{\dot W^{1, p}}\, d\tau\notag\\
			&\lesssim M_1(t)\Omega_1(t)\left|\beta\right|^{-\frac23(\frac12-\frac1{r_1}-\frac1p+\frac\delta2)}\left\|\omega\right\|_{Y_1}\label{eq:5.69}\\
			&\lesssim M_1(t)\Omega_1(t)\left|\beta\right|^{-\frac\delta3}\left\|\omega_0\right\|_{\dot H^{-1+\delta}}.\label{eq:5.70}
		\end{align}
	Combining the estimates \eqref{eq:5.45}, \eqref{eq:5.47} and \eqref{eq:5.70}, we obtain 
	\begin{align}
		\Omega_1(t)&\lesssim \left\|\omega_0\right\|_{L^1}+\left|\beta\right|^{-\frac\delta3}\left\|\omega_0\right\|_{\dot H^{-1+\delta}}\Omega_1(t)\label{eq:5.71}
		\\&+\Omega_0(t)^{\theta_0}
		\left\|\omega_0\right\|_{L^1}^\frac{\delta(1-\theta_0)}{1+\delta}\left(|\beta|^{-\frac{1+\delta}{3}} \left\|\omega_0\right\|_{\dot H^{\delta}}\right)^\frac{1-\theta_0}{1+\delta}|\beta|^{-\frac{\delta}{3}}\left\|\omega_0\right\|_{\dot H^{-1+\delta}}\label{eq:5.72}.
	\end{align}
	
	Noting that the assumption
		$\omega_0\in \dot H^{-1+\delta}\cap \dot H^{3-\frac2p}$
	ensures that
		$\omega\in C([0, \infty); \dot W^{1, p}), $
	the second term on the right-hand side is finite and therefore can be absorbed into the left-hand side by taking 
		$\left|\beta\right|^{-\frac\delta3}\left\|\omega_0\right\|_{\dot H^{-1+\delta}}$
	sufficiently small.
This completes the proof.
\qed

\begin{lemma}\label{lem:5.4}
	Suppose that
		$\omega_0\in \dot H^{-1+\delta}\cap\dot H^{3-\frac2p}\cap \dot W^{-1+\delta, p'}\cap L^1$
	satisfies the conditions in the statements of Proposition~\ref{prop:5.2} and Proposition~\ref{prop:5.3}. Then there is a constant
		$C_1>0$
	independent of
		$\beta$
	and
		$\omega_0$
	such that if
		$|\beta|^{-\frac{1+\delta}{3}}\left\|\omega_0\right\|_{\dot H^{\delta}}\leq C_1$
	holds, then there is a constant
		$C_2>0$
	dependent only on
		$\mathcal A, $
		$\left\|\omega_0\right\|_{L^1}, $
		$|\beta|^{-\frac\delta3}\left\|\omega_0\right\|_{\dot H^{-1+\delta}}, $
		$|\beta|^{-\frac{1+\delta}3}\left\|\omega_0\right\|_{\dot H^{\delta}}, $
	and
		$|\beta|^{-\frac13(1+\delta-\frac2{p'})}\left\|\omega_0\right\|_{\dot W^{-1+\delta, p'}}$
	satisfying 
		\begin{align}
			|\beta|^{-\frac13(1+\delta-\frac2{p'})}\left\|\omega(t)\right\|_{\dot W^{-1+\delta, p}}
			\leq C_2t^{-1+\frac2p}\min\left\{1, |\beta|^{-1+\frac2p}t^{-\frac32(1-\frac2p)}\right\} 
		\end{align}
	for all
		$t>0.$
\end{lemma}
\proof 
	Define the functions 
		$N_0(t)$
	and
		$\widetilde\Omega_0(t)$
	on
		$(0, \infty)$
	by
		\begin{align}
			N_0(t)\coloneqq t^{-1+\frac2p}\min\left\{1, |\beta|^{-1+\frac2p}t^{-\frac32(1-\frac2p)}\right\}\label{eq:5.80}
		\end{align}
	and
		\begin{align}
			\widetilde\Omega_0(t)\coloneqq \sup_{0<\tau\leq t}N_0(\tau)^{-1}\left\|\omega(\tau)\right\|_{\dot W^{-1+\delta, p}}, \label{eq:5.81}
		\end{align}
	respectively. We show that 
		$\widetilde \Omega_0$
	is bounded on 
		$(0, \infty)$
	if
		$|\beta|^{-\frac{1+\delta}{3}}\left\|\omega_0\right\|_{\dot H^{\delta}}$
	is small enough.
	
	As before, we write
		\begin{align}
			&\quad\left\|\omega(t)\right\|_{\dot W^{-1+\delta,p}}\label{eq:5.82}
			\leq \left\|T_\beta(t)\omega_0\right\|_{\dot W^{-1+\delta,p}}\\
			&+\int_0^\frac t2\left\|T_\beta(t-\tau)\operatorname{div}(u(\tau)\omega(\tau))\right\|_{\dot W^{-1+\delta,p}}\, d\tau
			+\int^t_\frac t2\left\|T_\beta(t-\tau)\operatorname{div}(u(\tau)\omega(\tau))\right\|_{\dot W^{-1+\delta, p}}\, d\tau.\label{eq:5.83}
		\end{align}
	and estimate each term.
	
	For the first term, we easily have
		\begin{align}
			\left\|T_\beta(t)\omega_0\right\|_{\dot W^{-1+\delta,p}}
			\lesssim N_0(t)\left\|\omega_0\right\|_{\dot W^{-1+\delta,p'}}.\label{eq:5.84}
		\end{align}
	
	For the second term, we use the estimate in the proof of Proposition~\ref{prop:2.5} and H\"older's inequality to obtain
		\begin{align}
			&\quad\int_0^\frac t2\left\|T_\beta(t-\tau)\operatorname{div}(u(\tau)\omega(\tau))\right\|_{\dot W^{-1+\delta,p}}\, d\tau\label{eq:5.85}\\
			&\lesssim \int_0^\frac t2N_0(t-\tau)\left\|\omega(\tau)\right\|_{\dot W^{-1+\delta, p}}\left\|\omega(\tau)\right\|_{\dot W^{\delta, p}}\, d\tau\label{eq:5.86}\\
			&\lesssim N_0(t)\Omega_\delta(t)\int_0^\frac t2\left\|\omega(\tau)\right\|_{\dot W^{-1+\delta, p}}M_\delta(\tau)\, d\tau\label{eq:5.87}\\
			&\leq N_0(t)\Omega_\delta(t)\left\|\omega\right\|_{Y_1}\left\|M_\delta\right\|_{L^{r_1'}([0, \infty))}.\label{eq:5.88}
		\end{align}
	The term
		$\left\|M_\delta\right\|_{L^{r_1'}([0, \infty))}$
	is finite if and only if it holds that
		$1-\frac1p+\frac\delta2
		<1-\frac1{r_1}
		<1-\frac1p+\frac32(1-\frac2p)+\frac\delta2.$ The right-hand inequality is obviously true as we have 
		$1-\frac1p+\frac32(1-\frac2p)+\frac\delta2=1+\frac12(1-\frac2p)>1.$
	The left-hand inequality is equivalent to
		$\frac1{r_1}<1-\frac2p.$
	By the assumption on	
		$r_1$, 
	we have
		$\frac1{r_1}
		\leq \frac12-\frac1p+\frac\delta2
		=-\frac12+\frac2p.$
	We see that
		$-\frac12+\frac2p<1-\frac2p$
	is true, since this is equivalent to
		$p>\frac83, $
	which is the case in our setting
		$p\in[\frac{30}{11}, 3].$
	As a result, we obtain
		$\left\|M_\delta\right\|_{L^{r_1'}([0, \infty))}
		\lesssim|\beta|^{-\frac23(\frac1p-\frac\delta2-\frac1{r_1})}$
	and
		\begin{align}
			&\quad\int_0^\frac t2\left\|T_\beta(t-\tau)\operatorname{div}(u(\tau)\omega(\tau))\right\|_{\dot W^{-1+\delta,p}}\, d\tau
			\lesssim N_0(t)\Omega_\delta(t)\left\|\omega\right\|_{Y_1}|\beta|^{-\frac23(\frac1p-\frac\delta2-\frac1{r_1})}\label{eq:5.89}\\
			&\lesssim N_0(t)\Omega_\delta(t)|\beta|^{-\frac\delta3}\left\|\omega_0\right\|_{\dot H^{-1+\delta}}|\beta|^{\frac13(1+\delta-\frac2{p'})}.\label{eq:5.90}
		\end{align}
	For the third term in \eqref{eq:5.83}, similarly we have
		\begin{align}
			&\quad\int^t_\frac t2\left\|T_\beta(t-\tau)\operatorname{div}(u(\tau)\omega(\tau))\right\|_{\dot W^{-1+\delta, p}}\, d\tau\label{eq:5.91}\\
			&\lesssim \int_\frac t2^tN_0(t-\tau)\left\|\omega(\tau)\right\|_{\dot W^{-1+\delta, p}}\left\|\omega(\tau)\right\|_{\dot W^{\delta, p}}\, d\tau\label{eq:5.92}\\
			&\lesssim N_0(t)\widetilde \Omega_0(t)\left\|\omega\right\|_{Y_2}\left\|N_0\right\|_{L^{r_2'}}
			\lesssim N_0(t)\widetilde \Omega_0(t)|\beta|^{-\frac{1+\delta}3}\left\|\omega_0\right\|_{\dot H^\delta}.\label{eq:5.93}
		\end{align}
		
	Therefore, under the assumption that
		$|\beta|^{-\frac{1+\delta}3}\left\|\omega_0\right\|_{\dot H^\delta}$
	is sufficiently small, from the above estimates we obtain
		\begin{align*}
			|\beta|^{-\frac13(1+\delta-\frac2{p'})}\widetilde\Omega_0(t)
			\lesssim |\beta|^{-\frac13(1+\delta-\frac2{p'})}\left\|\omega_0\right\|_{\dot W^{-1+\delta, p'}}
			+|\beta|^{-\frac{\delta}3}\left\|\omega_0\right\|_{\dot H^{-1+\delta}}\Omega_0(t)^{1-\delta}\Omega_1(t)^\delta.
		\end{align*}
	This completes the proof.
\qed

The following proposition gives temporal decay estimates for
	$\dot W^{s, p}$
in the case of higher regularity 
	$s$.
\begin{proposition}\label{prop:5.5}
	Assume that
		$\omega_0\in\dot H^{-1+\delta}\cap\dot H^{3-\frac2p}\cap \dot W^{-1+\delta, p'}\cap L^1$
	satisfies the conditions in the statements of Proposition~\ref{prop:5.2}, Proposition~\ref{prop:5.3} and Lemma~\ref{lem:5.4}. Then for any
		$m\geq0$, 
	there is a positive constant
		\begin{align*}
			C=C(m, \mathcal A, \left\|\omega_0\right\|_{L^1}, |\beta|^{-\frac\delta3}\left\|\omega_0\right\|_{\dot H^{-1+\delta}}, |\beta|^{-\frac{1+\delta}3}\left\|\omega_0\right\|_{\dot H^{\delta}}, |\beta|^{-\frac13(1+\delta-\frac2{p'})}\left\|\omega_0\right\|_{\dot W^{-1+\delta, p'}})
		\end{align*}
	satisfying
		$\Omega_{1+m}(t)\leq C$
	for all
		$t>0.$
\end{proposition}
\proof 
	Take
		$\alpha \in(0, 1)$
	sufficiently small to satisfy
		$\frac1{r_1}<\frac12-\frac1p+\frac\delta2-\frac\alpha2.$
	We show that the desired estimate 
		\begin{align*}
			\Omega_{1+k\alpha}(t)\leq C(\left\|\omega_0\right\|_{L^1}, |\beta|^{-\frac\delta3}\left\|\omega_0\right\|_{\dot H^{-1+\delta}}, |\beta|^{-\frac{1+\delta}3}\left\|\omega_0\right\|_{\dot H^{\delta}}, |\beta|^{-\frac13(1+\delta-\frac2{p'})}\left\|\omega_0\right\|_{\dot W^{-1+\delta, p'}})
		\end{align*}
	holds for all
		$k\in\mathbb N\cup\{0\}$
	and
		$t>0$,
	by the induction of
		$k$.
	The base case
		$k=0$
	is already proved in Proposition~\ref{prop:5.3}.
		
	As usual, we write
		\begin{align}
			&\quad\left\|\omega(t)\right\|_{\dot W^{1+(k+1)\alpha,p}}\label{eq:5.103}
			\leq \left\|T_\beta(t)\omega_0\right\|_{\dot W^{1+(k+1)\alpha,p}}\\
			&+\int_0^\frac t2\left\|T_\beta(t-\tau)\operatorname{div}(u(\tau)\omega(\tau))\right\|_{\dot W^{1+(k+1)\alpha,p}}\, d\tau
			+\int^t_\frac t2\left\|T_\beta(t-\tau)\operatorname{div}(u(\tau)\omega(\tau))\right\|_{\dot W^{1+(k+1)\alpha, p}}\, d\tau.\notag
		\end{align}
	The estimates for the first and second terms can be done similarly to Proposition~\ref{prop:5.3} to obtain
		\begin{align}
			\left\|T_\beta(t)\omega_0\right\|_{\dot W^{1+(k+1)\alpha,p}}
			&\lesssim M_{1+(k+1)\alpha}(t)\left\|\omega_0\right\|_{L^1}\label{eq:5.105}
		\end{align}
	and
		\begin{align}
			&\quad\int_0^\frac t2\left\|T_\beta(t-\tau)\operatorname{div}(u(\tau)\omega(\tau))\right\|_{\dot W^{1+(k+1)\alpha,p}}\, d\tau\label{eq:5.106}\\
			&\lesssim M_{1+(k+1)\alpha}(t)
			\Omega_0(t)^{\theta_0}
			\left\|\omega_0\right\|_{L^1}^\frac{\delta(1-\theta_0)}{1+\delta}\left(|\beta|^{-\frac{1+\delta}{3}} \left\|\omega_0\right\|_{\dot H^{\delta}}\right)^\frac{1-\theta_0}{1+\delta}|\beta|^{-\frac{\delta}{3}}\left\|\omega_0\right\|_{\dot H^{-1+\delta}}, \label{eq:5.108}
		\end{align}
	respectively. 
	
	For the third term, we have
		\begin{align}
			&\quad\int^t_\frac t2\left\|T_\beta(t-\tau)\operatorname{div}(u(\tau)\omega(\tau))\right\|_{\dot W^{1+(k+1)\alpha, p}}\, d\tau\label{eq:5.109}\\
			&\lesssim \int^t_\frac t2(t-\tau)^{-\frac12-\frac1p+\frac\delta2-\frac\alpha2}\min\left\{1, |\beta|^{-1+\frac2p}(t-\tau)^{-\frac32(1-\frac2p)}\right\}\left\|(-\Delta)^\frac{1+k\alpha}2(u(\tau)\omega(\tau))\right\|_{L^{p'}}\, d\tau\notag\\
			&\lesssim \int^t_\frac t2(t-\tau)^{-\frac12-\frac1p+\frac\delta2-\frac\alpha2}\min\left\{1, |\beta|^{-1+\frac2p}(t-\tau)^{-\frac32(1-\frac2p)}\right\}\left\|(-\Delta)^\frac{k\alpha}2\omega(\tau)\right\|_{L^p}\left\|\omega(\tau)\right\|_{\dot W^{\delta, p}}\, d\tau\label{eq:5.111}\\
			&+\int^t_\frac t2(t-\tau)^{-\frac12-\frac1p+\frac\delta2-\frac\alpha2}\min\left\{1, |\beta|^{-1+\frac2p}(t-\tau)^{-\frac32(1-\frac2p)}\right\}\left\|\omega(\tau)\right\|_{\dot W^{-1+\delta, p}}\left\|(-\Delta)^\frac{1+k\alpha}2\omega(\tau)\right\|_{L^{p}}\, d\tau\label{eq:5.112}.
		\end{align}
	Since the interpolation inequality yields
		\begin{align*}
			\left\|(-\Delta)^\frac{k\alpha}2\omega(\tau)\right\|_{L^p}
			\lesssim \left\|\omega(\tau)\right\|_{\dot W^{-1+\delta, p}}^{\frac{1}{2-\delta+k\alpha}}
			\left\|(-\Delta)^\frac{1+k\alpha}2\omega(\tau)\right\|_{L^{p}}^{\frac{1-\delta+k\alpha}{2-\delta+k\alpha}}
		\end{align*}
	and
		\begin{align*}
			\left\|\omega(\tau)\right\|_{\dot W^{\delta, p}}
			\lesssim \left\|\omega(\tau)\right\|_{\dot W^{-1+\delta, p}}^{\frac{1-\delta+k\alpha}{2-\delta+k\alpha}}
			\left\|(-\Delta)^\frac{1+k\alpha}2\omega(\tau)\right\|_{L^{p}}^{\frac{1}{2-\delta+k\alpha}}, 
		\end{align*}
	we have
		\begin{align*}
			\left\|(-\Delta)^\frac{k\alpha}2\omega(\tau)\right\|_{L^p}
			\left\|\omega(\tau)\right\|_{\dot W^{\delta, p}}
			\lesssim \left\|\omega(\tau)\right\|_{\dot W^{-1+\delta, p}}
			\left\|(-\Delta)^\frac{1+k\alpha}2\omega(\tau)\right\|_{L^{p}}. 
		\end{align*}
	Therefore, the estimate of \eqref{eq:5.111} can be reduced to that of \eqref{eq:5.112}.
	We can deduce
		\begin{align}
			&\int^t_\frac t2(t-\tau)^{-\frac12-\frac1p+\frac\delta2-\frac\alpha2}\min\left\{1, |\beta|^{-1+\frac2p}(t-\tau)^{-\frac32(1-\frac2p)}\right\}\left\|\omega(\tau)\right\|_{\dot W^{-1+\delta, p}}\left\|(-\Delta)^\frac{1+k\alpha}2\omega(\tau)\right\|_{L^{p}}\, d\tau\notag\\
			&\lesssim t^\frac\alpha2M_{1+(k+1)\alpha}(t)\Omega_{1+k\alpha}(t)\int^t_\frac t2(t-\tau)^{-\frac12-\frac1p+\frac\delta2-\frac\alpha2}\min\left\{1, |\beta|^{-1+\frac2p}(t-\tau)^{-\frac32(1-\frac2p)}\right\}\left\|\omega(\tau)\right\|_{\dot W^{-1+\delta, p}}\, d\tau\notag\\
			&\lesssim t^\frac\alpha2M_{1+(k+1)\alpha}(t)\Omega_{1+k\alpha}(t)|\beta|^{-\frac23(\frac12-\frac1{r_1}-\frac1p+\frac\delta2-\frac\alpha2)}\left\|\omega\right\|_{Y_1}\\
			&\lesssim \left(|\beta|^{\frac13}t^{\frac12}\right)^\alpha M_{1+(k+1)\alpha}(t)\Omega_{1+k\alpha}(t)\left|\beta\right|^{-\frac{\delta}3}\left\|\omega_0\right\|_{\dot H^{-1+\delta}}\label{eq:5.126}
		\end{align}
	and
		\begin{align}
			&\int^t_\frac t2(t-\tau)^{-\frac12-\frac1p+\frac\delta2-\frac\alpha2}\min\left\{1, |\beta|^{-1+\frac2p}(t-\tau)^{-\frac32(1-\frac2p)}\right\}\left\|\omega(\tau)\right\|_{\dot W^{-1+\delta, p}}\left\|(-\Delta)^\frac{1+k\alpha}2\omega(\tau)\right\|_{L^{p}}\, d\tau\notag\\
			&\lesssim t^{\frac12-\frac1p+\frac\delta2-\frac\alpha2}N_0(t)\widetilde\Omega_0(t)M_{1+k\alpha}(t)\Omega_{1+k\alpha}(t)\\
			&= |\beta|^{\frac13(1+\delta-\frac2{p'})}t^{\frac12-\frac1p+\frac\delta2}N_0(t) \cdot M_{1+(k+1)\alpha}(t)\Omega_{1+k\alpha}(t)\left(|\beta|^{-\frac13(1+\delta-\frac2{p'})}\widetilde \Omega_0(t)\right)\label{eq:5.128} \\
			&\leq \left(|\beta|^{\frac13}t^{\frac12} \right)^{-4+\frac8p+\delta} M_{1+(k+1)\alpha}(t)\Omega_{1+k\alpha}(t)\left(|\beta|^{-\frac13(1+\delta-\frac2{p'})}\widetilde \Omega_0(t)\right).\label{eq:5.129}
		\end{align}
	Note that
		$-4+\frac8p+\delta<0$
	is equivalent to
		$\frac1p<\frac37$
	and this is true in the present setting. Therefore, the interpolation of the above two estimates yields
		\begin{align}
			&\int^t_\frac t2(t-\tau)^{-\frac12-\frac1p+\frac\delta2-\frac\alpha2}\min\left\{1, |\beta|^{-1+\frac2p}(t-\tau)^{-\frac32(1-\frac2p)}\right\}\left\|\omega(\tau)\right\|_{\dot W^{-1+\delta, p}}\left\|(-\Delta)^\frac{1+k\alpha}2\omega(\tau)\right\|_{L^{p}}\, d\tau\notag\\
			&\lesssim  M_{1+(k+1)\alpha}(t)\Omega_{1+k\alpha}(t)
			\left(\left|\beta\right|^{-\frac{\delta}3}\left\|\omega_0\right\|_{\dot H^{-1+\delta}}\right)^{\frac{4-\frac8p-\delta}{\alpha+4-\frac8p-\delta}}
			\left(|\beta|^{-\frac13(1+\delta-\frac2{p'})}\widetilde \Omega_0(t)\right)^{\frac{\alpha}{\alpha+4-\frac8p-\delta}}\notag
		\end{align}
	This gives the desired estimate of \eqref{eq:5.112}.
	
	Combining the above estimates, we obtain the boundedness of
		$\Omega_{1+(k+1)\alpha}$
	on
		$(0, \infty).$
	This completes the proof.
\qed
	
\subsection{Temporal decays of $\dot H^{s}$-norms}
	In this subsection, we show the temporal decays of the
		$L^2$-based 
	Sobolev spaces
		$\dot H^s$
	for
		$s\geq-1.$
\begin{proposition}\label{prop:5.6}
	Assume that
		$\omega_0\in \dot H^{-1+\delta}\cap \dot H^{3-\frac2p}\cap \dot W^{-1+\delta, p'}\cap L^1$
	satisfies the conditions in Proposition~\ref{prop:5.2}, Proposition~\ref{prop:5.3}, and Lemma~\ref{lem:5.4}. Then, for any
		$s\geq0$
	there is a positive constant
		$C$
	dependent only on
		$s, $
		$\mathcal A$, 
		$\left\|\omega_0\right\|_{L^1}, |\beta|^{-\frac\delta3}\left\|\omega_0\right\|_{\dot H^{-1+\delta}},$
		$|\beta|^{-\frac{1+\delta}3}\left\|\omega_0\right\|_{\dot H^{\delta}},$
	and
		$|\beta|^{-\frac13(1+\delta-\frac2{p'})}\left\|\omega_0\right\|_{\dot W^{-1+\delta, p'}}$
	satisfying
		$\Omega_{s, 2}(t)\leq C$
	for all
		$t>0$.
\end{proposition}
\proof
	We divide the estimates into three parts as usual:
		\begin{align}
			&\quad\left\|\omega(t)\right\|_{\dot H^{s}}\label{eq:5.133}
			\leq \left\|T_\beta(t)\omega_0\right\|_{\dot H^{s}}\\
			&+\int_0^\frac t2\left\|T_\beta(t-\tau)\operatorname{div}(u(\tau)\omega(\tau))\right\|_{\dot H^{s}}\, d\tau
			+\int^t_\frac t2\left\|T_\beta(t-\tau)\operatorname{div}(u(\tau)\omega(\tau))\right\|_{\dot H^{s}}\, d\tau.\label{eq:5.134}
		\end{align}
	
	For the linear term, we easily obtain
		\begin{align}
			\left\|T_\beta(t)\omega_0\right\|_{\dot H^s}
			=\left\|e^{t\Delta}\omega_0\right\|_{\dot H^s}
			\lesssim t^{-\frac s2-\frac12}\left\|\omega_0\right\|_{L^1}.\label{eq:5.135}
		\end{align}
	
	For the first half of the integral terms, we have
		\begin{align}
			&\int_0^\frac t2\left\|T_\beta(t-\tau)\operatorname{div}(u(\tau)\omega(\tau))\right\|_{\dot H^{s}}\, d\tau\label{eq:5.136}\\
			&\lesssim \int_0^\frac t2(t-\tau)^{-\frac s2-\frac12-(\frac1{p'}-\frac12)}\left\|u(\tau)\omega(\tau)\right\|_{L^{p'}}\, d\tau\label{eq:5.137}\\
			&\lesssim t^{-\frac s2-\frac12-(\frac1{p'}-\frac12)}\int_0^\frac t2\left\|\omega(\tau)\right\|_{\dot W^{-1+\delta, p}}\left\|\omega(\tau)\right\|_{L^p}\, d\tau\label{eq:5.138}\\
			&\leq t^{-\frac s2-\frac12-(\frac1{p'}-\frac12)}\Omega_0(t)\int_0^\frac t2\left\|\omega(\tau)\right\|_{\dot W^{-1+\delta, p}}M_0(\tau)\, d\tau\label{eq:5.139}\\
			&\leq t^{-\frac s2-\frac12-(\frac1{p'}-\frac12)}\Omega_0(t)\left\|\omega\right\|_{Y_1}\left\|M_0\right\|_{L^{r_1'}((0, \frac t2))}.\label{eq:5.140}
		\end{align}
	Now, define
		$\theta\in(0, 1]$
	so that
		$-\frac1{p'}+\frac12-\frac1{r_1}+\frac1p-\frac32\theta(1-\frac2p)=0$
	holds. This choice is possible, since in our setting we have
		$-\frac1{p'}+\frac12-\frac1{r_1}+\frac1p
		=-\frac1{p}+\frac\delta2+\frac12-\frac1{r_1}>0$
	and
		$-\frac1{p'}+\frac12-\frac1{r_1}+\frac1p-\frac32(1-\frac2p)
		=-\frac1{p}+\frac\delta2+\frac12-\frac1{r_1}-\frac32(1-\frac2p)
		\leq0.$
	Then we obtain
		\begin{align}
			\left\|M_0\right\|_{L^{r_1'}((0, \frac t2))}
			&\leq\left(\int_0^\frac t2 \left(\tau^{-1+\frac1p} \left(|\beta|^{-1+\frac2p}\tau^{-\frac32(1-\frac2p)}\right)^\theta\right)^{r_1'}\, d\tau\right)^{\frac1{r_1'}}\label{eq:5.141}\\
			&\lesssim |\beta|^{-\theta(1-\frac2p)}t^{-\frac1{r_1}+\frac1p-\frac32\theta(1-\frac2p)}\label{eq:5.142}
		\end{align}
	and thus
		\begin{align}
			\int_0^\frac t2\left\|T_\beta(t-\tau)\operatorname{div}(u(\tau)\omega(\tau))\right\|_{\dot H^{s}}\, d\tau
			&\lesssim |\beta|^{-\theta(1-\frac2p)}t^{-\frac s2-\frac12}\Omega_0(t)\left\|\omega(\tau)\right\|_{Y_1}\label{eq:5.143}\\
			&\lesssim t^{-\frac s2-\frac12}\Omega_0(t)|\beta|^{-\frac\delta3}\left\|\omega(\tau)\right\|_{\dot H^{-1+\delta}}.\label{eq:5.144}
		\end{align}
	
	For the second half of the integral terms, we have
		\begin{align}
			&\int^t_\frac t2\left\|T_\beta(t-\tau)\operatorname{div}(u(\tau)\omega(\tau))\right\|_{\dot H^{s}}\, d\tau
			\lesssim \int^t_\frac t2(t-\tau)^{-\frac1{p'}+\frac12}\left\|(-\Delta)^\frac s2(u(\tau)\cdot\nabla\omega(\tau))\right\|_{L^{p'}}\, d\tau.\notag\\
			&\lesssim \int^t_\frac t2(t-\tau)^{-\frac1{p'}+\frac12}\left\|\omega(\tau)\right\|_{\dot W^{-1+\delta+s, p}}\left\|\omega(\tau)\right\|_{\dot W^{1,p}}\, d\tau\label{eq:5.162}\\
			&+\int^t_\frac t2(t-\tau)^{-\frac1{p'}+\frac12}\left\|\omega(\tau)\right\|_{\dot W^{-1+\delta, p}}\left\|\omega(\tau)\right\|_{\dot W^{1+s,p}}\, d\tau\label{eq:5.163}.
		\end{align}
		
	For \eqref{eq:5.162}, the interpolation inequality yields
		\begin{align*}
			\left\|\omega(\tau)\right\|_{\dot W^{-1+\delta+s}}
			&\lesssim
			\left\|\omega(\tau)\right\|_{\dot W^{-1+\delta, p}}^\frac{2-\delta}{2-\delta+s}
			\left\|\omega(\tau)\right\|_{\dot W^{1+s, p}}^\frac{s}{2-\delta+s}
		\end{align*}
	and
		\begin{align*}
			\left\|\omega(\tau)\right\|_{\dot W^{1, p}}
			&\lesssim
			\left\|\omega(\tau)\right\|_{\dot W^{-1+\delta, p}}^\frac{s}{2-\delta+s}
			\left\|\omega(\tau)\right\|_{\dot W^{1+s, p}}^\frac{2-\delta}{2-\delta+s}.
		\end{align*}
	Combining these inequalities, we obtain
		\begin{align*}
			\left\|\omega(\tau)\right\|_{\dot W^{-1+\delta+s, p}}\left\|\omega(\tau)\right\|_{\dot W^{1,p}}
			\lesssim \left\|\omega(\tau)\right\|_{\dot W^{-1+\delta, p}}\left\|\omega(\tau)\right\|_{\dot W^{1+s,p}}.
		\end{align*}
	Therefore, we only have to estimate \eqref{eq:5.163}.
	We have
		\begin{align}
			&\int^t_\frac t2(t-\tau)^{-\frac1{p'}+\frac12}\left\|\omega(\tau)\right\|_{\dot W^{-1+\delta, p}}\left\|\omega(\tau)\right\|_{\dot W^{1+s,p}}\, d\tau\label{eq:5.177}\\
			&\leq M_{1+s}(t)\Omega_{1+s}(t)\left\|\omega\right\|_{Y_1}\left(\int_\frac t2^t(t-\tau)^{-r_1'(\frac 1{p'}-\frac12)}\, d\tau\right)^{\frac1{r_1'}}\label{eq:5.178}.
		\end{align}
	The integral on the right-hand side is finite since the condition
		$\frac1{p'}-\frac12<1-\frac1{r_1}$
	is equivalent to
		$\frac1{r_1}<\frac12-\frac1p+\frac\delta2+(1-\frac1p)$, 
	and this is true. Thus we obtain
		\begin{align}
			&\int^t_\frac t2(t-\tau)^{-\frac1{p'}+\frac12}\left\|\omega(\tau)\right\|_{\dot W^{-1+\delta, p}}\left\|\omega(\tau)\right\|_{\dot W^{1+s,p}}\, d\tau\\
			&\lesssim t^{1-\frac1{r_1}-\frac1{p'}+\frac12}M_{1+s}(t)\Omega_{1+s}(t)\left\|\omega\right\|_{Y_1}\\
			&\lesssim |\beta|^{\frac13(1+\delta-\frac2p-\frac2{r_1})}t^{-\frac s2-\frac12-\frac1{r_1}-\frac1{p}+\frac\delta2+\frac12}\Omega_{1+s}(t)|\beta|^{-\frac\delta3}\left\|\omega_0\right\|_{\dot H^{-1+\delta}}\label{eq:5.179}\\
			&=\left(|\beta|^\frac13t^\frac12\right)^{1+\delta-\frac2p-\frac2{r_1}}t^{-\frac s2-\frac12}\Omega_{1+s}(t)|\beta|^{-\frac\delta3}\left\|\omega_0\right\|_{\dot H^{-1+\delta}}\label{eq:5.180}.
		\end{align}
	We also have
		\begin{align}
			&\int^t_\frac t2(t-\tau)^{-\frac1{p'}+\frac12}\left\|\omega(\tau)\right\|_{\dot W^{-1+\delta, p}}\left\|\omega(\tau)\right\|_{\dot W^{1+s,p}}\, d\tau\label{eq:5.181}\\
			&\lesssim t^{\frac32-\frac1{p'}}N_0(t)\widetilde \Omega_0(t)M_{1+s}(t)\Omega_{1+s}(t)\label{eq:5.182}\\
			&\leq|\beta|^{\frac13(-7+\delta+\frac{14}{p})}t^{-\frac s2-4+\frac\delta2+\frac{7}p}|\beta|^{-\frac13(1+\delta-\frac2{p'})}\widetilde \Omega_0(t)\Omega_{1+s}(t)\label{eq:5.183}\\
			&=\left(|\beta|^\frac13t^\frac12\right)^{-7+\delta+\frac{14}{p}}t^{-\frac s2-\frac12}|\beta|^{-\frac13(1+\delta-\frac2{p'})}\widetilde \Omega_0(t)\Omega_{1+s}(t).\label{eq:5.184}
		\end{align}
	Since the conditions 
		$\frac1{r_1}
		<\frac12-\frac1p+\frac\delta2$
	and
		$-7+\delta+\frac{14}p<0$
	hold, interpolating these estimates yields
		\begin{align}
			&\int^t_\frac t2(t-\tau)^{-\frac1{p'}+\frac12}\left\|\omega(\tau)\right\|_{\dot W^{-1+\delta, p}}\left\|\omega(\tau)\right\|_{\dot W^{1+s,p}}\, d\tau\label{eq:5.185}\\
			&\lesssim t^{-\frac s2-\frac12}\Omega_{1+s}(t)\left(|\beta|^{-\frac\delta3}\left\|\omega_0\right\|_{\dot H^{-1+\delta}}\right)^\frac{7-\delta-\frac{14}p}{8-\frac{16}p-\frac2{r_1}}\left(|\beta|^{-\frac13(1+\delta-\frac2{p'})}\widetilde \Omega_0(t)\right)^\frac{1+\delta-\frac2p-\frac2{r_1}}{8-\frac{16}p-\frac2{r_1}}.\label{eq:5.186}
		\end{align}
	Combining the above estimates, we obtain the boundedness of
		$\Omega_{s, 2}$
	on
		$(0, \infty).$
	This completes the proof.
\qed

\begin{lemma}\label{lem:5.7}
		Assume that
		$\omega_0\in \dot H^{-1+\delta}\cap \dot H^{3-\frac2p}\cap \dot W^{-1+\delta, p'}\cap L^1$
	satisfies the conditions in Proposition~\ref{prop:5.2}, Proposition~\ref{prop:5.3}, and Lemma~\ref{lem:5.4}. Then, there is a constant
		$C>0$
	dependent only on 
		$\mathcal A, $
		$\left\|\omega_0\right\|_{L^1}, $
		$|\beta|^{-\frac\delta3}\left\|\omega_0\right\|_{\dot H^{-1+\delta}}, $
		$|\beta|^{-\frac{1+\delta}3}\left\|\omega_0\right\|_{\dot H^{\delta}}, $
	and 
		$|\beta|^{-\frac13(1+\delta-\frac2{p'})}\left\|\omega_0\right\|_{\dot W^{-1+\delta, p'}}$
	satisfying
		$\left\|\omega(t)\right\|_{\dot H^{-1}}\leq C$
	for all
		$t>0$.
\end{lemma}
\proof
	As usual, we divide the estimates into three parts:
		\begin{align}
			&\quad\left\|\omega(t)\right\|_{\dot H^{-1}}\label{eq:5.187}
			\leq \left\|T_\beta(t)\omega_0\right\|_{\dot H^{-1}}\\
			&+\int_0^\frac t2\left\|T_\beta(t-\tau)\operatorname{div}(u(\tau)\omega(\tau))\right\|_{\dot H^{-1}}\, d\tau
			+\int^t_\frac t2\left\|T_\beta(t-\tau)\operatorname{div}(u(\tau)\omega(\tau))\right\|_{\dot H^{-1}}\, d\tau.\label{eq:5.188}
		\end{align}
		
	For the first term, it suffices to use the embedding
		$\dot W^{-1+\delta, p'}\cap \dot H^{-1+\delta}\hookrightarrow\dot H^{-1}$
	to obtain the estimate. In fact, the Sobolev embedding ensures that
		$\dot W^{-1+\delta, p'}\hookrightarrow\dot W^{-1, \frac{p}{2(p-2)}}$
	and 
		$\dot H^{-1+\delta}\hookrightarrow \dot W^{-1, \frac2{1-\delta}}.$
	We can also check the condition
		$\frac{p}{2(p-2)}<2.$
	Thus we have
		\begin{align}
			&\left\|T_\beta(t)\omega_0\right\|_{\dot H^{-1}}
			\leq \left\|\omega_0\right\|_{\dot H^{-1}}
			\leq \left\|\omega_0\right\|_{\dot W^{-1, \frac{p}{2(p-2)}}}^{\frac\delta{3+\delta-\frac8p}}
			\left\|\omega_0\right\|_{\dot W^{-1, \frac2{1-\delta}}}^{\frac{3-\frac8p}{3+\delta-\frac8p}}\label{eq:5.189}\\
			&\lesssim \left(|\beta|^{-\frac13(1+\delta-\frac2{p'})}\left\|\omega_0\right\|_{\dot W^{-1+\delta, p'}}\right)^{\frac\delta{3+\delta-\frac8p}}
			\left(|\beta|^{-\frac\delta3}\left\|\omega_0\right\|_{\dot H^{-1+\delta}}\right)^{\frac{3-\frac8p}{3+\delta-\frac8p}}
			<\infty.\label{eq:5.190}
		\end{align}
	
	For the second term, it is possible to set
		$s=-1$
	in the estimate of the corresponding term in Proposition~\ref{prop:5.6}. Thus we obtain
		\begin{align}
			\int_0^\frac t2\left\|T_\beta(t-\tau)\operatorname{div}(u(\tau)\omega(\tau))\right\|_{\dot H^{-1}}\, d\tau
			\lesssim \Omega_0(t)|\beta|^{-\frac\delta3}\left\|\omega(\tau)\right\|_{\dot H^{-1+\delta}}.\label{eq:5.191}
		\end{align}
		
	For the last term, we write
		\begin{align}
			&\int^t_\frac t2\left\|T_\beta(t-\tau)\operatorname{div}(u(\tau)\omega(\tau))\right\|_{\dot H^{-1}}\, d\tau
			\lesssim \int^t_\frac t2(t-\tau)^{-\frac1{p'}+\frac12}\left\|u(\tau)\omega(\tau)\right\|_{L^{p'}}\, d\tau\label{eq:5.192}\\
			&\lesssim \int^t_\frac t2(t-\tau)^{-\frac1{p'}+\frac12}\left\|\omega(\tau)\right\|_{\dot W^{-1+\delta, p}}\left\|\omega(\tau)\right\|_{L^{p}}\, d\tau,\label{eq:5.193}
		\end{align}
	Then it is possible to set
		$s=-1$
	in the estimate of \eqref{eq:5.163} to obtain
		\begin{align}
			&\int^t_\frac t2(t-\tau)^{-\frac1{p'}+\frac12}\left\|\omega(\tau)\right\|_{\dot W^{-1+\delta, p}}\left\|\omega(\tau)\right\|_{L^p}\, d\tau.\label{eq:5.194}\\
			&\lesssim \Omega_{0}(t)\left(|\beta|^{-\frac\delta3}\left\|\omega_0\right\|_{\dot H^{-1+\delta}}\right)^\frac{7-\delta-\frac{14}p}{8-\frac{16}p-\frac2{r_1}}\left(|\beta|^{-\frac13(1+\delta-\frac2{p'})}\widetilde \Omega_0(t)\right)^\frac{1+\delta-\frac2p-\frac2{r_1}}{8-\frac{16}p-\frac2{r_1}}.\label{eq:5.195}
		\end{align}
	This completes the proof.
\qed

\subsection{Temporal decays of $\dot W^{s, \infty}$-norms}
	We divide the proof of the decay estimates of
		$\dot W^{s, \infty}$-norms
	into several steps.
\begin{lemma}\label{lem:5.8}
		Assume that
		$\omega_0\in \dot H^{-1+\delta}\cap \dot H^{3-\frac2p}\cap \dot W^{-1+\delta, p'}\cap L^1$
	satisfies the conditions in Proposition~\ref{prop:5.2}, Proposition~\ref{prop:5.3}, and Lemma~\ref{lem:5.4}. Then, for any 
		$s\geq0$
	there is a constant
		$C>0$
	dependent only on 
		$s$, 
		$\mathcal A$, 
		$\left\|\omega_0\right\|_{L^1}$, 
		$|\beta|^{-\frac\delta3}\left\|\omega_0\right\|_{\dot H^{-1+\delta}}$, 
		$|\beta|^{-\frac{1+\delta}3}\left\|\omega_0\right\|_{\dot H^{\delta}}$, 
	and 
		$|\beta|^{-\frac13(1+\delta-\frac2{p'})}\left\|\omega_0\right\|_{\dot W^{-1+\delta, p'}}$
	satisfying
		\begin{align}
			\left\|T_\beta(t)\omega_0\right\|_{\dot W^{s, \infty}}
			+\int^\frac t2_0\left\|T_\beta(t-\tau)\operatorname{div}(u(\tau)\omega(\tau))\right\|_{\dot W^{s, \infty}}\, d\tau
			\leq CM_{s, \infty}(t)\label{eq:5.196}
		\end{align}
	for all
		$t>0$.
\end{lemma}
\proof
	For the first term on the left-hand side of the inequality, it suffices to apply Proposition~\ref{cor:2.3} to obtain
			$\left\|T_\beta(t)\omega_0\right\|_{\dot W^{s, \infty}}
			\lesssim M_{s, \infty}(t)\left\|\omega_0\right\|_{L^1}.$

	For the second term Proposition~\ref{cor:2.3} yields
		\begin{align}
			&\int^\frac t2_0\left\|T_\beta(t-\tau)\operatorname{div}(u(\tau)\omega(\tau))\right\|_{\dot W^{s, \infty}}\, d\tau\label{eq:5.197}\\
			&\lesssim \int^\frac t2_0 (t-\tau)^{-\frac12}M_{s, \infty}(t-\tau)\left\|u(\tau)\omega(\tau)\right\|_{L^1}\, d\tau\label{eq:5.198}\\
			&\lesssim t^{-\frac12}M_{s, \infty}(t)\int^\frac t2_0 \left\|u(\tau)\right\|_{L^2}\left\|\omega(\tau)\right\|_{L^2}\, d\tau\label{eq:5.199}\\
			&\lesssim t^{-\frac12}M_{s, \infty}(t)\left\|\omega\right\|_{L^\infty([0, \infty); \dot H^{-1})}\Omega_{0, 2}(t)\int^\frac t2_0 \tau^{-\frac12}\, d\tau\label{eq:5.200}\\
			&\lesssim M_{s, \infty}(t)\left\|\omega\right\|_{L^\infty([0, \infty); \dot H^{-1})}\Omega_{0, 2}(t).\label{eq:5.201}
		\end{align}
	This completes the proof.
\qed

We prepare the estimates of the auxiliary 
	$\dot W^{s, 4}$-norms
for
	$s\geq0$
in order to obtain the estimate for 
	$\displaystyle \int_\frac t2^t\left\|T_\beta(t-\tau)\operatorname{div}(u(\tau)\omega(\tau))\right\|_{\dot W^{s, \infty}}\, d\tau.$
\begin{lemma}\label{lem:5.9}
	Assume that
		$\omega_0\in \dot H^{-1+\delta}\cap \dot H^{3-\frac2p}\cap \dot W^{-1+\delta, p'}\cap L^1$
	satisfies the conditions in Proposition~\ref{prop:5.2}, Proposition~\ref{prop:5.3}, and Lemma~\ref{lem:5.4}. Then, for any 
		$s\geq0$
	there is a constant
		$C>0$
	dependent only on 
		$s$, 
		$\mathcal A$, 
		$\left\|\omega_0\right\|_{L^1}, $
		$|\beta|^{-\frac\delta3}\left\|\omega_0\right\|_{\dot H^{-1+\delta}}, $
		$|\beta|^{-\frac{1+2\delta}{9}}\left\|\omega_0\right\|_{\dot H^{\frac23(-1+\delta)}}, $
		$|\beta|^{-\frac{1+\delta}3}\left\|\omega_0\right\|_{\dot H^{\delta}}, $
		$|\beta|^{-\frac{4+2\delta}9}\left\|\omega_0\right\|_{\dot H^{\frac{1+2\delta}{3}}}, $
		$|\beta|^{-\frac{2}3}\left\|\omega_0\right\|_{\dot H^{1}},$
		$|\beta|^{-1}\left\|\omega_0\right\|_{\dot H^{2}},$
	and 
		$|\beta|^{-\frac13(1+\delta-\frac2{p'})}\left\|\omega_0\right\|_{\dot W^{-1+\delta, p'}}$
	satisfying
		$\Omega_{s, 4}(t)\leq C$
	for all
		$t>0$.
\end{lemma}
\proof
	Proposition~\ref{cor:2.3} gives the estimate
		$\left\|T_\beta(t)\omega_0\right\|_{\dot W^{s, 4}}\lesssim M_{s, 4}(t)\left\|\omega_0\right\|_{L^1}.$
	By interpolating the estimates obtained by Proposition~\ref{prop:5.6} and Lemma~\ref{lem:5.8}, we also see that
	\begin{align}
		&\int^\frac t2_0\left\|T_\beta(t-\tau)\operatorname{div}(u(\tau)\omega(\tau))\right\|_{\dot W^{s, 4}}\, d\tau\label{eq:5.202}\\
		&\lesssim M_{s, 4}(t)\left(\Omega_0(t)|\beta|^{-\frac\delta3}\left\|\omega(\tau)\right\|_{\dot H^{-1+\delta}}\right)^\frac12\left(\left\|\omega\right\|_{L^\infty([0, \infty); \dot H^{-1})}\Omega_{0, 2}(t)\right)^\frac12\label{eq:5.203}.
	\end{align}
	Therefore, it suffices to show the estimate of
		\begin{align}
			\int^t_\frac t2\left\|T_\beta(t-\tau)\operatorname{div}(u(\tau)\omega(\tau))\right\|_{\dot W^{s, 4}}\, d\tau\label{eq:5.204}.
		\end{align}
	By Proposition~\ref{prop:2.4}, we have
		\begin{align}
			&\int^t_\frac t2\left\|T_\beta(t-\tau)\operatorname{div}(u(\tau)\omega(\tau))\right\|_{\dot W^{s, 4}}\, d\tau\label{eq:5.205}\\
			&\lesssim \int^t_\frac t2(t-\tau)^{-(\frac12-\frac14)}\left\|(-\Delta)^{\frac s2}(u(\tau)\cdot\nabla\omega(\tau))\right\|_{L^2}\, d\tau\label{eq:5.206}\\
			&\lesssim \int^t_\frac t2(t-\tau)^{-\frac14}\left\|(-\Delta)^{\frac s2}u(\tau)\right\|_{L^{\frac{2p}{p-2}}}\left\|\nabla\omega(\tau)\right\|_{L^p}\, d\tau\label{eq:5.207}\\
			&+  \int^t_\frac t2(t-\tau)^{-\frac14}\left\|u(\tau)\right\|_{L^{\frac{2p}{p-2}}}\left\|(-\Delta)^{\frac s2}\nabla\omega(\tau)\right\|_{L^p}\, d\tau.\label{eq:5.208}
		\end{align}
	Since we can write
		$\frac{p-2}{2p}=\frac1p-\frac12(\frac4p-1)$
	and the condition
		$\frac4p-1>0$
	holds, the Sobolev embedding gives the estimate
		\begin{align}
			\left\|(-\Delta)^{\frac s2}u(\tau)\right\|_{L^{\frac{2p}{p-2}}}
			\lesssim \left\|(-\Delta)^{\frac s2}u(\tau)\right\|_{\dot W^{\frac4p-1, p}}
			\lesssim \left\|\omega(\tau)\right\|_{\dot W^{s+\frac4p-2, p}}.\label{eq:5.209}
		\end{align}
	Noting the equality
		$s+\frac4p-2
		=\frac{3s}3+\frac23(-1+\delta)$, 
	we can also use the interpolation inequality to obtain
		\begin{align}
			\left\|(-\Delta)^{\frac s2}u(\tau)\right\|_{L^{\frac{2p}{p-2}}}
			\lesssim \left\|\omega(\tau)\right\|_{\dot W^{3s, p}}^\frac13\left\|\omega(\tau)\right\|_{\dot W^{-1+\delta, p}}^\frac23.\label{eq:5.210}
		\end{align}
	Therefore, for \eqref{eq:5.207}, we have
		\begin{align}
			& \int^t_\frac t2(t-\tau)^{-\frac14}\left\|(-\Delta)^{\frac s2}u(\tau)\right\|_{L^{\frac{2p}{p-2}}}\left\|\nabla\omega(\tau)\right\|_{L^p}\, d\tau\label{eq:5.211}\\
			&\lesssim \int^t_\frac t2(t-\tau)^{-\frac14}\left\|\omega(\tau)\right\|_{\dot W^{3s, p}}^\frac13\left\|\omega(\tau)\right\|_{\dot W^{-1+\delta, p}}^\frac23\left\|\nabla\omega(\tau)\right\|_{L^p}\, d\tau\label{eq:5.212}\\
			&\lesssim M_{3s}(t)^\frac13\Omega_{3s}(t)^\frac13 N_0(t)^\frac23\widetilde\Omega_0(t)^\frac23 t^{1-\frac1{r_1}-\frac14}\left\|\omega\right\|_{L^{r_1}([0, \infty); \dot W^{1, p})}\label{eq:5.213}\\
			&\lesssim |\beta|^{\frac29(1+\delta-\frac2{p'})+\frac13(3-\frac2p-\frac2{r_1})}t^{\frac34-\frac1{r_1}-\frac{s}{2}-\frac1{3p}-1+\frac2p}\Omega_{3s}(t)^\frac13\left(|\beta|^{-\frac13(1+\delta-\frac2{p'})}\widetilde\Omega_0(t)\right)^\frac23\label{eq:5.214}\\
			&\quad\cdot
			\left(|\beta|^{-\frac{\delta}{3}}\left\|\omega_0\right\|_{\dot H^{-1+\delta}}
			+|\beta|^{-\frac{2}{3}}\left\|\omega_0\right\|_{\dot H^{1}}
			+|\beta|^{-\frac{\delta}{3}}\left\|\omega_0\right\|_{\dot H^{-1+\delta}}
			|\beta|^{-1}\left\|\omega_0\right\|_{\dot H^2}\right)\label{eq:5.215}\\
			&=\left(|\beta|^\frac13t^\frac12\right)^{\frac{10}{3p}+1-\frac2{r_1}}t^{-\frac s2-\frac34}\Omega_{3s}(t)^\frac13\left(|\beta|^{-\frac13(1+\delta-\frac2{p'})}\widetilde\Omega_0(t)\right)^\frac23\label{eq:5.216}\\
			&\quad\cdot
			\left(|\beta|^{-\frac{\delta}{3}}\left\|\omega_0\right\|_{\dot H^{-1+\delta}}
			+|\beta|^{-\frac{2}{3}}\left\|\omega_0\right\|_{\dot H^{1}}
			+|\beta|^{-\frac{\delta}{3}}\left\|\omega_0\right\|_{\dot H^{-1+\delta}}
			|\beta|^{-1}\left\|\omega_0\right\|_{\dot H^2}\right)\label{eq:5.217}.
		\end{align}
	We can also deduce 
		\begin{align}
			& \int^t_\frac t2(t-\tau)^{-\frac14}\left\|(-\Delta)^{\frac s2}u(\tau)\right\|_{L^{\frac{2p}{p-2}}}\left\|\nabla\omega(\tau)\right\|_{L^p}\, d\tau\label{eq:5.218}\\
			&\lesssim \int^t_\frac t2(t-\tau)^{-\frac14}\left\|\omega(\tau)\right\|_{\dot W^{3s, p}}^\frac13\left\|\omega(\tau)\right\|_{\dot W^{-1+\delta, p}}^\frac23\left\|\nabla\omega(\tau)\right\|_{L^p}\, d\tau\label{eq:5.219}\\
			&\lesssim M_{3s}(t)^\frac13\Omega_{3s}(t)^\frac13 N_0(t)^\frac23\widetilde\Omega_0(t)^\frac23 M_1(t)\Omega_1(t)t^{1-\frac14}\label{eq:5.220}\\
			&\lesssim |\beta|^{\frac29(1+\delta-\frac2{p'})-2+\frac4p}t^{\frac34-\frac s2-\frac1{3p}-\frac52(1-\frac2p)-\frac12-1+\frac1p-\frac32(1-\frac2p)}\label{eq:5.221}\\
			&\quad\cdot\Omega_{3s}(t)^\frac13\left(|\beta|^{-\frac13(1+\delta-\frac2{p'})}\widetilde\Omega_0(t)\right)^\frac23\Omega_1(t)\label{eq:5.222}\\
			&=\left(|\beta|^\frac13t^\frac12\right)^{-8+\frac{52}{3p}}t^{-\frac s2-\frac34}
			\Omega_{3s}(t)^\frac13\left(|\beta|^{-\frac13(1+\delta-\frac2{p'})}\widetilde\Omega_0(t)\right)^\frac23\Omega_1(t).\label{eq:5.223}
		\end{align}
	In our setting, the inequality
		$-8+\frac{52}{3p}<-\frac32$
	holds. Thus by interpolating these estimates, we obtain
		\begin{align}
			& \int^t_\frac t2(t-\tau)^{-\frac14}\left\|(-\Delta)^{\frac s2}u(\tau)\right\|_{L^{\frac{2p}{p-2}}}\left\|\nabla\omega(\tau)\right\|_{L^p}\, d\tau\label{eq:5.224}\\
			&\lesssim t^{-\frac s2-\frac34}\Omega_{3s}(t)^\frac13\left(|\beta|^{-\frac13(1+\delta-\frac2{p'})}\widetilde\Omega_0(t)\right)^\frac23\Omega_1(t)^\frac{1+\frac{10}{3p}-\frac2{r_1}}{9-\frac{14}{p}-\frac2{r_1}}\label{eq:5.225}\\
			&\quad\cdot
			\left(|\beta|^{-\frac{\delta}{3}}\left\|\omega_0\right\|_{\dot H^{-1+\delta}}
			+|\beta|^{-\frac{2}{3}}\left\|\omega_0\right\|_{\dot H^{1}}
			+|\beta|^{-\frac{\delta}{3}}\left\|\omega_0\right\|_{\dot H^{-1+\delta}}
			|\beta|^{-1}\left\|\omega_0\right\|_{\dot H^2}\right)^\frac{8-\frac{52}{3p}}{9-\frac{14}{p}-\frac2{r_1}}\label{eq:5.226}
		\end{align}
	and
		\begin{align}
			& \int^t_\frac t2(t-\tau)^{-\frac14}\left\|(-\Delta)^{\frac s2}u(\tau)\right\|_{L^{\frac{2p}{p-2}}}\left\|\nabla\omega(\tau)\right\|_{L^p}\, d\tau\label{eq:5.227}\\
			&\lesssim |\beta|^{-\frac12}t^{-\frac s2-\frac32}\Omega_{3s}(t)^\frac13\left(|\beta|^{-\frac13(1+\delta-\frac2{p'})}\widetilde\Omega_0(t)\right)^\frac23\Omega_1(t)^\frac{\frac{5}2+\frac{10}{3p}-\frac2{r_1}}{9-\frac{14}{p}-\frac2{r_1}}\label{eq:5.228}\\
			&\quad\cdot
			\left(|\beta|^{-\frac{\delta}{3}}\left\|\omega_0\right\|_{\dot H^{-1+\delta}}
			+|\beta|^{-\frac{2}{3}}\left\|\omega_0\right\|_{\dot H^{1}}
			+|\beta|^{-\frac{\delta}{3}}\left\|\omega_0\right\|_{\dot H^{-1+\delta}}
			|\beta|^{-1}\left\|\omega_0\right\|_{\dot H^2}\right)^\frac{\frac{13}2-\frac{52}{3p}}{9-\frac{14}{p}-\frac2{r_1}}\label{eq:5.229}.
		\end{align}
	Combining them, we have the estimate for \eqref{eq:5.207}.
	
	For \eqref{eq:5.208}, we have
		\begin{align}
			&\int^t_\frac t2(t-\tau)^{-\frac14}\left\|u(\tau)\right\|_{L^{\frac{2p}{p-2}}}\left\|(-\Delta)^{\frac s2}\nabla\omega(\tau)\right\|_{L^p}\, d\tau\label{eq:5.230}\\
			&\lesssim M_{1+s}(t)\Omega_{1+s}(t)\left\|\omega\right\|_{L^{r_1}([0, \infty); \dot W^{\frac23(-1+\delta), p})}t^{1-\frac1{r_1}-\frac14}\label{eq:5.231}\\
			&\lesssim |\beta|^{\frac13(1+\delta+\frac13(1-\delta)-\frac2p-\frac2{r_1})}t^{-\frac s2-\frac12-1+\frac1p+\frac34-\frac1{r_1}}\Omega_{1+s}(t)\label{eq:5.232}\\
			&\quad\cdot\bigg(|\beta|^{-\frac{\delta}{3}}\left\|\omega_0\right\|_{\dot H^{-1+\delta}}
			+|\beta|^{-\frac{\delta+\frac13(1-\delta)}{3}}\left\|\omega_0\right\|_{\dot H^{-1+\delta+\frac13(1-\delta)}}\\
			&\qquad\qquad+|\beta|^{-\frac{\delta}{3}}\left\|\omega_0\right\|_{\dot H^{-1+\delta}}
			|\beta|^{-\frac13(1+\delta+\frac13(1-\delta))}\left\|\omega_0\right\|_{\dot H^{\delta+\frac13(1-\delta)}}\bigg)\label{eq:5.233}\\
			&= \left(|\beta|^\frac13t^\frac12\right)^{\frac2p-\frac2{r_1}}t^{-\frac s2-\frac34}\Omega_{1+s}(t)\label{eq:5.234}\\
			&\quad\cdot
			\bigg(|\beta|^{-\frac{\delta}{3}}\left\|\omega_0\right\|_{\dot H^{-1+\delta}}
			+|\beta|^{-\frac{1+2\delta}{9}}\left\|\omega_0\right\|_{\dot H^{\frac23(-1+\delta)}}+|\beta|^{-\frac{\delta}{3}}\left\|\omega_0\right\|_{\dot H^{-1+\delta}}
			|\beta|^{-\frac{4+2\delta}9}\left\|\omega_0\right\|_{\dot H^{\frac{1+2\delta}{3}}}\bigg).\notag
		\end{align}
	Note that the condition
		$\frac1{r_1}<\frac12-\frac1p+\frac\delta2$
	implies 
		$\frac1p>\frac1{r_1}.$
	We also have
		\begin{align}
			&\int^t_\frac t2(t-\tau)^{-\frac14}\left\|u(\tau)\right\|_{L^{\frac{2p}{p-2}}}\left\|(-\Delta)^{\frac s2}\nabla\omega(\tau)\right\|_{L^p}\, d\tau\label{eq:5.236}\\
			&\lesssim \int^t_\frac t2(t-\tau)^{-\frac14}\left\|\omega(\tau)\right\|_{L^{p}}^\frac13\left\|\omega(\tau)\right\|_{\dot W^{-1+\delta, p}}^\frac23\left\|\omega(\tau)\right\|_{\dot W^{s+1, p}}\, d\tau\label{eq:5.237}\\
			&\lesssim M_0(t)^\frac13\Omega_0(t)^\frac13N_0(t)^\frac23\widetilde\Omega_0(t)^\frac23M_{1+s}(t)\Omega_{1+s}(t)t^\frac34\label{eq:5.238}\\
			&\lesssim |\beta|^{\frac29(1+\delta-\frac2{p'})-2+\frac4p}t^{\frac34-\frac1{3p}-\frac52(1-\frac2p)-\frac s2-\frac12-1+\frac1p-\frac32(1-\frac2p)}\label{eq:5.239}\\
			&\quad\cdot\Omega_0(t)^\frac13\left(|\beta|^{-\frac13(1+\delta-\frac2{p'})}\widetilde\Omega_0(t)\right)^\frac23\Omega_{1+s}(t)\label{eq:5.240}\\
			&=\left(|\beta|^\frac13t^\frac12\right)^{-8+\frac{52}{3p}}t^{-\frac s2-\frac34}\Omega_0(t)^\frac13\left(|\beta|^{-\frac13(1+\delta-\frac2{p'})}\widetilde\Omega_0(t)\right)^\frac23\Omega_{1+s}(t).\label{eq:5.241}
		\end{align}
	The interpolation of these two estimates yields
		\begin{align}
			&\int^t_\frac t2(t-\tau)^{-\frac14}\left\|u(\tau)\right\|_{L^{\frac{2p}{p-2}}}\left\|(-\Delta)^{\frac s2}\nabla\omega(\tau)\right\|_{L^p}\, d\tau\label{eq:5.242}\\
			&\lesssim t^{-\frac s2-\frac34}\Omega_{1+s}(t)\left(\Omega_0(t)^\frac13\left(|\beta|^{-\frac13(1+\delta-\frac2{p'})}\widetilde\Omega_0(t)\right)^\frac23\right)^\frac{\frac2p-\frac2{r_1}}{8-\frac{46}{3p}-\frac2{r_1}}\label{eq:5.243}\\
			&\quad\cdot\bigg(|\beta|^{-\frac{\delta}{3}}\left\|\omega_0\right\|_{\dot H^{-1+\delta}}
			+|\beta|^{-\frac{1+2\delta}{9}}\left\|\omega_0\right\|_{\dot H^{\frac23(-1+\delta)}}\\
			&\phantom{\bigg(|\beta|^{-\frac{\delta}{3}}\left\|\omega_0\right\|_{\dot H^{-1+\delta}}}+|\beta|^{-\frac{\delta}{3}}\left\|\omega_0\right\|_{\dot H^{-1+\delta}}
			|\beta|^{-\frac{4+2\delta}9}\left\|\omega_0\right\|_{\dot H^{\frac{1+2\delta}{3}}}\bigg)^\frac{8-\frac{52}{3p}}{8-\frac{46}{3p}-\frac2{r_1}}\label{eq:5.244}
		\end{align}
	and
		\begin{align}
			&\int^t_\frac t2(t-\tau)^{-\frac14}\left\|u(\tau)\right\|_{L^{\frac{2p}{p-2}}}\left\|(-\Delta)^{\frac s2}\nabla\omega(\tau)\right\|_{L^p}\, d\tau\label{eq:5.245}\\
			&\lesssim |\beta|^{-\frac12}t^{-\frac s2-\frac32}\Omega_{1+s}(t)\left(\Omega_0(t)^\frac13\left(|\beta|^{-\frac13(1+\delta-\frac2{p'})}\widetilde\Omega_0(t)\right)^\frac23\right)^\frac{\frac{3}2+\frac{2}{p}-\frac2{r_1}}{8-\frac{46}{3p}-\frac2{r_1}}\label{eq:5.246}\\
			&\quad\cdot\bigg(|\beta|^{-\frac{\delta}{3}}\left\|\omega_0\right\|_{\dot H^{-1+\delta}}
			+|\beta|^{-\frac{1+2\delta}{9}}\left\|\omega_0\right\|_{\dot H^{\frac23(-1+\delta)}}\\
			&\phantom{\bigg(|\beta|^{-\frac{\delta}{3}}\left\|\omega_0\right\|_{\dot H^{-1+\delta}}}+|\beta|^{-\frac{\delta}{3}}\left\|\omega_0\right\|_{\dot H^{-1+\delta}}
			|\beta|^{-\frac{4+2\delta}9}\left\|\omega_0\right\|_{\dot H^{\frac{1+2\delta}{3}}}\bigg)^\frac{\frac{13}2-\frac{52}{3p}}{8-\frac{46}{3p}-\frac2{r_1}}.\label{eq:5.247}
		\end{align}
		These are sufficient to obtain the estimate of \eqref{eq:5.208}.
\qed

\begin{proposition}\label{prop:5.10}
	Assume that
		$0\leq \delta<\frac{1}{13}$
	and 
		$\omega_0\in \dot H^{-1+\delta}\cap \dot H^{3-\frac2p}\cap \dot W^{-1+\delta, p'}\cap L^1$
	satisfies the conditions in Proposition~\ref{prop:5.2}, Proposition~\ref{prop:5.3}, and Lemma~\ref{lem:5.4}. Then, for any 
		$s\geq0$
	there is a constant
		$C>0$
	dependent only on 
		$s, $
		$\mathcal A, $
		$\left\|\omega_0\right\|_{L^1}, $
		$|\beta|^{-\frac\delta3}\left\|\omega_0\right\|_{\dot H^{-1+\delta}}$, 
		$|\beta|^{-\frac{1+2\delta}{9}}\left\|\omega_0\right\|_{\dot H^{\frac23(-1+\delta)}}, $
		$|\beta|^{-\frac{1+\delta}3}\left\|\omega_0\right\|_{\dot H^{\delta}}$, 
		$|\beta|^{-\frac{4+2\delta}9}\left\|\omega_0\right\|_{\dot H^{\frac{1+2\delta}{3}}}, $
		$|\beta|^{-\frac{2}3}\left\|\omega_0\right\|_{\dot H^{1}},$
		$|\beta|^{-\frac{1}{3}(\frac{11}4-\frac2p)}\left\|\omega_0\right\|_{\dot H^{\frac74-\frac2p}}, $
		$|\beta|^{-1}\left\|\omega_0\right\|_{\dot H^2}, $
		$|\beta|^{-\frac{1}{3}(\frac{15}4-\frac2p)}\left\|\omega_0\right\|_{\dot H^{\frac{11}4-\frac2p}}$
	and 
		$|\beta|^{-\frac13(1+\delta-\frac2{p'})}\left\|\omega_0\right\|_{\dot W^{-1+\delta, p'}}$
	satisfying
		\begin{align}
			\int^t_\frac t2\left\|T_\beta(t-\tau)\operatorname{div}(u(\tau)\omega(\tau))\right\|_{\dot W^{s, \infty}}\, d\tau
			\leq CM_{s, \infty}(t)\label{eq:5.248}
		\end{align}
	for all
		$t>0$.
\end{proposition}
\proof
	Here, we only prove the case of
		$s=0.$
	See Appendix~\ref{sect:c} for the proof of the general cases.
	
	Similarly to \eqref{eq:5.204}, we have
		\begin{align}
			&\int^t_\frac t2\left\|T_\beta(t-\tau)\operatorname{div}(u(\tau)\omega(\tau))\right\|_{L^{\infty}}\, d\tau\label{eq:5.249}\\
			&\lesssim \int^t_\frac t2(t-\tau)^{-\frac12}\left\|u(\tau)\right\|_{\dot W^{\frac23(-1+\delta), p}}\left\|\nabla\omega(\tau)\right\|_{L^{p}}\, d\tau\label{eq:5.250}\\
			&\lesssim t^{1-\frac1{r_1}-\frac12}M_1(t)\Omega_1(t)\left\|\omega\right\|_{L^{r_1}([0, \infty); \dot W^{\frac23(-1+\delta), p})}\label{eq:5.251}\\
			&\lesssim \left(|\beta|^\frac13t^\frac12\right)^{\frac2p-\frac2{r_1}}t^{-1}\Omega_1(t)\label{eq:5.252}\\
			&\quad\cdot\bigg(|\beta|^{-\frac{\delta}{3}}\left\|\omega_0\right\|_{\dot H^{-1+\delta}}
			+|\beta|^{-\frac{1+2\delta}{9}}\left\|\omega_0\right\|_{\dot H^{\frac23(-1+\delta)}}+|\beta|^{-\frac{\delta}{3}}\left\|\omega_0\right\|_{\dot H^{-1+\delta}}
			|\beta|^{-\frac{4+2\delta}9}\left\|\omega_0\right\|_{\dot H^{\frac{1+2\delta}{3}}}\bigg).\quad\label{eq:5.253}
		\end{align}
	We can also deduce
		\begin{align}
			&\int^t_\frac t2\left\|T_\beta(t-\tau)\operatorname{div}(u(\tau)\omega(\tau))\right\|_{L^{\infty}}\, d\tau\label{eq:5.254}\\
			&\lesssim \int^t_\frac t2(t-\tau)^{-\frac38}\left\|e^{\frac12(t-\tau)\Delta}e^{(t-\tau)\beta L_1}(u(\tau)\cdot\nabla\omega(\tau))\right\|_{L^{\frac83}}\, d\tau\label{eq:5.255}\\
			&\lesssim \int^t_\frac t2(t-\tau)^{-\frac38-1+\frac68}\min\left\{1, |\beta|^{-1+\frac{6}{8}}(t-\tau)^{-\frac32(1-\frac68)}\right\}\left\|u(\tau)\cdot\nabla\omega(\tau)\right\|_{L^{\frac85}}\, d\tau.\label{eq:5.256}
		\end{align}
	Now, define the exponent
		$b$
	by
		$\frac58=\frac1p-\frac\delta2+\frac1b, $
	i.e., 
		$\frac1b=\frac2p-\frac38.$
	Then the condition of
		$p$
	ensures that
		$\frac{1}{b}\in [\frac23-\frac38, \frac{22}{30}-\frac38]=[\frac7{24}, \frac{43}{120}]\subset [\frac14, \frac12].$
	Thus, by interpolating the estimates obtained by Proposition~\ref{prop:5.6} and Lemma~\ref{lem:5.9}, we see that for any
		$s\geq0$, 
	the function
		$\Omega_{s, b}$
	is bounded on the interval
		$(0, \infty)$.
	Using this estimate gives
		\begin{align}
			&\int^t_\frac t2\left\|T_\beta(t-\tau)\operatorname{div}(u(\tau)\omega(\tau))\right\|_{L^{\infty}}\, d\tau\label{eq:5.257}\\
			&\lesssim \int^t_\frac t2(t-\tau)^{-\frac58}\min\left\{1, |\beta|^{-\frac{1}{4}}(t-\tau)^{-\frac38}\right\}\left\|u(\tau)\cdot\nabla\omega(\tau)\right\|_{L^{\frac85}}\, d\tau\label{eq:5.258}\\
			&\lesssim \int^t_\frac t2(t-\tau)^{-\frac58}\min\left\{1, |\beta|^{-\frac{1}{4}}(t-\tau)^{-\frac38}\right\}\left\|\omega(\tau)\right\|_{\dot W^{-1+\delta, p}}\left\|\omega(\tau)\right\|_{\dot W^{1, b}}\, d\tau\label{eq:5.259}\\
			&\lesssim N_0(t)\widetilde \Omega_0(t) M_{1, b}(t)\Omega_{1, b}(t)\int^t_\frac t2(t-\tau)^{-\frac58}\min\left\{1, |\beta|^{-\frac{1}{4}}(t-\tau)^{-\frac38}\right\}\, d\tau\label{eq:5.260}\\
			&\leq|\beta|^{\frac13(1+\delta-\frac2{p'})-2+\frac2p+\frac2b}t^{-\frac52(1-\frac2p)-\frac12-1+\frac1b-\frac32(1-\frac2b)}\Omega_{1, b}(t)\label{eq:5.261}\\
			&\quad\cdot\left(|\beta|^{-\frac13(1+\delta-\frac2{p'})}\widetilde \Omega_0(t)\right)\int^t_\frac t2(t-\tau)^{-\frac58}\min\left\{1, |\beta|^{-\frac{1}{4}}(t-\tau)^{-\frac38}\right\}\, d\tau\label{eq:5.262}\\
			&=|\beta|^{\frac13(\frac{26}p-\frac{45}4)}t^{-7+\frac{13}p}\Omega_{1, b}(t)\label{eq:5.263}\\
			&\quad\cdot\left(|\beta|^{-\frac13(1+\delta-\frac2{p'})}\widetilde \Omega_0(t)\right)\int^t_\frac t2(t-\tau)^{-\frac58}\min\left\{1, |\beta|^{-\frac{1}{4}}(t-\tau)^{-\frac38}\right\}\, d\tau.\label{eq:5.264}
		\end{align}
	Note that both 
		$-7+\frac{13}p<-\frac52$
	and
		$-7+\frac{13}p+\frac38>-\frac52$
	hold under the assumption of the proposition, since they are equivalent to
		$-\frac{5}{52}<\delta<\frac1{13}.$
	Therefore, we obtain the condition
		$\frac{104}{3p}-11\in (0, 1)$
	and the estimate
		\begin{align}
			&\int^t_\frac t2\left\|T_\beta(t-\tau)\operatorname{div}(u(\tau)\omega(\tau))\right\|_{L^{\infty}}\, d\tau\label{eq:5.265}\\
			&\lesssim|\beta|^{\frac13(\frac{26}p-\frac{45}4)}t^{-7+\frac{13}p}\Omega_{1, b}(t)\label{eq:5.266}\\
			&\quad\cdot\left(|\beta|^{-\frac13(1+\delta-\frac2{p'})}\widetilde \Omega_0(t)\right)\int^t_\frac t2(t-\tau)^{-\frac58}\left(  |\beta|^{-\frac{1}{4}}(t-\tau)^{-\frac38}\right)^{\frac{104}{3p}-11} \, d\tau\label{eq:5.267}\\
			&\lesssim |\beta|^{-1}t^{-\frac52}\Omega_{1, b}(t)\left(|\beta|^{-\frac13(1+\delta-\frac2{p'})}\widetilde \Omega_0(t)\right).\label{eq:5.268}
		\end{align}
	Interpolating \eqref{eq:5.253} and \eqref{eq:5.268}, we have
		\begin{align}
			&\int^t_\frac t2\left\|T_\beta(t-\tau)\operatorname{div}(u(\tau)\omega(\tau))\right\|_{L^{\infty}}\, d\tau\label{eq:5.269}\\
			&\lesssim t^{-1}\left(\Omega_{1, b}(t)\left(|\beta|^{-\frac13(1+\delta-\frac2{p'})}\widetilde \Omega_0(t)\right)\right)^\frac{\frac2p-\frac2{r_1}}{3+\frac2p-\frac2{r_1}}\label{eq:5.270}\\
			&\quad\cdot\bigg(\Omega_1(t)\bigg(|\beta|^{-\frac{\delta}{3}}\left\|\omega_0\right\|_{\dot H^{-1+\delta}}
			+|\beta|^{-\frac{1+2\delta}{9}}\left\|\omega_0\right\|_{\dot H^{\frac23(-1+\delta)}}+|\beta|^{-\frac{\delta}{3}}\left\|\omega_0\right\|_{\dot H^{-1+\delta}}
			|\beta|^{-\frac{4+2\delta}9}\left\|\omega_0\right\|_{\dot H^{\frac{1+2\delta}{3}}}\bigg)\bigg)^\frac{3}{3+\frac2p-\frac2{r_1}}.\label{eq:5.271}
		\end{align}
	Combining \eqref{eq:5.268} and \eqref{eq:5.271}, we obtain the desired estimate.
\qed

Now, we can prove the temporal decay part of Theorem~\ref{thm:1.3}.
	By Lemma~\ref{lem:5.8} and Proposition~\ref{prop:5.10}, we see that for any
		$s\geq0$, 
	the function
		$\Omega_{s, \infty}$
	is bounded on
		$(0, \infty)$.
	Proposition~\ref{prop:5.6} gives the boundedness of
		$\Omega_{s, 2}.$
	Interpolating these bounds, we see that
		$\Omega_{s, a}$
	is bounded on
		$(0, \infty)$
	for any 
		$s\geq0$
	and 
		$a\in[2, \infty].$
	The dependence on the initial data can be deduced by interpolation.

\subsection{Asymptotics of the solution}
In this subsection, we prove the asymptotics part of Theorem~\ref{thm:1.3}.

\begin{lemma}\label{lem:5.11}
	Assume that
		$0\leq \delta<\frac1{13}, $
		$\frac12-\frac1p+\frac\delta2-\frac32(1-\frac2p)
		<\frac1{r_1}
		<\frac12-\frac1p+\frac\delta2$
	and the initial data
		$\omega_0\in \dot H^{-1+\delta}\cap \dot H^{3-\frac2p}\cap \dot W^{-1+\delta, p'}\cap L^1$
	satisfies the conditions in Proposition~\ref{prop:5.2}, Proposition~\ref{prop:5.3}, and Lemma~\ref{lem:5.4}. Then there is a constant
		$c_0>0$
	and
		$C=C(\beta, \omega_0)$
	satisfying
		$\left\|\omega(t)\right\|_{\dot H^{-1}}
		\leq Ct^{-c_0}$
	for $t>0.$
\end{lemma}
\proof
Define
	$p_0\in[p', \infty)$
by
	$\frac1{p_0}=\frac1{p'}-\frac\delta2=2(1-\frac2p)>0.$
The condition
	$p_0<2$
is equivalent to
	$\frac1{p}<\frac38, $
and this is true in our setting.
Thus, by the embedding
	$\dot W^{\delta, p'}\hookrightarrow L^{p_0}, $
we obtain
	\begin{align}
		\left\|T_\beta(t)\omega_0\right\|_{\dot H^{-1}}
		=\left\|e^{t\Delta}\omega_0\right\|_{\dot H^{-1}}
		\lesssim t^{-\frac1{p_0}+\frac12}\left\|\omega_0\right\|_{\dot W^{-1, p_0}}
		\lesssim t^{-\frac32+\frac4p}\left\|\omega_0\right\|_{\dot W^{-1+\delta, p'}}, \label{eq:5.272}
	\end{align}
and this ensures the decay of the linear term.

For the first half of the Duhamel integral, as in Proposition~\ref{prop:5.6} and Lemma~\ref{lem:5.7}, we have
	\begin{align}
		&\int_0^\frac t2\left\|T_\beta(t-\tau)\operatorname{div}(u(\tau)\omega(\tau))\right\|_{\dot H^{-1}}\, d\tau\label{eq:5.273}\\
		&\lesssim t^{-(\frac1{p'}-\frac12)}\Omega_0(t)\left\|\omega\right\|_{Y_1}\left\|M_0\right\|_{L^{r_1'}((0, \frac t2))}.\label{eq:5.274}
	\end{align}
Under the condition 
	$\frac12-\frac1p+\frac\delta2-\frac32(1-\frac2p)
	<\frac1{r_1}
	<\frac12-\frac1p+\frac\delta2$, 
we can take
	$\theta, \gamma\in(0, 1)$
such that
	$-\frac1{p'}+\frac12-\frac1{r_1}+\frac1p-\frac32\theta(1-\frac2p)=0, 1-\frac1p+\frac32(\theta+\gamma)(1-\frac2p)<1-\frac1{r_1}, $
and
	$\theta+\gamma<1.$
Then we have
	\begin{align}
		&\left\|M_0\right\|_{L^{r_1'}((0, \frac t2))}
		=\left(\int_0^\frac t2\left(\tau^{-1+\frac1p}\min\left\{1, |\beta|^{-1+\frac2p}\tau^{-\frac32(1-\frac2p)}\right\}\right)^{r_1'}\, d\tau\right)^{1-\frac1{r_1}}\label{eq:5.275}\\
		&\leq\left(\int_0^\frac t2\left(\tau^{-1+\frac1p}\left(|\beta|^{-1+\frac2p}\tau^{-\frac32(1-\frac2p)}\right)^{\theta+\gamma}\right)^{r_1'}\, d\tau\right)^{1-\frac1{r_1}}\label{eq:5.276}\\
		&\lesssim |\beta|^{-(\theta+\gamma)(1-\frac2p)}t^{-\frac1{r_1}+\frac1p-\frac32(\theta+\gamma)(1-\frac2p)}.\label{eq:5.277}
	\end{align}
Inserting this estimate to the previous one, we obtain
	\begin{align}
		&\int_0^\frac t2\left\|T_\beta(t-\tau)\operatorname{div}(u(\tau)\omega(\tau))\right\|_{\dot H^{-1}}\, d\tau\label{eq:5.278}\\
		&\lesssim \left(\left|\beta\right|^\frac13t^\frac12\right)^{-3\gamma(1-\frac2p)}\Omega_0(t)|\beta|^{-\frac\delta3}\left\|\omega_0\right\|_{\dot H^{-1+\delta}}.\label{eq:5.279}
	\end{align}
	
For the second half, as in Proposition~\ref{prop:5.6}, we have
		\begin{align}
			&\int^t_\frac t2\left\|T_\beta(t-\tau)\operatorname{div}(u(\tau)\omega(\tau))\right\|_{\dot H^{-1}}\, d\tau
			\lesssim \int^t_\frac t2(t-\tau)^{-\frac1{p'}+\frac12}\left\|u(\tau)\omega(\tau)\right\|_{L^{p'}}\, d\tau\label{eq:5.280}\\
			 &\lesssim \left(|\beta|^\frac13t^\frac12\right)^{-7+\delta+\frac{14}{p}} |\beta|^{-\frac13(1+\delta-\frac2{p'})}\widetilde \Omega_0(t) \Omega_0(t).\label{eq:5.281}
		\end{align}
Combining the above estimates, we obtain the desired decay of
	$\left\|\omega(t)\right\|_{\dot H^{-1}}.$
\qed

\begin{proposition}\label{lem:5.12}
	Assume that
		$0\leq \delta<\frac1{13}$, 
		$\frac12-\frac1p+\frac\delta2-\frac32(1-\frac2p)
		<\frac1{r_1}
		<\frac12-\frac1p+\frac\delta2$
	and the initial data
		$\omega_0\in \dot H^{-1+\delta}\cap \dot H^{3-\frac2p}\cap \dot W^{-1+\delta, p'}\cap L^1$
	satisfies the conditions in Proposition~\ref{prop:5.2}, Proposition~\ref{prop:5.3}, and Lemma~\ref{lem:5.4}. Then, for any 
		$s\geq0$
	and
		$a\in[2, \infty]$, 
	it holds that
		\begin{align}
			\lim_{t\to\infty}M_{s, a}(t)^{-1}\left\|\omega(t)-K_{\beta, t}\int_{\mathbb R^2}\omega_0(y)\, dy\right\|_{\dot W^{s, a}}=0.\label{eq:5.282}
		\end{align}
\end{proposition}
\proof
By the interpolation inequality, it suffices to show the estimates for the endpoint cases
	$a=2, \infty.$
We write
	\begin{align}
		&M_{s, a}(t)^{-1}\left\|\omega(t)-K_{\beta, t}\int_{\mathbb R^2}\omega_0(y)\, dy\right\|_{\dot W^{s, a}}\label{eq:5.283}\\
		&\leq M_{s, a}(t)^{-1}\left\|T_\beta(t)\omega_0-K_{\beta, t}\int_{\mathbb R^2}\omega_0(y)\, dy\right\|_{\dot W^{s, a}}\label{eq:5.284}\\
		&\quad+M_{s, a}(t)^{-1}\int_0^\frac t2\left\|T_\beta(t-\tau)\operatorname{div}(u(\tau)\omega(\tau))\right\|_{\dot W^{s, a}}\, d\tau\label{eq:5.285}\\
		&\quad+M_{s, a}(t)^{-1}\int^t_\frac t2\left\|T_\beta(t-\tau)\operatorname{div}(u(\tau)\omega(\tau))\right\|_{\dot W^{s, a}}\, d\tau\label{eq:5.286}
	\end{align}
and show that each term in the right-hand side converges to 0 as 
	$t$
tends to
	$\infty$. 
The convergence of the first term follows by Proposition~\ref{prop:2.4'}, so we are left to estimate the second and third terms.
	
We first treat the case of
	$a=2$.
From the argument in Proposition~\ref{prop:5.6}, we have
	\begin{align}
		&\int_0^\frac t2\left\|T_\beta(t-\tau)\operatorname{div}(u(\tau)\omega(\tau))\right\|_{\dot H^{s}}\, d\tau\label{eq:5.287}
		\\
		&\lesssim t^{-\frac s2-\frac12-(\frac1{p'}-\frac12)}\Omega_0(t)\left\|\omega(\tau)\right\|_{Y_1}\left\|M_0\right\|_{L^{r_1'}((0, \frac t2))}.\label{eq:5.288}
	\end{align}
By the argument in Lemma~\ref{lem:5.11}, we have
	\begin{align}
		&\left\|M_0\right\|_{L^{r_1'}((0, \frac t2))}
		\lesssim |\beta|^{-(\theta+\gamma)(1-\frac2p)}t^{-\frac1{r_1}+\frac1p-\frac32(\theta+\gamma)(1-\frac2p)}.\label{eq:5.291}
	\end{align}
Inserting this estimate to the previous one, we obtain
	\begin{align}
		&\int_0^\frac t2\left\|T_\beta(t-\tau)\operatorname{div}(u(\tau)\omega(\tau))\right\|_{\dot H^{s}}\, d\tau\label{eq:5.292}\\
		&\lesssim \left(\left|\beta\right|^\frac13t^\frac12\right)^{-3\gamma(1-\frac2p)}t^{-\frac s2-\frac12}\Omega_0(t)|\beta|^{-\frac\delta3}\left\|\omega_0\right\|_{\dot H^{-1+\delta}}.\label{eq:5.293}
	\end{align}
This is enough to estimate the first half of the Duhamel integral. For the second half, 
from \eqref{eq:5.184} we have
	\begin{align}
		&\int^t_\frac t2\left\|T_\beta(t-\tau)\operatorname{div}(u(\tau)\omega(\tau))\right\|_{\dot H^s}\, d\tau\label{eq:5.296}\\
		&\lesssim\left(|\beta|^\frac13t^\frac12\right)^{-7+\delta+\frac{14}{p}}t^{-\frac s2-\frac12}|\beta|^{-\frac13(1+\delta-\frac2{p'})}\widetilde \Omega_0(t)\Omega_{1+s}(t).
		\label{eq:5.298}
	\end{align}
These are sufficient to conclude the estimates for
	$a=2.$
	
Let us next consider the case of
	$a=\infty.$
For the first half of the time integral, we have
		\begin{align}
			&\int^\frac t2_0\left\|T_\beta(t-\tau)\operatorname{div}(u(\tau)\omega(\tau))\right\|_{\dot W^{s, \infty}}\, d\tau\label{eq:5.299}\\
			&\lesssim t^{-\frac12}M_{s, \infty}(t)\int^\frac t2_0 \left\|u(\tau)\right\|_{L^2}\left\|\omega(\tau)\right\|_{L^2}\, d\tau\label{eq:5.300}.
		\end{align}
Lemma~\ref{lem:5.7} and Lemma~\ref{lem:5.11} ensures that there are two constants
	$C=C(\beta, \omega_0)$
and 
	$0<\varepsilon<\frac12$
satisfying
	$\left\|u(\tau)\right\|_{L^2}\leq C t^{-\varepsilon}$
for
	$t>0.$
Using also the boundedness of
	$\Omega_{0, 2}(t), $
we obtain
		\begin{align}
			&\int^\frac t2_0\left\|T_\beta(t-\tau)\operatorname{div}(u(\tau)\omega(\tau))\right\|_{\dot W^{s, \infty}}\, d\tau\label{eq:5.301}\\
			&\lesssim_{\beta, \omega_0} t^{-\frac12}M_{s, \infty}(t)\Omega_{0, 2}(t)\int^\frac t2_0 \tau^{-\varepsilon-\frac12}\, d\tau\label{eq:5.302}
			\lesssim t^{-\varepsilon}M_{s, \infty}(t)\Omega_{0, 2}(t).
		\end{align}
		
For the second half, in the case of 
	$s=0$, 
from \eqref{eq:5.264} we have 
		\begin{align}
			&\int^t_\frac t2\left\|T_\beta(t-\tau)\operatorname{div}(u(\tau)\omega(\tau))\right\|_{L^{\infty}}\, d\tau\label{eq:5.303}\\
			&\lesssim|\beta|^{\frac13(\frac{26}p-\frac{45}4)}t^{-7+\frac{13}p}\Omega_{1, b}(t)\label{eq:5.304}\\
			&\quad\cdot\left(|\beta|^{-\frac13(1+\delta-\frac2{p'})}\widetilde \Omega_0(t)\right)\int^t_\frac t2(t-\tau)^{-\frac58}\min\left\{1, |\beta|^{-\frac{1}{4}}(t-\tau)^{-\frac38}\right\}\, d\tau.\label{eq:5.305}
		\end{align}
By taking
	$\gamma\in(0, 1)$
small enough to satisfy
	$\frac{104}{3p}-11+\gamma<1, $
we obtain
		\begin{align}
			&\int^t_\frac t2\left\|T_\beta(t-\tau)\operatorname{div}(u(\tau)\omega(\tau))\right\|_{L^{\infty}}\, d\tau\label{eq:5.306}\\
			&\lesssim|\beta|^{\frac13(\frac{26}p-\frac{45}4)}t^{-7+\frac{13}p}\Omega_{1, b}(t)\label{eq:5.307}\\
			&\quad\cdot\left(|\beta|^{-\frac13(1+\delta-\frac2{p'})}\widetilde \Omega_0(t)\right)\int^t_\frac t2(t-\tau)^{-\frac58}\left(  |\beta|^{-\frac{1}{4}}(t-\tau)^{-\frac38}\right)^{\frac{104}{3p}-11+\gamma} \, d\tau\label{eq:5.308}\\
			&\lesssim \left(\left|\beta\right|^\frac13t^\frac12\right)^{-\frac34\gamma}|\beta|^{-1}t^{-\frac52}\Omega_{1, b}(t)\left(|\beta|^{-\frac13(1+\delta-\frac2{p'})}\widetilde \Omega_0(t)\right).\label{eq:5.309}
		\end{align}
For 
	$s\in\mathbb N$, 
by the argument in Appendix~\ref{sect:c}, we similarly obtain
	\begin{align}
		&\int^t_\frac t2\left\|T_\beta(t-\tau)\operatorname{div}(u(\tau)\omega(\tau))\right\|_{\dot W^{s, \infty}}\, d\tau\label{eq:5.310}\\
		&\lesssim \int^t_\frac t2(t-\tau)^{-\frac58}\min\left\{1, |\beta|^{-\frac{1}{4}}(t-\tau)^{-\frac38}\right\}\left\|\omega(\tau)\right\|_{\dot W^{-1+\delta+s, p}}\left\|\omega(\tau)\right\|_{\dot W^{1, b}}\, d\tau\label{eq:5.311}\\
		&+\int^t_\frac t2(t-\tau)^{-\frac58}\min\left\{1, |\beta|^{-\frac{1}{4}}(t-\tau)^{-\frac38}\right\}\left\|\omega(\tau)\right\|_{\dot W^{-1+\delta, p}}\left\|\omega(\tau)\right\|_{\dot W^{1+s, b}}\, d\tau\label{eq:5.312}\\
		& \lesssim \left(\left|\beta\right|^\frac13t^\frac12\right)^{-\frac34\gamma} |\beta|^{-1}t^{-\frac s2-\frac52}\label{eq:5.313}\\
		&\qquad\cdot\left(\Omega_{1, b}(t)\left(|\beta|^{-\frac13(1+\delta-\frac2{p'})}\widetilde\Omega_s(t)\right)
		+\Omega_{1+s, b}(t)\left(|\beta|^{-\frac13(1+\delta-\frac2{p'})}\widetilde\Omega_0(t)\right)\right).\label{eq:5.314}
	\end{align}
Combining these estimates, we obtain the desired decays for
	$\left\|\omega(t)\right\|_{\dot W^{s, \infty}}.$
This completes the proof.
\qed

By Proposition~\ref{lem:5.12}, we obtain the asymptotics part of Theorem~\ref{thm:1.3}.

\appendix

\section{Calculation of $\partial_t\omega$}\label{sect:b}
In this section, we show that under the condition \eqref{eq:4.1}, 
	\begin{align}
		\lim_{h\downarrow 0}\frac{\omega(t)-\omega(t-h)}{h}=-\operatorname{div}(u(t)\omega(t))+\Delta \omega(t)+\beta L_1\omega(t)\label{eq:b.1}
	\end{align}
	holds for
		$t>0$
	in 
		$\dot H^{\delta+m}$, 
	where
		$m\geq0$
	is arbitrary.
	The strategy is almost the same as that of Proposition~\ref{prop:4.4}. In the following argument, we assume that
		$0<h<\frac t2.$
	
	We write
		\begin{align}
			&\frac{\omega(t)-\omega(t-h)}{h}
			=\frac{T_\beta(t)\omega_0-T_\beta(t-h)\omega_0}{h}\label{eq:b.2}\\
			&-\frac1h\left(\int_0^{t}T_\beta(t-\tau)\operatorname{div}(u(\tau)\omega(\tau))\, d\tau-\int_0^{t-h}T_\beta(t-\tau)\operatorname{div}(u(\tau)\omega(\tau))\, d\tau\right)\label{eq:b.3}\\
			&-\frac1h\left(\int_0^{t-h}T_\beta(t-\tau)\operatorname{div}(u(\tau)\omega(\tau))\, d\tau
			-\int_0^{t-h}T_\beta(t-h-\tau)\operatorname{div}(u(\tau)\omega(\tau))\, d\tau\right).\label{eq:b.4}
		\end{align}
	
	For \eqref{eq:b.2}, we have
		\begin{align}
			&\quad\left\|\frac{T_\beta(t)\omega_0-T_\beta(t-h)\omega_0}{h}-\Delta T_\beta(t)\omega_0-\beta L_1 T_\beta(t)\omega_0\right\|_{\dot H^{\delta+m}}\label{eq:b.5}\\
			&\leq\left\|\frac{e^{h\Delta}e^{h\beta L_1}T_\beta(t-h)\omega_0-e^{h\beta L_1}T_\beta(t-h)\omega_0}{h}-e^{h\beta L_1}\Delta T_\beta(t-h)\omega_0\right\|_{\dot H^{\delta+m}}\label{eq:b.6}\\
			&+\left\|e^{h\beta L_1}\Delta T_\beta(t-h)\omega_0-e^{h\Delta}e^{h\beta L_1}\Delta T_\beta(t-h)\omega_0\right\|_{\dot H^{\delta+m}}\label{eq:b.7}\\
			&+\left\|\frac{e^{h\beta L_1}T_\beta(t-h)\omega_0-T_\beta(t-h)\omega_0}{h}-\beta L_1 T_\beta(t-h)\omega_0\right\|_{\dot H^{\delta+m}}\label{eq:b.8}\\
			&+\left\|\beta L_1 T_\beta(t-h)\omega_0-\beta L_1T_\beta(t)\omega_0\right\|_{\dot H^{\delta+m}}.\label{eq:b.9}
		\end{align}
	and we can deduce that each line tends to 0 as
		$h\to0.$
		
		We next treat \eqref{eq:b.3}. We have
			\begin{align}
				&\quad\left\|\frac1h\int^t_{t-h}T_\beta(t-\tau)\operatorname{div}(u(\tau)\omega(\tau))\, d\tau-\operatorname{div}(u(t)\omega(t))\right\|_{\dot H^{\delta+m}}\label{eq:b.10}\\
				&\leq\left\|\frac1h\int^t_{t-h}T_\beta(t-\tau)\left(\operatorname{div}(u(\tau)\omega(\tau))-\operatorname{div}(u(t)\omega(t))\right)\, d\tau\right\|_{\dot H^{\delta+m}}\label{eq:b.11}\\
				&+\left\|\frac1h\int^t_{t-h}\left(T_\beta(t-\tau)-1\right)\operatorname{div}(u(t)\omega(t))\, d\tau\right\|_{\dot H^{\delta+m}}\label{eq:b.12}\\
				&\leq \sup_{t-h<\tau<t}\left\|\operatorname{div}(u(\tau)\omega(\tau))-\operatorname{div}(u(t)\omega(t))\right\|_{\dot H^{\delta+m}}\label{eq:b.13}\\
				&+\sup_{t-h<\tau<t}\left\|\left(T_\beta(t-\tau)-1\right)\operatorname{div}(u(t)\omega(t))\right\|_{\dot H^{\delta+m}}.\label{eq:b.14}
			\end{align}
		and the bound tends to 0 as 
			$h\to0$.
		
		For \eqref{eq:b.4}, we write
			\begin{align}
				&\quad\frac1h\int_0^{t-h}\left(T_\beta(h)-1\right)T_\beta(t-h-\tau)\operatorname{div}(u(\tau)\omega(\tau))\, d\tau
				\label{eq:b.15}\\
				&=\frac1h\int_0^{\frac {t-h}2}\left(e^{h\Delta}-1\right)e^{h\beta L_1}T_\beta(t-h-\tau)\operatorname{div}(u(\tau)\omega(\tau))\, d\tau
				\label{eq:b.16}\\
				&+\frac1h\int_0^{\frac {t-h}2}\left(e^{h\beta L_1}-1\right)T_\beta(t-h-\tau)\operatorname{div}(u(\tau)\omega(\tau))\, d\tau
				\label{eq:b.17}\\
				&+\frac1h\int_\frac {t-h}2^{t-h}\left(e^{h\Delta}-1\right)e^{h\beta L_1}T_\beta(t-h-\tau)\operatorname{div}(u(\tau)\omega(\tau))\, d\tau
				\label{eq:b.18}\\
				&+\frac1h\int_\frac {t-h}2^{t-h}\left(e^{h\beta L_1}-1\right)T_\beta(t-h-\tau)\operatorname{div}(u(\tau)\omega(\tau))\, d\tau
				\label{eq:b.19}
			\end{align}
		and calculate the limits of each line.
		For \eqref{eq:b.16}, we have
			\begin{align}
				&\quad\left\|\frac1h\int_0^{\frac {t-h}2}\left(e^{h\Delta}-1\right)e^{h\beta L_1}T_\beta(t-h-\tau)\operatorname{div}(u(\tau)\omega(\tau))\, d\tau
				-\int_0^{\frac t2}\Delta T_\beta(t-\tau)\operatorname{div}(u(\tau)\omega(\tau))\, d\tau
				\right\|_{\dot H^{\delta +m}}
				\notag\\
				&\leq \left\|\frac1h\int_0^{\frac {t-h}2}\left(e^{h\Delta}-1\right)e^{h\beta L_1}T_\beta(t-h-\tau)\operatorname{div}(u(\tau)\omega(\tau))\, d\tau\right.\label{eq:b.21}\\
				&\phantom{\frac1h\int_0^{\frac {t-h}2}\left(e^{h\Delta}-1\right)}\left.-\frac1h\int_0^{\frac {t-h}2}\left(e^{h\Delta}-1\right)T_\beta(t-h-\tau)\operatorname{div}(u(\tau)\omega(\tau))\, d\tau\right\|_{\dot H^{\delta+m}}
				\label{eq:b.22}\\
				&+\left\|\frac1h\int_0^{\frac {t-h}2}\left(e^{h\Delta}-1\right)T_\beta(t-h-\tau)\operatorname{div}(u(\tau)\omega(\tau))\, d\tau
				-\int_0^{\frac {t-h}2}\Delta T_\beta(t-h-\tau)\operatorname{div}(u(\tau)\omega(\tau))\, d\tau\right\|_{\dot H^{\delta+m}}\label{eq:b.23}\\
				&+\left\|\int_0^{\frac {t-h}2}\Delta T_\beta(t-h-\tau)\operatorname{div}(u(\tau)\omega(\tau))\, d\tau
				-\int_0^{\frac {t-h}2}\Delta T_\beta(t-\tau)\operatorname{div}(u(\tau)\omega(\tau))\, d\tau\right\|_{\dot H^{\delta+m}}\label{eq:b.24}\\
				&+\left\|-\int_{\frac {t-h}2}^\frac t2\Delta T_\beta(t-\tau)\operatorname{div}(u(\tau)\omega(\tau))\, d\tau\right\|_{\dot H^{\delta+m}}\label{eq:b.25}
			\end{align}
		and we can show that each line converges to 0 as
			$h\to0.$
		For \eqref{eq:b.22}, 
			\begin{align}
				&\left\|\frac1h\int_0^{\frac {t-h}2}\left(e^{h\Delta}-1\right)\left(e^{h\beta L_1}-1\right)T_\beta(t-h-\tau)\operatorname{div}(u(\tau)\omega(\tau))\, d\tau\right\|_{\dot H^{\delta+m}}
				\label{eq:b.26}\\
				&\lesssim h\int_0^{\frac {t-h}2}\left\|\Delta\beta L_1e^{(t-h-\tau)\Delta}(-\Delta)^\frac{1+\delta+m}{2}(u(\tau)\omega(\tau))\right\|_{L^2}\, d\tau
				\label{eq:b.27}\\
				&\lesssim h|\beta|\int_0^{\frac {t-h}2}(t-h-\tau)^{-\frac{2+m}{2}-\frac 1q+\frac12}\left\|(-\Delta)^\frac{\delta}{2}(u(\tau)\omega(\tau))\right\|_{L^q}\, d\tau
				\label{eq:b.28}\\
				&\lesssim h|\beta|t^{1-\frac1{r_1}-\frac1{r_2}-\frac{2+m}{2}-\frac 1q+\frac12}\left\|\omega\right\|_{Y_1}\left\|\omega\right\|_{Y_2}, 
				\label{eq:b.29}
			\end{align}
		and the bound goes to 0 as
			$h\to0.$
		For \eqref{eq:b.23}, we have
			\begin{align}
				&\quad\left\|\int_0^{\frac {t-h}2}\left(\frac{e^{h\Delta}-1}{h}-\Delta\right)T_\beta(t-h-\tau)\operatorname{div}(u(\tau)\omega(\tau))\, d\tau\right\|_{\dot H^{\delta+m}}
				\label{eq:b.30}\\
				&\lesssim\int_0^{\frac {t-h}2}\left\|\left(\frac{e^{h\Delta}-1}{h}-\Delta\right)e^{(t-h-\tau)\Delta}(-\Delta)^{\frac{1+\delta+m}{2}}(u(\tau)\omega(\tau))\right\|_{L^2}\, d\tau.
				\label{eq:b.31}
			\end{align}
		The Plancherel theorem gives
			\begin{align}
				&\quad\left\|\left(\frac{e^{h\Delta}-1}{h}-\Delta\right)e^{(t-h-\tau)\Delta}(-\Delta)^{\frac{1+\delta+m}{2}}(u(\tau)\omega(\tau))\right\|_{L^2}\label{eq:b.32}\\
				&\leq\frac1{2\pi}\left(\int_{\mathbb R^2}\left||\xi|^2\int_0^1\left(e^{-h\theta |\xi|^2}-1\right)\, d\theta\right|^2\left|\mathcal F\left(e^{\frac{t}4\Delta}(-\Delta)^{\frac{1+\delta+m}{2}}(u(\tau)\omega(\tau))\right)(\xi)\right|^2\, d\xi\right)^\frac12\label{eq:b.33}
			\end{align}
		and this tends to 0 for each
			$\tau\in(0, \frac t2)$
		as
			$h\to0, $
		by the dominated convergence theorem.
		We also have
			\begin{align}
				&\quad1_{(0, \frac{t-h}{2})}(\tau)\left\|\left(\frac{e^{h\Delta}-1}{h}-\Delta\right)e^{(t-h-\tau)\Delta}(-\Delta)^{\frac{1+\delta+m}{2}}(u(\tau)\omega(\tau))\right\|_{L^2}\label{eq:b.34}\\
				&\lesssim1_{(0, \frac{t-h}{2})}(\tau)\left\|e^{(t-h-\tau)\Delta}(-\Delta)^{\frac{3+\delta+m}{2}}(u(\tau)\omega(\tau))\right\|_{L^2}\label{eq:b.35}\\
				&\lesssim(t-\tau)^{-\frac{3+m}{2}-\frac1q+\frac12}\left\|(-\Delta)^{\frac{\delta}{2}}(u(\tau)\omega(\tau))\right\|_{L^q}\label{eq:b.36}
			\end{align}
		and this is integrable on
			$(0, \frac t2).$
		Then again the dominated convergence theorem gives
			\begin{align}
				\int_0^{\frac {t-h}2}\left\|\left(\frac{e^{h\Delta}-1}{h}-\Delta\right)e^{(t-h-\tau)\Delta}(-\Delta)^{\frac{1+\delta+m}{2}}(u(\tau)\omega(\tau))\right\|_{L^2}\, d\tau
				\to 0\quad\text{as $h\to0$.}\label{eq.b.37}
			\end{align}
		
		For	\eqref{eq:b.24}, we have
			\begin{align}
				&\quad\left\|\int_0^{\frac {t-h}2}\Delta \left(T_\beta(h)-1\right)T_\beta(t-h-\tau)\operatorname{div}(u(\tau)\omega(\tau))\, d\tau
				\right\|_{\dot H^{\delta+m}}\label{eq:b.38}\\
				&\lesssim\int_0^{\frac {t-h}2}\left\|\left(T_\beta(h)-1\right)e^{(t-h-\tau)\Delta}(-\Delta)^\frac{3+\delta+m}{2}(u(\tau)\omega(\tau))\right\|_{L^2}\, d\tau
				\label{eq:b.39}\\
				&\leq\int_0^{\frac {t-h}2}\left\|\left(e^{h\Delta}e^{h\beta L_1}-e^{h\beta L_1}\right)e^{(t-h-\tau)\Delta}(-\Delta)^\frac{3+\delta+m}{2}(u(\tau)\omega(\tau))\right\|_{L^2}\, d\tau\label{eq:b.40}\\
				&+\int_0^{\frac {t-h}2}\left\|\left(e^{h\beta L_1}-1\right)e^{(t-h-\tau)\Delta}(-\Delta)^\frac{3+\delta+m}{2}(u(\tau)\omega(\tau))\right\|_{L^2}\, d\tau\label{eq:b.41}\\
				&\lesssim h\int_0^{\frac {t-h}2}\left\|e^{(t-h-\tau)\Delta}(-\Delta)^\frac{5+\delta+m}{2}(u(\tau)\omega(\tau))\right\|_{L^2}\, d\tau\label{eq:b.42}\\
				&+h|\beta|\int_0^{\frac {t-h}2}\left\|e^{(t-h-\tau)\Delta}(-\Delta)^\frac{2+\delta+m}{2}(u(\tau)\omega(\tau))\right\|_{L^2}\, d\tau\label{eq:b.43}\\
				&\lesssim h\int_0^{\frac {t-h}2}(t-h-\tau)^{-\frac{5+m}{2}-\frac1q+\frac12}\left\|(-\Delta)^\frac{\delta}{2}(u(\tau)\omega(\tau))\right\|_{L^q}\, d\tau\label{eq:b.44}\\
				&+h|\beta|\int_0^{\frac {t-h}2}(t-h-\tau)^{-\frac{2+m}{2}-\frac1q+\frac12}\left\|(-\Delta)^\frac{\delta}{2}(u(\tau)\omega(\tau))\right\|_{L^q}\, d\tau\label{eq:b.45}\\
				&\lesssim ht^{-\frac{5+m}{2}-\frac1q+\frac12}\int_0^{\frac {t}2}\left\|(-\Delta)^\frac{\delta}{2}(u(\tau)\omega(\tau))\right\|_{L^q}\, d\tau\label{eq:b.46}\\
				&+ht^{-\frac{2+m}{2}-\frac1q+\frac12}|\beta|\int_0^{\frac {t}2}\left\|(-\Delta)^\frac{\delta}{2}(u(\tau)\omega(\tau))\right\|_{L^q}\, d\tau\label{eq:b.47}
			\end{align}
		and the bound tends to 0 as
			$h\to 0.$
			
		For \eqref{eq:b.25}, we have
			\begin{align}
				&\quad\left\|-\int_{\frac {t-h}2}^\frac t2\Delta T_\beta(t-\tau)\operatorname{div}(u(\tau)\omega(\tau))\, d\tau\right\|_{\dot H^{\delta+m}}\label{eq:b.48}\\
				&\lesssim\int_{\frac {t-h}2}^\frac t2\left\|e^{(t-\tau)\Delta}(-\Delta)^{\frac{3+\delta+m}{2}}(u(\tau)\omega(\tau))\right\|_{L^2}\, d\tau\label{eq:b.49}\\
				&\lesssim \int_{\frac {t-h}2}^\frac t2(t-\tau)^{-\frac{3+m}{2}-\frac1q+\frac12}\left\|(-\Delta)^{\frac{\delta}{2}}(u(\tau)\omega(\tau))\right\|_{L^q}\, d\tau\label{eq:b.50}\\
				&\lesssim \int_{\frac {t-h}2}^\frac t2(t-\tau)^{-\frac{3+m}{2}-\frac1q+\frac12}\left\|(-\Delta)^{\frac{\delta}{2}}(u(\tau)\omega(\tau))\right\|_{L^q}\, d\tau\label{eq:b.51}\\
				&\lesssim t^{-\frac{3+m}{2}}h^{1-\frac1{r_1}-\frac1{r_2}-\frac1q+\frac12}\left\|\omega\right\|_{Y_1}\left\|\omega\right\|_{Y_2}\to0\quad\text{as $h\to 0$}\label{eq:b.52}
			\end{align}
		and this gives the estimate for \eqref{eq:b.16}.

		For \eqref{eq:b.17}, similarly we have
			\begin{align}
				&\left\|\frac1h\int_0^{\frac {t-h}2}\left(e^{h\beta L_1}-1\right)T_\beta(t-h-\tau)\operatorname{div}(u(\tau)\omega(\tau))\, d\tau
				-\int_0^{\frac t2}\beta L_1T_\beta(t-\tau)\operatorname{div}(u(\tau)\omega(\tau))\, d\tau\right\|_{\dot H^{\delta+m}}\notag\\
				&\leq \left\|\frac1h\int_0^{\frac {t-h}2}\left(e^{h\beta L_1}-1\right)T_\beta(t-h-\tau)\operatorname{div}(u(\tau)\omega(\tau))\, d\tau
				-\int_0^{\frac {t-h}2}\beta L_1T_\beta(t-h-\tau)\operatorname{div}(u(\tau)\omega(\tau))\, d\tau\right\|_{\dot H^{\delta+m}}\label{eq:b.54}\\
				&+\left\|\int_0^{\frac {t-h}2}\beta L_1T_\beta(t-h-\tau)\operatorname{div}(u(\tau)\omega(\tau))\, d\tau
				-\int_0^{\frac {t-h}2}\beta L_1T_\beta(t-\tau)\operatorname{div}(u(\tau)\omega(\tau))\, d\tau\right\|_{\dot H^{\delta+m}}\label{eq:b.55}\\
				&+\left\|\int_{\frac {t-h}2}^\frac{t}{2}\beta L_1T_\beta(t-\tau)\operatorname{div}(u(\tau)\omega(\tau))\, d\tau\right\|_{\dot H^{\delta+m}}\label{eq:b.56}			\end{align}
		For \eqref{eq:b.54}, we have
			\begin{align}
				&\quad\left\|\int_0^{\frac {t-h}2}\left(\frac{e^{h\beta L_1}-1}{h}-\beta L_1\right)T_\beta(t-h-\tau)\operatorname{div}(u(\tau)\omega(\tau))\, d\tau \right\|_{\dot H^{\delta+m}}\label{eq:b.57}\\
				&\lesssim\int_0^{\frac {t-h}2}\left\|\left(\frac{e^{h\beta L_1}-1}{h}-\beta L_1\right)e^{(t-h-\tau)\Delta}(-\Delta)^\frac{1+\delta+m}{2}(u(\tau)\omega(\tau))\right\|_{L^2}\, d\tau. \label{eq:b.58}
			\end{align}
		For
			$\tau\in(0, \frac{t}{2})$, 
		the Plancherel theorem yields
			\begin{align}
				&\quad1_{(0, \frac{t-h}{2})}(\tau)\left\|\left(\frac{e^{h\beta L_1}-1}{h}-\beta L_1\right)e^{(t-h-\tau)\Delta}(-\Delta)^\frac{1+\delta+m}{2}(u(\tau)\omega(\tau))\right\|_{L^2}\label{eq:b.59}\\
				&\leq\frac1{2\pi}\left(\int_{\mathbb R^2}\left|\beta\frac{i\xi_1}{|\xi|^2}\int_0^1\left(e^{h\theta \beta\frac{i\xi_1}{|\xi|^2}}-1\right)\, d\theta\right|^2\left|\mathcal F\left(e^{\frac{t}4\Delta}(-\Delta)^{\frac{1+\delta+m}{2}}(u(\tau)\omega(\tau))\right)(\xi)\right|^2\, d\xi\right)^\frac12\notag
			\end{align}
		and this converges to 0 as 
			$h\to 0.$
		We also have
			\begin{align}
				&\quad1_{(0, \frac{t-h}{2})}(\tau)\left\|\left(\frac{e^{h\beta L_1}-1}{h}-\beta L_1\right)e^{(t-h-\tau)\Delta}(-\Delta)^\frac{1+\delta+m}{2}(u(\tau)\omega(\tau))\right\|_{L^2}\label{eq:b.61}\\
				&\lesssim 1_{(0, \frac{t-h}{2})}(\tau)\left|\beta\right|\left\|e^{(t-h-\tau)\Delta}(-\Delta)^\frac{\delta+m}{2}(u(\tau)\omega(\tau))\right\|_{L^2}\label{eq:b.62}\\
				&\lesssim 1_{(0, \frac{t-h}{2})}(\tau)\left|\beta\right|(t-h-\tau)^{-\frac m2-\frac1q+\frac12}\left\|(-\Delta)^\frac{\delta}{2}(u(\tau)\omega(\tau))\right\|_{L^q}\label{eq:b.63}\\
				&\lesssim\left|\beta\right|t^{-\frac m2-\frac1q+\frac12}\left\|(-\Delta)^\frac{\delta}{2}(u(\tau)\omega(\tau))\right\|_{L^q}.\label{eq:b.64}
			\end{align}
		The right-hand side is integrable on
			$\tau\in(0, \frac t2).$
		Therefore, the dominated convergence theorem ensures that \eqref{eq:b.54} tends to 0 as
			$h\to0.$
			
		For \eqref{eq:b.55}, we have
			\begin{align}
				&\quad\left\|\int_0^{\frac {t-h}2}\beta L_1\left(T_\beta(h)-1\right)T_\beta(t-h-\tau)\operatorname{div}(u(\tau)\omega(\tau))\, d\tau
				\right\|_{\dot H^{\delta+m}}
				\label{eq:b.65}\\
				&\lesssim |\beta |\int_0^{\frac {t-h}2}\left\|\left(T_\beta(h)-1\right)e^{(t-h-\tau)\Delta}(-\Delta)^\frac{\delta+m}{2}(u(\tau)\omega(\tau))\right\|_{L^2}\, d\tau\label{eq:b.66}\\
				&\leq |\beta |\int_0^{\frac {t-h}2}\left\|\left(e^{h\Delta}e^{h\beta L_1}-e^{h\beta L_1}\right)e^{(t-h-\tau)\Delta}(-\Delta)^\frac{\delta+m}{2}(u(\tau)\omega(\tau))\right\|_{L^2}\, d\tau\label{eq:b.67}\\
				&+|\beta |\int_0^{\frac {t-h}2}\left\|\left(e^{h\beta L_1}-1\right)e^{(t-h-\tau)\Delta}(-\Delta)^\frac{\delta+m}{2}(u(\tau)\omega(\tau))\right\|_{L^2}\, d\tau.\label{eq:b.68}
			\end{align}
		We can deduce that for \eqref{eq:b.67}, 
			\begin{align}
				&\quad|\beta |\int_0^{\frac {t-h}2}\left\|\left(e^{h\Delta}-1\right)e^{(t-h-\tau)\Delta}(-\Delta)^\frac{\delta+m}{2}(u(\tau)\omega(\tau))\right\|_{L^2}\, d\tau\label{eq:b.69}\\
				&\lesssim h|\beta |\int_0^{\frac {t-h}2}\left\|e^{(t-h-\tau)\Delta}(-\Delta)^\frac{2+\delta+m}{2}(u(\tau)\omega(\tau))\right\|_{L^2}\, d\tau\label{eq:b.70}\\
				&\lesssim h|\beta |\int_0^{\frac {t-h}2}(t-h-\tau)^{-\frac {2+m}{2}-\frac1q+\frac12}\left\|(-\Delta)^\frac{\delta}{2}(u(\tau)\omega(\tau))\right\|_{L^q}\, d\tau
				\to0\quad\text{as $h\to0.$}\label{eq:b.71}
			\end{align}
		For \eqref{eq:b.68}, take
			$\gamma_2\in(0, 1)$
		small enough to satisfy
			$\frac1q>\frac12+\frac{\gamma_2}{2}.$
		Then we have
			\begin{align}
				&\quad|\beta |\int_0^{\frac {t-h}2}\left\|\left(e^{h\beta L_1}-1\right)e^{(t-h-\tau)\Delta}(-\Delta)^\frac{\delta+m}{2}(u(\tau)\omega(\tau))\right\|_{L^2}\, d\tau\label{eq:b.72}\\
				&\lesssim |\beta|^{1+\gamma_2}h^{\gamma_2}\int_0^{\frac {t-h}2}\left\|e^{(t-h-\tau)\Delta}(-\Delta)^\frac{\delta+m-\gamma_2}{2}(u(\tau)\omega(\tau))\right\|_{L^2}\, d\tau\label{eq:b.73}\\
				&\lesssim |\beta|^{1+\gamma_2}h^{\gamma_2}\int_0^{\frac {t-h}2}(t-h-\tau)^{-\frac m2-\frac1q+\frac{1+\gamma_2}{2}}\left\|(-\Delta)^\frac{\delta}{2}(u(\tau)\omega(\tau))\right\|_{L^q}\, d\tau\label{eq:b.74}\\
				&\to 0\quad\text{as $h\to0.$}\label{eq:b.75}
			\end{align}
		and this is sufficient for the estimate of \eqref{eq:b.55}.
		For \eqref{eq:b.56}, we have
			\begin{align}
				&\quad\left\|\int_{\frac {t-h}2}^\frac{t}{2}\beta L_1T_\beta(t-\tau)\operatorname{div}(u(\tau)\omega(\tau))\, d\tau\right\|_{\dot H^{\delta+m}}\label{eq:b.76}\\
				&\lesssim|\beta|\int_{\frac {t-h}2}^\frac{t}{2}\left\|e^{(t-\tau)\Delta}(-\Delta)^{\frac{\delta+m}{2}}(u(\tau)\omega(\tau))\right\|_{L^2}\, d\tau\label{eq:b.77}\\
				&\lesssim|\beta|\int_{\frac {t-h}2}^\frac{t}{2}(t-\tau)^{-\frac{m}2-\frac1q+\frac12}\left\|(-\Delta)^{\frac{\delta}{2}}(u(\tau)\omega(\tau))\right\|_{L^q}\, d\tau\label{eq:b.78}\\
				&\lesssim|\beta|t^{-\frac{m}2}h^{1-\frac1{r_1}-\frac1{r_2}-\frac1q+\frac12}\left\|\omega\right\|_{Y_1}\left\|\omega\right\|_{Y_2}\label{eq:b.79}\\
				&\to 0\quad\text{as $h\to 0.$}\label{eq:b.80}
			\end{align}
		and this gives the estimate for \eqref{eq:b.17}.
		
		The estimates for \eqref{eq:b.18} and \eqref{eq:b.19} can similarly be done. For \eqref{eq:b.18}, we have
			\begin{align}
				&\quad\left\|\frac1h\int_\frac {t-h}2^{t-h}\left(e^{h\Delta}-1\right)e^{h\beta L_1}T_\beta(t-h-\tau)\operatorname{div}(u(\tau)\omega(\tau))\, d\tau
				-\int_\frac t2^{t}\Delta T_\beta(t-\tau)\operatorname{div}(u(\tau)\omega(\tau))\, d\tau\right\|_{\dot H^{\delta+m}}
				\notag\\
				&\leq\left\|\frac1h\int_\frac {t-h}2^{t-h}\left(e^{h\Delta}-1\right)e^{h\beta L_1}T_\beta(t-h-\tau)\operatorname{div}(u(\tau)\omega(\tau))\, d\tau\right.\label{eq:b.82}\\
				&\quad\phantom{\leq\|\frac1h\int_\frac {t-h}2^{t-h}}\left.-\frac1h\int_\frac {t-h}2^{t-h}\left(e^{h\Delta}-1\right)e^{h\beta L_1} T_\beta(t-\tau)\operatorname{div}(u(\tau)\omega(\tau))\, d\tau\right\|_{\dot H^{\delta+m}}
				\label{eq:b.83}\\
				&+\left\|\frac1h\int_\frac {t-h}2^{t-h}\left(e^{h\Delta}-1\right)e^{h\beta L_1} T_\beta(t-\tau)\operatorname{div}(u(\tau)\omega(\tau))\, d\tau\right.\label{eq:b.84}\\
				&\left.\quad\phantom{\leq\|\frac1h\int_\frac {t-h}2^{t-h}}-\frac1h\int_\frac {t-h}2^{t-h}\left(e^{h\Delta}-1\right) T_\beta(t-\tau)\operatorname{div}(u(\tau)\omega(\tau))\, d\tau\right\|_{\dot H^{\delta+m}}
				\label{eq:b.85}\\
				&+\left\|\frac1h\int_\frac {t-h}2^{t-h}\left(e^{h\Delta}-1\right) T_\beta(t-\tau)\operatorname{div}(u(\tau)\omega(\tau))\, d\tau
				-\int_\frac {t-h}2^{t-h}\Delta T_\beta(t-\tau)\operatorname{div}(u(\tau)\omega(\tau))\, d\tau\right\|_{\dot H^{\delta+m}}
				\label{eq:b.86}\\
				&+\left\|\int_\frac {t-h}2^{t-h}\Delta T_\beta(t-\tau)\operatorname{div}(u(\tau)\omega(\tau))\, d\tau
				-\int_\frac t2^{t}\Delta T_\beta(t-\tau)\operatorname{div}(u(\tau)\omega(\tau))\, d\tau\right\|_{\dot H^{\delta+m}}.\label{eq:b.87}
			\end{align}
		We show that each line converges to 0.
		
		For \eqref{eq:b.83}, we see that
			\begin{align}
				&\quad\left\|\frac1h\int_\frac {t-h}2^{t-h}\left(e^{h\Delta}-1\right)e^{h\beta L_1}\left(T_\beta(h)-1\right)T_\beta(t-h-\tau)\operatorname{div}(u(\tau)\omega(\tau))\, d\tau\right\|_{\dot H^{\delta+m}}\label{eq:b.88}\\
				&\lesssim \int_\frac {t-h}2^{t-h}\left\|\left(T_\beta(h)-1\right)e^{(t-h-\tau)\Delta}(-\Delta)^\frac{3+\delta+m}{2}(u(\tau)\omega(\tau))\right\|_{L^2}\, d\tau\label{eq:b.89}\\
				&\leq\int_\frac {t-h}2^{t-h}\left\|\left(e^{h\Delta}e^{h\beta L_1}-e^{h\beta L_1}\right)e^{(t-h-\tau)\Delta}(-\Delta)^\frac{3+\delta+m}{2}(u(\tau)\omega(\tau))\right\|_{L^2}\, d\tau\label{eq:b.90}\\
				&+\int_\frac {t-h}2^{t-h}\left\|\left(e^{h\beta L_1}-1\right)e^{(t-h-\tau)\Delta}(-\Delta)^\frac{3+\delta+m}{2}(u(\tau)\omega(\tau))\right\|_{L^2}\, d\tau\label{eq:b.91}\\
				&\lesssim h\int_\frac {t}4^{t}\left\|(-\Delta)^\frac{5+\delta+m}{2}(u(\tau)\omega(\tau))\right\|_{L^2}\, d\tau
				+h|\beta|\int_\frac {t}4^{t}\left\|(-\Delta)^\frac{2+\delta+m}{2}(u(\tau)\omega(\tau))\right\|_{L^2}\, d\tau\label{eq:b.92}\\
				&\to 0\quad\text{as $h\to0.$}\label{eq:b.93}
			\end{align}
		For \eqref{eq:b.85}, a similar argument shows that this term also tends to 0.
		
		For \eqref{eq:b.86}, we have
			\begin{align}
				&\quad\left\|\int_\frac {t-h}2^{t-h}\left(\frac{e^{h\Delta}-1}{h}-\Delta\right) T_\beta(t-\tau)\operatorname{div}(u(\tau)\omega(\tau))\, d\tau
				\right\|_{\dot H^{\delta+m}}\label{eq:b.94}\\
				&\lesssim \int_\frac {t}4^{t}\left\|\left(\frac{e^{h\Delta}-1}{h}-\Delta\right) e^{(t-\tau)\Delta}(-\Delta)^\frac{1+\delta+m}{2}(u(\tau)\omega(\tau))\right\|_{L^2}\, d\tau.\label{eq:b.95}
			\end{align}
		As in the estimate of \eqref{eq:b.23}, applying the dominated convergence theorem is enough to show that this term converges to 0, since 
			$\left\|(-\Delta)^\frac{3+\delta+m}{2}(u(\tau)\omega(\tau))\right\|_{L^2}$
		is integrable on
			$(\frac t4, t).$
		For \eqref{eq:b.87}, we have
			\begin{align}
				&\quad\left\|\int_\frac {t-h}2^{t-h}\Delta T_\beta(t-\tau)\operatorname{div}(u(\tau)\omega(\tau))\, d\tau
				-\int_\frac t2^{t}\Delta T_\beta(t-\tau)\operatorname{div}(u(\tau)\omega(\tau))\, d\tau\right\|_{\dot H^{\delta+m}}\label{eq:b.96}\\
				&\lesssim h\sup_{\tau\in (\frac t4, t)}\left\|(-\Delta)^\frac{3+\delta+m}{2}\left(u(\tau)\omega(\tau)\right)\right\|_{L^2}
				\to 0\quad\text{as $h\to0$}, \label{eq:b.97}
			\end{align}
		and this completes the estimate for \eqref{eq:b.18}.
			
		Finally for \eqref{eq:b.19}, we have
			{
			\begin{align}\displaybreak[0]
				&\quad\left\|\frac1h\int_\frac {t-h}2^{t-h}\left(e^{h\beta L_1}-1\right)T_\beta(t-h-\tau)\operatorname{div}(u(\tau)\omega(\tau))\, d\tau
				-\int_\frac t2^{t}\beta L_1T_\beta(t-\tau)\operatorname{div}(u(\tau)\omega(\tau))\, d\tau\right\|_{\dot H^{\delta+m}}
				\notag\\
				&\leq\left\|\frac1h\int_\frac {t-h}2^{t-h}\left(e^{h\beta L_1}-1\right)T_\beta(t-h-\tau)\operatorname{div}(u(\tau)\omega(\tau))\, d\tau\right.\label{eq:b.99}\\
				&\left.\quad\phantom{\leq\frac1h\int_\frac {t-h}2^{t-h}}-\frac1h\int_\frac {t-h}2^{t-h}\left(e^{h\beta L_1}-1\right)T_\beta(t-\tau)\operatorname{div}(u(\tau)\omega(\tau))\, d\tau\right\|_{\dot H^{\delta+m}}\label{eq:b.100}\\
				&+\left\|\frac1h\int_\frac {t-h}2^{t-h}\left(e^{h\beta L_1}-1\right)T_\beta(t-\tau)\operatorname{div}(u(\tau)\omega(\tau))\, d\tau
				-\int_\frac {t-h}2^{t-h}\beta L_1T_\beta(t-\tau)\operatorname{div}(u(\tau)\omega(\tau))\, d\tau\right\|_{\dot H^{\delta+m}}\label{eq:b.101}\\
				&+\left\|\int_\frac {t-h}2^{t-h}\beta L_1T_\beta(t-\tau)\operatorname{div}(u(\tau)\omega(\tau))\, d\tau
				-\int_\frac {t}2^{t}\beta L_1T_\beta(t-\tau)\operatorname{div}(u(\tau)\omega(\tau))\, d\tau\right\|_{\dot H^{\delta+m}}\label{eq:b.102}.
			\end{align}}
		For \eqref{eq:b.100}, we have
			\begin{align}
				&\quad\left\|\frac1h\int_\frac {t-h}2^{t-h}\left(e^{h\beta L_1}-1\right)\left(T_\beta(h)-1\right)T_\beta(t-h-\tau)\operatorname{div}(u(\tau)\omega(\tau))\, d\tau\right\|_{\dot H^{\delta+m}}\label{eq:b.103}\\
				&\lesssim \frac1h\int_\frac {t-h}2^{t-h}\left\|\left(e^{h\beta L_1}-1\right)\left(T_\beta(h)-1\right)e^{(t-h-\tau)\Delta}(-\Delta)^{\frac{1+\delta+m}{2}}(u(\tau)\omega(\tau))\right\|_{L^2}\, d\tau\label{eq:b.104}\\
				&\lesssim |\beta|\int_\frac {t-h}2^{t-h}\left\|\left(e^{h\Delta}e^{h\beta L_1}-e^{h\beta L_1}\right)e^{(t-h-\tau)\Delta}(-\Delta)^{\frac{\delta+m}{2}}(u(\tau)\omega(\tau))\right\|_{L^2}\, d\tau\label{eq:b.105}\\
				&+|\beta|\int_\frac {t-h}2^{t-h}\left\|\left(e^{h\beta L_1}-1\right)e^{(t-h-\tau)\Delta}(-\Delta)^{\frac{\delta+m}{2}}(u(\tau)\omega(\tau))\right\|_{L^2}\, d\tau\label{eq:b.106}
			\end{align}
		For \eqref{eq:b.105}, we have
			\begin{align}
				&\quad|\beta|\int_\frac {t-h}2^{t-h}\left\|\left(e^{h\Delta}e^{h\beta L_1}-e^{h\beta L_1}\right)e^{(t-h-\tau)\Delta}(-\Delta)^{\frac{\delta+m}{2}}(u(\tau)\omega(\tau))\right\|_{L^2}\, d\tau\label{eq:b.107}\\
				&\lesssim h|\beta|\int_\frac {t}4^{t}\left\|(-\Delta)^{\frac{2+\delta+m}{2}}(u(\tau)\omega(\tau))\right\|_{L^2}\, d\tau
				\to0\quad\text{as $h\to0.$}\label{eq:b.108}
			\end{align}
		For \eqref{eq:b.106}, we have
			\begin{align}
				&\quad|\beta|\int_\frac {t-h}2^{t-h}\left\|\left(e^{h\beta L_1}-1\right)e^{(t-h-\tau)\Delta}(-\Delta)^{\frac{\delta+m}{2}}(u(\tau)\omega(\tau))\right\|_{L^2}\, d\tau\label{eq:b.109}\\
				&\lesssim|\beta|^{1+\gamma_2}h^{\gamma_2}\int_\frac {t-h}2^{t-h}\left\|e^{(t-h-\tau)\Delta}(-\Delta)^{\frac{\delta+m-\gamma_2}{2}}(u(\tau)\omega(\tau))\right\|_{L^2}\, d\tau\label{eq:b.110}\\
				&\lesssim |\beta|^{1+\gamma_2}h^{\gamma_2}\int_\frac {t-h}2^{t-h}(t-h-\tau)^{-\frac1q+\frac12+\frac{\gamma_2}{2}}\left\|(-\Delta)^{\frac{\delta+m}{2}}(u(\tau)\omega(\tau))\right\|_{L^q}\, d\tau\label{eq:b.111}\\
				&\lesssim |\beta|^{1+\gamma_2}h^{\gamma_2}t^{1-\frac1{r_1}-\frac1{r_2}-\frac1q+\frac12+\frac{\gamma_2}{2}}\left(\left\|\omega\right\|_{L^{r_1}([\frac t4, \infty); \dot W^{-1+\delta+m, p_1})}\left\|\omega\right\|_{Y_2}
				+\left\|\omega\right\|_{L^{r_2}([\frac t4, \infty); \dot W^{\delta+m, p_2})}\left\|\omega\right\|_{Y_1}\right)\notag\\
				&\to 0\quad\text{as $h\to0$.}\label{eq:b.113}
			\end{align}
			
		For \eqref{eq:b.101}, we have
			\begin{align}
			&\quad\left\|\int_\frac {t-h}2^{t-h}\left(\frac{e^{h\beta L_1}-1}{h}-\beta L_1\right)T_\beta(t-\tau)\operatorname{div}(u(\tau)\omega(\tau))\, d\tau
				\right\|_{\dot H^{\delta+m}}\label{eq:b.114}\\
				&\lesssim \int_\frac {t}4^{t}\left\|\left(\frac{e^{h\beta L_1}-1}{h}-\beta L_1\right)(-\Delta)^\frac{1+\delta+m}{2}(u(\tau)\omega(\tau))\right\|_{L^2}\, d\tau\label{eq:b.115}
			\end{align}
		and the dominated convergence theorem ensures that this tends to 0.
		
		For \eqref{eq:b.102}, we have
			\begin{align}
				&\quad\left\|\int_\frac {t-h}2^{t-h}\beta L_1T_\beta(t-\tau)\operatorname{div}(u(\tau)\omega(\tau))\, d\tau
				-\int_\frac {t}2^{t}\beta L_1T_\beta(t-\tau)\operatorname{div}(u(\tau)\omega(\tau))\, d\tau\right\|_{\dot H^{\delta+m}}\notag\\
				&\lesssim h|\beta|\sup_{\frac t4<\tau<t}\left\|(-\Delta)^\frac{\delta+m}{2}(u(\tau)\omega(\tau))\right\|_{L^2}
				\to 0\quad\text{as $h\to0$.}\label{eq.b.117}
			\end{align}
		This completes the estimate of the desired convergence.

\section{Proof of Proposition~\ref{prop:5.10} for general $s$}\label{sect:c}

In this section, we prove Proposition~\ref{prop:5.10} for general
	$s>0$.
Here, we assume that the exponents
	$\mathcal A=(\delta, p, r_1, p, r_2)$
satisfies the assumption stated in the beginning of Section~\ref{sect:5}.
We begin by proving an analogue of Lemma~\ref{lem:5.4} for  higher-regularity norms.
\begin{lemma}\label{lem:c.1}
	Let
		$0\leq\delta\leq \frac15$
	be given and assume that
		$\omega_0\in \dot H^{-1+\delta}\cap \dot H^{3-\frac2p}\cap \dot W^{-1+\delta, p'}\cap L^1$
	satisfies the conditions in Proposition~\ref{prop:5.2}, Proposition~\ref{prop:5.3}, and Lemma~\ref{lem:5.4}. Then, for any 
		$s\geq0$
	there is a constant
		$C>0$
	dependent only on 
		$s, $
		$\mathcal A, $
		$\left\|\omega_0\right\|_{L^1}, $
		$|\beta|^{-\frac\delta3}\left\|\omega_0\right\|_{\dot H^{-1+\delta}},$ 
		$|\beta|^{-\frac{1+\delta}3}\left\|\omega_0\right\|_{\dot H^{\delta}},$ 
	and 
		$|\beta|^{-\frac13(1+\delta-\frac2{p'})}\left\|\omega_0\right\|_{\dot W^{-1+\delta, p'}}$
	satisfying
		\begin{align}
			|\beta|^{-\frac13(1+\delta-\frac2{p'})}\left\|\omega(t)\right\|_{\dot W^{-1+\delta+s, p}}
			\leq Ct^{-\frac s2-1+\frac2p}\min\left\{1, |\beta|^{-1+\frac2p}t^{-\frac32(1-\frac2p)}\right\}\label{eq:c.1}
		\end{align}
	for all
		$t>0$.
\end{lemma}
\proof
	It suffices to prove the estimate for 
		$s\in\mathbb N.$
	Define the functions
		$N_s$
	and
		$\widetilde\Omega_s$
	on
		$(0, \infty)$
	by
		\begin{align}
			N_s(t)\coloneqq t^{-\frac s2-1+\frac2p}\min\left\{1, |\beta|^{-1+\frac2p}t^{-\frac32(1-\frac2p)}\right\}\label{eq:c.2}
		\end{align}
	and
		\begin{align}
			\widetilde \Omega_s(t)\coloneqq \sup_{0<\tau\leq t}N_s(\tau)^{-1}\left\|\omega(\tau)\right\|_{\dot W^{-1+\delta+s, p}}.\label{eq:c.3}
		\end{align}
	Similarly to Lemma~\ref{lem:5.4}, we can prove 
		\begin{align}
			&\left\|T_\beta(t)\omega_0\right\|_{\dot W^{-1+\delta+s,p}}
			+\int_0^\frac t2\left\|T_\beta(t-\tau)\operatorname{div}(u(\tau)\omega(\tau))\right\|_{\dot W^{-1+\delta+s,p}}\, d\tau\label{eq:c.4}\\
			&\lesssim N_s(t)\left(\left\|\omega_0\right\|_{\dot W^{-1+\delta, p'}}+\Omega_\delta(t)|\beta|^{-\frac\delta3}\left\|\omega_0\right\|_{\dot H^{-1+\delta}}|\beta|^{\frac13(1+\delta-\frac2{p'})}\right).\label{eq:c.5}
		\end{align}
	Therefore, we are left to prove the estimate for the integral over
		$(\frac t2, t)$.
	We have
		\begin{align}
			&\quad\int^t_\frac t2\left\|T_\beta(t-\tau)\operatorname{div}(u(\tau)\omega(\tau))\right\|_{\dot W^{-1+\delta+s, p}}\, d\tau\label{eq:c.6}\\
			&\lesssim \int_\frac t2^tN_0(t-\tau)\left\|(-\Delta)^\frac {s+\delta}2(u(\tau)\omega(\tau))\right\|_{L^{p'}}\, d\tau\label{eq:c.7}\\
			&\lesssim \int_\frac t2^tN_0(t-\tau)\left\|\omega(\tau)\right\|_{\dot W^{-1+\delta+s, p}}\left\|\omega(\tau)\right\|_{\dot W^{\delta, p}}\, d\tau\label{eq:c.8}\\
			&+ \int_\frac t2^tN_0(t-\tau)\left\|\omega(\tau)\right\|_{\dot W^{-1+\delta, p}}\left\|\omega(\tau)\right\|_{\dot W^{\delta+s, p}}\, d\tau\label{eq:c.9}.
		\end{align}
	
	For \eqref{eq:c.8}, the interpolation inequality gives
		\begin{align*}
			\left\|\omega(\tau)\right\|_{\dot W^{-1+\delta+s, p}}
			\lesssim \left\|\omega(\tau)\right\|_{\dot W^{-1+\delta, p}}^\frac{1}{1+s}
			\left\|\omega(\tau)\right\|_{\dot W^{\delta+s, p}}^\frac{s}{1+s}
		\end{align*}
	and
		\begin{align*}
			\left\|\omega(\tau)\right\|_{\dot W^{\delta}}
			\lesssim \left\|\omega(\tau)\right\|_{\dot W^{-1+\delta, p}}^\frac{s}{1+s}
			\left\|\omega(\tau)\right\|_{\dot W^{\delta+s, p}}^\frac{1}{1+s}.
		\end{align*}
	Combining these inequalities, we have
		\begin{align*}
			\left\|\omega(\tau)\right\|_{\dot W^{-1+\delta+s, p}}
			\left\|\omega(\tau)\right\|_{\dot W^{\delta}}
			\lesssim \left\|\omega(\tau)\right\|_{\dot W^{-1+\delta, p}}
			\left\|\omega(\tau)\right\|_{\dot W^{\delta+s, p}}
		\end{align*}
	and therefore, the estimate of \eqref{eq:c.8} is reduced to that of \eqref{eq:c.9}.
		
	For \eqref{eq:c.9}, we have
		\begin{align}
			 &|\beta|^{-\frac13(1+\delta-\frac2{p'})}\int_\frac t2^tN_0(t-\tau)\left\|\omega(\tau)\right\|_{\dot W^{-1+\delta, p}}\left\|\omega(\tau)\right\|_{\dot W^{\delta+s, p}}\, d\tau\label{eq:c.18}\\
			 &\lesssim |\beta|^{-\frac13(1+\delta-\frac2{p'})}M_{\delta+s}(t)\Omega_{\delta+s}(t)\left\|N_0\right\|_{L^{r_1'}([0, \frac t2))}\left\|\omega\right\|_{Y_1}\label{eq:c.19}\\
			 &\lesssim |\beta|^{-\frac13(1+\delta-\frac2{p'})}t^{-\frac \delta2-\frac s2-1+\frac1p+1-\frac1{r_1}-1+\frac2p}\Omega_{\delta+s}(t)\left\|\omega\right\|_{Y_1}\label{eq:c.20}\\
			 &\lesssim \left(|\beta|^\frac13t^\frac12\right)^{2(1-\frac2p-\frac1{r_1})}t^{-\frac s2-1+\frac2p}\Omega_{\delta+s}(t)|\beta|^{-\frac\delta3}\left\|\omega_0\right\|_{\dot H^{-1+\delta}}\label{eq:c.21}
		\end{align}
	and
		\begin{align}
			&|\beta|^{-\frac13(1+\delta-\frac2{p'})}\int_\frac t2^tN_0(t-\tau)\left\|\omega(\tau)\right\|_{\dot W^{-1+\delta, p}}\left\|\omega(\tau)\right\|_{\dot W^{\delta+s, p}}\, d\tau\label{eq:c.22}\\
			&\lesssim |\beta|^{-\frac13(1+\delta-\frac2{p'})}t^{\frac2p}N_0(t)\widetilde \Omega_0(t)M_{\delta+s}(t)\Omega_{\delta+s}(t)\label{eq:c.23}\\
			&\lesssim |\beta|^{-2+\frac4p}t^{\frac6p-3}t^{-\frac s2-1+\frac2p}|\beta|^{-\frac13(1+\delta-\frac2{p'})}\widetilde \Omega_0(t)\Omega_{\delta+s}(t)\label{eq:c.24}\\
			&\lesssim \left(|\beta|^\frac13t^\frac12\right)^{-6+\frac{12}p}t^{-\frac s2-1+\frac2p}|\beta|^{-\frac13(1+\delta-\frac2{p'})}\widetilde \Omega_0(t)\Omega_{\delta+s}(t).\label{eq:c.25}
		\end{align}
	We can check that
		$1-\frac2p-\frac1{r_1}>0$
	and
		$\frac12(-6+\frac{12}p)<-\frac32(1-\frac2p).$
	Thus we obtain
		\begin{align}
			&|\beta|^{-\frac13(1+\delta-\frac2{p'})}\int_\frac t2^tN_0(t-\tau)\left\|\omega(\tau)\right\|_{\dot W^{-1+\delta, p}}\left\|\omega(\tau)\right\|_{\dot W^{\delta+s, p}}\, d\tau\label{eq:c.26}\\
			&\lesssim t^{-\frac s2-1+\frac2p}\Omega_{\delta+s}(t)\left(|\beta|^{-\frac\delta3}\left\|\omega_0\right\|_{\dot H^{-1+\delta}}\right)^\frac{3-\frac{6}{p}}{4-\frac{8}p-\frac1{r_1}}\left(|\beta|^{-\frac13(1+\delta-\frac2{p'})}\widetilde \Omega_0(t)\right)^\frac{1-\frac2p-\frac1{r_1}}{4-\frac{8}p-\frac1{r_1}}\label{eq:c.27}
		\end{align}
	and
		\begin{align}
			&|\beta|^{-\frac13(1+\delta-\frac2{p'})}\int_\frac t2^tN_0(t-\tau)\left\|\omega(\tau)\right\|_{\dot W^{-1+\delta, p}}\left\|\omega(\tau)\right\|_{\dot W^{\delta+s, p}}\, d\tau\label{eq:c.28}\\
			&\lesssim |\beta|^{-1+\frac2p}t^{-\frac s2-\frac52(1-\frac2p)}\Omega_{\delta+s}(t)\left(|\beta|^{-\frac\delta3}\left\|\omega_0\right\|_{\dot H^{-1+\delta}}\right)^\frac{3-\frac6p}{8-\frac{16}p-\frac2{r_1}}\left(|\beta|^{-\frac13(1+\delta-\frac2{p'})}\widetilde \Omega_0(t)\right)^\frac{5-\frac{10}p-\frac2{r_1}}{8-\frac{16}p-\frac2{r_1}}.\notag
		\end{align}
	These are sufficient to yield the desired estimate.
\qed

Finally, we treat the proof of Proposition~\ref{prop:5.10} for the case of general
	$s>0$.
Here, we assume that
	$0\leq\delta<\frac1{13}, $
and we may also assume that
	$s\in\mathbb N$
without loss of generality.
We have
	\begin{align}
		&\int^t_\frac t2\left\|T_\beta(t-\tau)\operatorname{div}(u(\tau)\omega(\tau))\right\|_{\dot W^{s, \infty}}\, d\tau\label{eq:c.30}\\
		&\lesssim \int^t_\frac t2(t-\tau)^{-\frac58}\min\left\{1, |\beta|^{-\frac{1}{4}}(t-\tau)^{-\frac38}\right\}\left\|(-\Delta)^\frac s2u(\tau)\cdot\nabla\omega(\tau)\right\|_{L^{\frac85}}\, d\tau\label{eq:c.31}
		\\
		&\lesssim \int^t_\frac t2(t-\tau)^{-\frac58}\min\left\{1, |\beta|^{-\frac{1}{4}}(t-\tau)^{-\frac38}\right\}\left\|\omega(\tau)\right\|_{\dot W^{-1+\delta+s, p}}\left\|\omega(\tau)\right\|_{\dot W^{1, b}}\, d\tau\label{eq:c.32}\\
		&+\int^t_\frac t2(t-\tau)^{-\frac58}\min\left\{1, |\beta|^{-\frac{1}{4}}(t-\tau)^{-\frac38}\right\}\left\|\omega(\tau)\right\|_{\dot W^{-1+\delta, p}}\left\|\omega(\tau)\right\|_{\dot W^{1+s, b}}\, d\tau\label{eq:c.33}.
	\end{align}
	
For \eqref{eq:c.32}, we have
	\begin{align}
		& \int^t_\frac t2(t-\tau)^{-\frac58}\min\left\{1, |\beta|^{-\frac{1}{4}}(t-\tau)^{-\frac38}\right\}\left\|\omega(\tau)\right\|_{\dot W^{-1+\delta+s, p}}\left\|\omega(\tau)\right\|_{\dot W^{1, b}}\, d\tau\label{eq:c.34}\\
		&\lesssim M_{-1+\delta+s}(t)\Omega_{-1+\delta+s}(t)\left\|\omega\right\|_{L^{r_1}([0, \infty); \dot W^{1, b})}\left(\int_{\frac t2}^t(t-\tau)^{-\frac58r_1'}\, d\tau\right)^{\frac1{r_1'}}.\label{eq:c.35}
	\end{align}
Our setting ensures that
	$\frac1b<\frac1p$, 
	$\dot W^{s+\frac34-\frac2p, p}\hookrightarrow \dot W^{s, b}$
and
	$\frac58<1-\frac1{r_1}.$
Therefore, we obtain
	\begin{align}
		& \int^t_\frac t2(t-\tau)^{-\frac58}\min\left\{1, |\beta|^{-\frac{1}{4}}(t-\tau)^{-\frac38}\right\}\left\|\omega(\tau)\right\|_{\dot W^{-1+\delta+s, p}}\left\|\omega(\tau)\right\|_{\dot W^{1, b}}\, d\tau\label{eq:c.36}\\
		&\lesssim t^{-\frac{-1+\delta+s}2-1+\frac1p}\Omega_{-1+\delta+s}(t)\left\|\omega\right\|_{L^{r_1}([0, \infty); \dot W^{\frac74-\frac2p, p})}t^{1-\frac1{r_1}-\frac58}\label{eq:c.37}\\
		&\lesssim \left(|\beta|^\frac13t^\frac12\right)^{\frac{15}4-\frac4p-\frac2{r_1}}t^{-\frac s2-1}\Omega_{-1+\delta+s}(t)\label{eq:c.38}\\
		&\cdot\left(|\beta|^{-\frac{\delta}{3}}\left\|\omega_0\right\|_{\dot H^{-1+\delta}}
		+|\beta|^{-\frac{1}{3}(\frac{11}4-\frac2p)}\left\|\omega_0\right\|_{\dot H^{\frac74-\frac2p}}
		+|\beta|^{-\frac{\delta}{3}}\left\|\omega_0\right\|_{\dot H^{-1+\delta}}
		|\beta|^{-\frac{1}{3}(\frac{15}4-\frac2p)}\left\|\omega_0\right\|_{\dot H^{\frac{11}4-\frac2p}}\right).\label{eq:c.39}
	\end{align}
As in the proof of Proposition~\ref{prop:5.10}, we can show
	\begin{align}
		& \int^t_\frac t2(t-\tau)^{-\frac58}\min\left\{1, |\beta|^{-\frac{1}{4}}(t-\tau)^{-\frac38}\right\}\left\|\omega(\tau)\right\|_{\dot W^{-1+\delta+s, p}}\left\|\omega(\tau)\right\|_{\dot W^{1, b}}\, d\tau\label{eq:c.40}\\
		&\lesssim N_s(t)\widetilde\Omega_s(t)M_{1, b}(t)\Omega_{1, b}(t)\int^t_\frac t2 (t-\tau)^{-\frac58}\min\left\{1, |\beta|^{-\frac{1}{4}}(t-\tau)^{-\frac38}\right\}\, d\tau\label{eq:c.41}\\
		&\lesssim |\beta|^{-1}t^{-\frac s2-\frac52}\Omega_{1, b}(t)\left(|\beta|^{-\frac13(1+\delta-\frac2{p'})}\widetilde\Omega_s(t)\right).\label{eq:c.42}
	\end{align}
By interpolating \eqref{eq:c.39} and \eqref{eq:c.42}, we have
	\begin{align}
		& \int^t_\frac t2(t-\tau)^{-\frac58}\min\left\{1, |\beta|^{-\frac{1}{4}}(t-\tau)^{-\frac38}\right\}\left\|\omega(\tau)\right\|_{\dot W^{-1+\delta+s, p}}\left\|\omega(\tau)\right\|_{\dot W^{1, b}}\, d\tau\label{eq:c.43}\\
		&\lesssim t^{-\frac s2-1}\left(\Omega_{1, b}(t)\left(|\beta|^{-\frac13(1+\delta-\frac2{p'})}\widetilde\Omega_s(t)\right)\right)^\frac{\frac{15}4-\frac4p-\frac2{r_1}}{\frac{27}4-\frac4p-\frac2{r_1}}\label{eq:c.44}\\
		&\cdot\Big[\Omega_{-1+\delta+s}(t)\Big(|\beta|^{-\frac{\delta}{3}}\left\|\omega_0\right\|_{\dot H^{-1+\delta}}
		+|\beta|^{-\frac{1}{3}(\frac{11}4-\frac2p)}\left\|\omega_0\right\|_{\dot H^{\frac74-\frac2p}}
		\\&\phantom{\Omega_{-1+\delta+s}(t)\Big(|\beta|^{-\frac{\delta}{3}}}+|\beta|^{-\frac{\delta}{3}}\left\|\omega_0\right\|_{\dot H^{-1+\delta}}
		|\beta|^{-\frac{1}{3}(\frac{15}4-\frac2p)}\left\|\omega_0\right\|_{\dot H^{\frac{11}4-\frac2p}}\Big)\Big]^\frac{3}{\frac{27}4-\frac4p-\frac2{r_1}}.\label{eq:c.45}
	\end{align}
This yields the desired estimate of \eqref{eq:c.32}.

For \eqref{eq:c.33}, we similarly obtain
	\begin{align}
		& \int^t_\frac t2(t-\tau)^{-\frac58}\min\left\{1, |\beta|^{-\frac{1}{4}}(t-\tau)^{-\frac38}\right\}\left\|\omega(\tau)\right\|_{\dot W^{-1+\delta, p}}\left\|\omega(\tau)\right\|_{\dot W^{1+s, b}}\, d\tau\label{eq:c.46}\\
		&\lesssim M_{1+s, b}(t)\Omega_{1+s, b}(t)\left\|\omega\right\|_{Y_1}t^{\frac38-\frac1{r_1}}\label{eq:c.47}\\
		&\lesssim \left(|\beta|^\frac13t^\frac12\right)^{1+\delta-\frac2p-\frac2{r_1}}t^{-\frac s2-1}\Omega_{1+s, b}(t)|\beta|^{-\frac\delta3}\left\|\omega_0\right\|_{\dot H^{-1+\delta}}\label{eq:c.48}
	\end{align}
and
	\begin{align}
		& \int^t_\frac t2(t-\tau)^{-\frac58}\min\left\{1, |\beta|^{-\frac{1}{4}}(t-\tau)^{-\frac38}\right\}\left\|\omega(\tau)\right\|_{\dot W^{-1+\delta, p}}\left\|\omega(\tau)\right\|_{\dot W^{1+s, b}}\, d\tau\label{eq:c.49}\\
		&\lesssim N_0(t)\widetilde\Omega_0(t)M_{1+s, b}(t)\Omega_{1+s, b}(t)\int^t_\frac t2 (t-\tau)^{-\frac58}\min\left\{1, |\beta|^{-\frac{1}{4}}(t-\tau)^{-\frac38}\right\}\, d\tau\label{eq:c.50}\\
		&\lesssim |\beta|^{-1}t^{-\frac s2-\frac52}\Omega_{1+s, b}(t)\left(|\beta|^{-\frac13(1+\delta-\frac2{p'})}\widetilde\Omega_0(t)\right).\label{eq:c.51}
	\end{align}
Now the interpolation of these estimates gives
	\begin{align}
		& \int^t_\frac t2(t-\tau)^{-\frac58}\min\left\{1, |\beta|^{-\frac{1}{4}}(t-\tau)^{-\frac38}\right\}\left\|\omega(\tau)\right\|_{\dot W^{-1+\delta, p}}\left\|\omega(\tau)\right\|_{\dot W^{1+s, b}}\, d\tau\label{eq:c.52}\\
		&\lesssim t^{-\frac s2-1}\Omega_{1+s, b}(t)\left(|\beta|^{-\frac\delta3}\left\|\omega_0\right\|_{\dot H^{-1+\delta}}\right)^\frac{3}{4+\delta-\frac2p-\frac2{r_1}}
		\left(|\beta|^{-\frac13(1+\delta-\frac2{p'})}\widetilde\Omega_0(t)\right)^\frac{1+\delta-\frac2p-\frac2{r_1}}{4+\delta-\frac2p-\frac2{r_1}}.\label{eq:c.53}
	\end{align}
This is sufficient to conclude the proof.

\subsection*{Acknowledgement}
The author would like to express his sincere gratitude to Professor Ryo Takada for his valuable suggestions and continuous encouragement. The author is also grateful to Professor Yasunori Maekawa for his insightful comments.
This research was supported by FoPM, WINGS Program, the University of Tokyo.

%
%

\bibliographystyle{alpha} 
\bibliography{refs}          

\end{document}